\documentclass[11pt,reqno]{amsart}

\numberwithin{equation}{section}
\usepackage{times}
\usepackage{amsmath,amsfonts,amstext,amssymb,amsbsy,amsopn,amsthm,eucal}
\usepackage{txfonts}
\usepackage{dsfont}
\usepackage{graphicx}   % for figures
\usepackage{hyperref}
\usepackage{accents}
\usepackage{enumerate}
\usepackage{xcolor}
\usepackage{verbatim}
\usepackage{esint}

\usepackage[normalem]{ulem}
\usepackage{cancel}

\newcommand{\NN}{\mathds{N}}
\newcommand{\RR}{\mathbb{R}}

\newcommand{\Vol}{{\rm Vol}}

\newcommand{\cC}{\mathcal{C}}

\newcommand{\cH}{\mathcal{H}}

\newcommand{\cL}{\mathcal{L}}

\newcommand{\cN}{\mathcal{N}}

\newcommand{\cS}{\mathcal{S}}

\newcommand{\cV}{\mathcal{V}}

\newcommand{\e}{\text{e}}
\newcommand{\Tu}{\Tilde{u}}
\newcommand{\Ta}{\Tilde{a}}
\newcommand{\Tb}{\Tilde{b}}
\newcommand{\Tc}{\Tilde{c}}

\newcommand{\Tr}{\hat{r}}

\newtheorem{theorem}[equation]{Theorem}

\newtheorem{proposition}[equation]{Proposition}
\newtheorem{lemma}[equation]{Lemma}
\newtheorem{corollary}[equation]{Corollary}
\theoremstyle{definition}
\newtheorem{definition}[equation]{Definition}

\theoremstyle{remark}
\newtheorem{remark}[equation]{Remark}
\theoremstyle{remark}
\newtheorem{example}[equation]{Example}
\theoremstyle{remark}

\theoremstyle{remark}
\theoremstyle{remark}

\begin{document}

%\author{Yiqi Huang and Wenshuai Jiang}
\thanks{}
\thanks{}

\title[The Nodal sets of parabolic equations]{The Nodal sets of solutions to parabolic equations}

\dedicatory{Dedicated to Gang Tian's 65th Birthday}

\author{Yiqi Huang}
\address[Yiqi Huang]{Department of Mathematics, MIT, 77 Massachusetts Avenue, Cambridge, MA 02139-4307, USA}
 \email{yiqih777@mit.edu}

\author{Wenshuai Jiang}
\address[Wenshuai Jiang]{School of Mathematical Sciences, Zhejiang University, Hangzhou 310058, China}
 \email{wsjiang@zju.edu}

%%\date{\today}

\begin{abstract}
In this paper, we study the parabolic equations $\partial_t u=\partial_j\left(a^{ij}(x,t)\partial_iu\right)+b^j(x,t)\partial_ju+c(x,t)u$ in a domain of $\mathbb{R}^n$ under the condition that $a^{ij}$ are Lipschitz continuous. Consider the nodal set $Z_t=\{x: u(x,t)=0\}$ at a time $t$-slice. Simple examples show that the singular set $\cS_t=\{x: u(x,t)=|\nabla_x u|(x,t)=0\}$ may coincide with nodal set. This makes the methods used in the study of nodal sets for elliptic equations fail, rendering the parabolic case much more complicated. 

The current strongest results in the literature establish the finiteness of the $(n-1)$-dimensional Hausdorff measure of $Z_t$, assuming either $n=1$ by Angenent \cite{An} or that the coefficients are time-independent and analytic by Lin \cite{Lin91}. With general coefficients, the codimension-one estimate was obtained under some doubling assumption by Han-Lin \cite{HLparabolic} but only for space-time nodal sets. In the first part, we prove that $\cH^{n-1}(Z_t) < \infty$ in full generality, i.e. for any dimension, with time-dependent coefficients and with merely Lipschitz regular leading coefficients $a^{ij}$. 

In the second part, we study the evolutionary behavior of nodal sets. When $n=1$, it is proved by Angenent \cite{An} that the number of nodal points is non-increasing in time. For the $n$-dimensional case, we construct examples showing that measure monotonicity fails. In contrast, we prove dimension monotonicity, i.e., the Hausdorff dimension of the nodal set is non-increasing in time. This is the first monotonicity property for nodal sets in general dimensions. All the assumptions here are sharp.

\end{abstract}

\maketitle

\tableofcontents

\section{Introduction}

The study of nodal sets of functions has long been a prominent research topic, particularly concerning the zero sets of solutions to partial differential equations. In this paper, we systematically investigate the properties of nodal sets of general parabolic equations. Let us consider a solution $u(x,t)$ to the parabolic equations with time-dependent coefficients
\begin{equation}\label{e:ParabolicE}
    \partial_t u = \partial_i(a^{ij}(x,t) \partial_j u) + b^i(x,t) \partial_i u +c(x, t) u,
\end{equation}
in $Q_2:=\{(x,t): |x|<2, -4<t\le 0\} \subset \RR^{n} \times \RR$. 
where the coefficients $a^{ij}$ are elliptic and the coefficients $b, c$ are bounded,
\begin{equation}\label{e:assumption_aijbc}
    (1 + \lambda)^{-1} \delta^{ij} \leq a^{ij} \leq (1 + \lambda) \delta^{ij}, \quad |b^i|, |c| \leq \lambda.
\end{equation}

We assume $a^{ij}$ satisfies Lipschitz regularity condition with respect to the parabolic distance
\begin{align}\label{e:Lipschitzcondition}
    |a^{ij}(x,t)-a^{ij}(y,s)|\le \lambda \cdot d( (x,t), (y,s)) \equiv \lambda \cdot (||x-y||^2+|s-t|)^{1/2}.
\end{align}

\subsection{Finite Measure Theorem}

First we study the measure of nodal sets of $u(\cdot, t)$ for each fixed $t$.

Particularly, when $u(x,t)$ and all coefficients are time-independent, $u$ satisfies an elliptic equation, whose nodal set has been extensively studied. If the coefficients are analytic, Donnelly and Fefferman \cite{DF1}  obtained optimal upper and lower bound estimates for the Haudorff measure of nodal sets of eigenfunctions of Laplacian, and Lin \cite{Lin91} provided another proof for the optimal upper bound. The non-analytic conditions are much more complicated. To name a few, there are several explicit but not optimal upper bound estimates obtained by Hardt-Simon \cite{HS} and also by Han-Lin \cite{HL94a} with another proof, and also lower bounds by \cite{CM,SZ}. Recently, Logunov \cite{Loglower} proved the optimal lower bound for the measure, thereby resolving Yau's conjecture, and also obtained polynomial growth upper bounds \cite{Logupper}, see also \cite{Dong,DF3,HHL,HHN,HJ,KZZ,LinShen19,LTY,LMNN,NV} for related results.

As for the solutions to parabolic equations in domain of $\mathbb{R}^n$, the lack of unique continuation in the time direction makes controlling the local behavior of solutions much more involved. In \cite{Lin91}, Lin proved the measure of nodal sets at each time slice is finite for the equations with analytic and time-independent coefficients. Nonetheless, with general non-analytic and time-dependent coefficients, very few results have been discovered. In the case $n=1$, Angenent \cite{An} proved that the nodal point set at each time slice is discrete. For higher dimensions, Han-Lin \cite{HLparabolic} proved that the nodal set has finite space-time codimension-one Hausdorff measure under doubling assumption on the solution. However, at each time slice, controlling the measure of the nodal sets remains an open problem.

As our first result, we address this case in full generality: for any dimension, with time-dependent coefficients and with merely Lipschitz regular leading coefficients $a^{ij}$, we can prove that the $(n-1)$-dimensional Hausdorff measure is finite at each time slice. 

\begin{theorem}\label{t:mainLip}
    Let $u$ be a solution of \eqref{e:ParabolicE} in $Q_2$ satisfying \eqref{e:assumption_aijbc} and \eqref{e:Lipschitzcondition}. For any $t_0>-4$, if $u(\cdot, t_0)$ is not identically zero, then we have the following Hausdorff measure estimate
    \begin{align}\label{e:Finite_Measure_Thm}
        \cH^{n-1}(Z_{t_0} \cap B_1)<C(n,\lambda,\Lambda)<\infty,
    \end{align}
    where $Z_{t_0}=\{x\in B_2: u(x,t_0)=0\}$ and $\Lambda \equiv \int_{Q_2} u^2 dxdt/ \int_{B_{3/2} \times \{t_0\}} u^2 dx$. 
\end{theorem}

\begin{remark}
Note that here the Lipschitz regularity \eqref{e:Lipschitzcondition} is a sharp assumption. If one only assumes H\"older assumption on $a^{ij}$, i.e. $|a^{ij}(x,t)-a^{ij}(y,s)|\le \lambda \cdot  |(x,t), (y,s)|^{\alpha}$ for $\alpha\in (0,1)$, then 
there exists some nontrivial solution that vanishes on some open subset, see \cite{Man, Miller, P}.
\end{remark}

\begin{remark}\label{r:Thm_holds_Inequality}
    More generally the same result \eqref{e:Finite_Measure_Thm} holds for $u$ satisfying $|\partial_t u-  \partial_i(a^{ij}(x,t) \partial_j u)| \le \lambda(|\nabla u| + |u|)$ with \eqref{e:assumption_aijbc} and \eqref{e:Lipschitzcondition}. The proof is verbatim with minor modifications. 
\end{remark}

\begin{remark}
    It is proved that the unique continuation property holds for solution $u$ to \eqref{e:ParabolicE} satisfying \eqref{e:assumption_aijbc} and \eqref{e:Lipschitzcondition}, i.e.,  if $u(\cdot, t_0)$ has infinite vanishing order at $x_0\in B_2$ for some fixed $t_0>-4$, then $u(\cdot,t_0)\equiv 0$ at time $t_0$, see \cite{Chen98a,EsFr,EV,Linunique}.
\end{remark}

\begin{remark}
    Following the same lines as in \cite{HJ} we can improve the Hausdorff estimates \eqref{e:Finite_Measure_Thm} to Minkowski estimates $\Vol(B_r(Z_{t_0} \cap B_1)) \le C(n,\lambda,\Lambda) r$. 
\end{remark}

%{\color{blue} It is known from \cite{Man,Miller} that there exists solution $u$ of \eqref{e:ParabolicE} in $(x,t)\in \mathbb{R}^2\times \mathbb{R}$ such that $u$ is nonzero but $u(\cdot,t)\equiv 0$ for all $t\ge T$, where the coefficients $a^{ij}(x,t)$ are smooth with respect to $x$ and H\"older continue with respect to $t$ for any H\"older exponent $\alpha<1$.  In particular, the assumption \eqref{e:assumption_aijbc} and \eqref{e:Lipschitzcondition} are satisfied. Therefore, the assumption that $u(\cdot, t_0)$ is not identically zero is necessary. }

Note that here our estimate is local. Hence one can directly apply it to the global solutions in $\mathbb{R}^n\times \mathbb{R}$. We say a solution $u(x,t)$ of \eqref{e:ParabolicE} in $\mathbb{R}^n\times (-4,0]$ satisfies the backward uniqueness property if the following holds: $$ \text{(BUP):  ~~~If $u(\cdot,0)\equiv 0$ then $u\equiv 0$ in $\mathbb{R}^n\times (-4,0].$} ~$$
As a direct consequence, Theorem \ref{t:mainLip} gives the following theorem.
\begin{theorem}\label{t:mainLipBU}
     Let $u$ be a nonzero solution of \eqref{e:ParabolicE} in $\mathbb{R}^n\times (-4,0]$ satisfying \eqref{e:assumption_aijbc} and \eqref{e:Lipschitzcondition} and (BUP). Then for any $t>-4$ and any $x\in \mathbb{R}^n$ we have 
     \begin{align}
         \cH^{n-1}(Z_t\cap B_1(x))<\infty.
     \end{align}
\end{theorem}

The backward uniqueness property for parabolic equations has been extensively investigated, see for instance \cite{CM22,DP,ESV,Linunique,LM,Poon,WuZh}. From \cite{WuZh}, we know that the BUP holds if $|\nabla_xa^{ij}(x,t)|\le \frac{\lambda}{1+|x|}$ and $|\partial_t a^{ij}(x,t)|\le \lambda$ and $|u(x,t)|\le B e^{A|x|^2}$ for some $A,B>0$. Therefore,  Theorem \ref{t:mainLipBU} recovers the main theorem of Angenent \cite{An} and generalizes it to higher dimensions. Moreover, while Angenent \cite{An} proved the number of zero point set is locally finite, Theorem \ref{t:mainLipBU} gives an explicit and effective control of the number of zero points.

Actually, our method also works for equations with H\"older continuous leading coefficients. For any fixed $0<\alpha<1$, assume $a^{ij}$ satisfies H\"older condition with respect to parabolic distance 
\begin{align}\label{e:Holdercondition}
    |a^{ij}(x,t)-a^{ij}(y,s)|\le \lambda \left(||x-y||^2+|s-t|\right)^{\alpha/2}.
\end{align}

Note that in this case the unique continuation property fails, the following doubling assumption at time $t=0$ must be imposed to exclude the bad behaviors that the solution vanishes in some open ball. See \cite{HLparabolic} for similar growth assumption and see \cite{HJ} for the elliptic case.
\begin{align}\label{e:doublingat0}
    \sup_{Q_{2r}(x,0) \subset Q_2}  \frac{\fint_{ Q_{2r}(x,0)} u^2}{\fint_{Q_{r}(x,0)} u^2}  \leq \Lambda,   
\end{align}
where $Q_r(x,0)=\{(y,s): |x-y|<r, -r^2<s\le 0\}.$

\begin{theorem}\label{t:mainHolder}
     Let $u$ be a solution of \eqref{e:ParabolicE} in $Q_2$ satisfying \eqref{e:assumption_aijbc} \eqref{e:Holdercondition}  and \eqref{e:doublingat0}. Assume $u(\cdot, 0)$ is not identically zero, then we have the following Hausdorff measure estimate
    \begin{align}
        \cH^{n-1}(Z_{0} \cap B_1)<C(n,\alpha,\lambda,\Lambda)<\infty,
    \end{align}
    where $Z_{0}=\{x\in B_2: u(x,0)=0\}$. 
\end{theorem}

\begin{remark}
    Recall that Han-Lin \cite{HLparabolic} proved the space-time estimate for nodal set by assuming the doubling assumption for all subballs, while here we only require the doubling at time $t=0$ to get the measure estimate at this time slice.
\end{remark}

The main body of this paper will be devoted to proving the more general H\"older case, i.e. Theorem  \ref{t:mainHolder}. Then Theorem \ref{t:mainLip} is a direct consequence once the doubling property \ref{e:doublingat0} is established under the Lipschitz assumption.

\subsection{Dimension Monotonicity Theorem}

Next we study the evolutionary behavior of nodal sets over time. According to the work of Angenent \cite{An}, when $n=1$ and additionally $a_t$ is bounded,  the number of zero points of $u(x,t)$, which is a solution to $u_t = a(x,t) u_{xx} + b(x,t) u_x + c(x,t) u$ with Dirichlet boundary condition, is non-increasing in $t$. This result suggests that in principle the complexity of solutions to parabolic equations decays over time, naturally raising interest in the behavior of such solutions in higher dimensions.  However, as far as we know, nothing has been established for general dimensions yet.

In this subsection, we continue the study of evolutionary properties of parabolic solutions in higher dimensions. Unlike the case $n=1$ in \cite{An}, in general, one cannot expect the Hausdorff measure of the nodal set to be non-increasing. We construct examples demonstrating that the Hausdorff measure of the nodal sets can increase over time, even for solutions to heat equations when $n\ge 2$. 

\begin{example}
 Let $f(s)=\frac{3}{2}s-\frac{1}{2}s^2-1$ when $0\le s\le 3/2$ and $f(s)\equiv 1/8$ when $s\ge 3/2$. Then $f$ is $C^1$-smooth and $s=1$ is the only zero point. Consider the heat solution $$u(x,t)=\int_{\mathbb{R}^2}f(|y|^2)\rho_t(x,y)dy.$$
It is easy to check that $u(x,t)=u(x',t)$ for $|x|=|x'|$. Denote $f(r^2,t)=u(x,t)$ with $|x|=r$. Then $f(s,0)=f(s)$,  noting that $f(1)=0$ and $f'(1)=1/2>0$, $\Delta f(|y|^2)|_{\{|y|=1\}}=-2<0$, we can see that $\partial_sf(1,t)>1/4>0$, $\partial_tf(1,t)<-1<0$ for $t$ sufficiently small. This implies $f(1,t)<0$ and $f(r^2_t,t)=0$ for some $r_t>1$. Hence the nodal set $Z_t=\{x: |x|=r_t\}$ and $\cH^{1}(Z_t)=2\pi r_t$ is increasing near $0$.
\end{example}

\begin{example}
On the other hand, one can also construct solution with Dirichlet boundary. Let $f(s)=\frac{3}{2}s-\frac{1}{2}s^2-1$ when $0\le s\le 3/2$ and $f(s)\equiv 1/8$ when $3/2\le s\le 2$ and $f(s)=\frac{1}{24}(s^2-8s+15)$ when $2\le s\le 3$. By smoothing $f$ near $s=3/2$ and $s=2$, we get a smooth $\tilde{f}$. Consider the solution $u(x,t)$ of $(\partial_t-\Delta)u=0$ in $B_3(0^2)\times [0,1]$ with Dirichlet boundary $u=0$ on $\partial B_3\times [0,1]$ and $u(x,0)=\tilde{f}(|x|^2)$. One can check as above that near $t=0$ the Hausdorff measure of $Z_t=\{x: u(x,t)=0, |x|<3\}$ is increasing. Actually, one can use Dirichlet heat kernel estimate to show that $u>0$ near the boundary (\cite{MS}). Hence we only need to consider the situation near $|x|=1$.
\end{example}

These examples reveal that the evolutionary behavior of parabolic solutions is too complicated to expect any monotonicity in the size of nodal sets in higher dimensions. However, while measure monotonicity appears unrealistic, dimension monotonicity is a more practical expectation. In our second result, we provide an affirmative answer to this dimension monotonicity. Similar to the case when $n=1$, if we assume $a^{ij}$ is Lipschitz in the time direction, i.e., $|a^{ij}(x,t)-a^{ij}(x,s)|\le \lambda |s-t|$, then we can show that the Hausdorff dimension of nodal sets is non-increasing over time. 

\begin{theorem}\label{t:non-increasingNodal}

    Let $u$ be a {nonzero} solution of \eqref{e:ParabolicE} in $Q_2$ with Dirichlet boundary condition $u=0$ on $\{|x|=2,-4<t\le 0\}$. Assume that the coefficients satisfy \eqref{e:assumption_aijbc} and \eqref{e:Lipschitzcondition} and $a^{ij}$ is Lipschitz in time direction. Then the Hausdorff dimension of the nodal set $Z_t:=\{x: u(x,t)=0, |x|<2\}$ is non-increasing,
    \begin{align}
        \dim Z_t\le \dim Z_s\le n-1
    \end{align}
    for any $0\ge t\ge s>-4$.
\end{theorem}

\begin{remark}
    It should be noted that the Lipschitz assumption in the time direction is sharp, as it ensures the unique continuation property for backward uniqueness of parabolic equations. According to the counter-examples in \cite{Man, Miller}, if $a^{ij}$ is only H\"older continuous in time, i.e. $|a^{ij}(x,t)-a^{ij}(x,s)|\le \lambda |s-t|^{\alpha}$ for $\alpha \in (0,1)$, then there exists solution $u$ such that $u(\cdot, T) \equiv 0$ while $u(\cdot, t)$ is not for $t < T$.
\end{remark}

This is the first monotonicty property over time for nodal sets of higher-dimensional parabolic solutions. Similar to the applications in the $n=1$ case to curve shortening flows or mean curvature flows, as seen in \cite{AAG, An2, An05,CZ,CHH,DHS,Gr}, more geometric applications are expected. Furthermore, as noted in Remark \ref{r:Thm_holds_Inequality}, Theorem \ref{t:non-increasingNodal} holds in the more general inequality setting. This broadens its applicability to more nonlinear equations.

\subsection{Outline}
In Section \ref{s:2_HE}, we study standard heat equations, recall Poon's monotonicity, and prove some of its consequences. We also establish the quantitative uniqueness for tangent maps of caloric functions and introduce the concept of quantitative stratification. In Section \ref{s:3_Para}, we examine general parabolic equations, introducing the localized frequency function and proving the almost monotonicity of this function. In Section \ref{s:4_Uniq}, we prove the quantitative uniqueness of tangent maps for general parabolic solutions. In Section \ref{s:5_Cone}, we present the cone-splitting theorem for general parabolic solutions. The quantitative uniqueness and the cone-splitting theorem are then utilized to establish the neck region theory in Section \ref{s:6_NR}. Finally, in Section \ref{s:7_proof}, we conclude with the proof of the main theorems.

\textbf{Acknowledgements}
The authors would like to thank Prof. Toby Colding and Prof. Gang Tian for their constant support and their interests of this work.  The authors would like to thank Prof. Kening Lu for bringing the reference \cite{An} to our attention and raising the question about dimension monotonicity during a conference at Qingdao in 2023. W. Jiang was supported by National Key Research and Development Program of China (No. 2022YFA1005501), National Natural Science Foundation of China (Grant No. 12125105 and 12071425) and the Fundamental Research Funds for the Central Universities K20220218.  Y. Huang was partially supported by NSF DMS Grant 2104349.

\section{Background and Heat Equations}\label{s:2_HE}
In this section, we will recall and prove some results which will be used in our proofs.  First we will recall some results of caloric functions.
\subsection{Caloric Function and monotone frequency}
In this subsection, we consider the solution to the standard heat equation in $\RR^n \times [-T, 0]$,
\begin{equation*}
    \partial_t u = \Delta u.
\end{equation*}
Noting that if there is no any growth control the solution may not be unique with given initial data.  Let us assume some appropriate growth assumption on $u$, say polynomial growth. 
First we recall the monotonicity formula for $u$ as in Poon \cite{Poon}.
Fix some point $(x_0, t_0)$. Consider 
\begin{equation*}
\begin{split}
    E_{x_0, t_0}(r) &= 2 r^2 \int_{t = t_0 - r^2} |\nabla u|^2 G_{x_0, t_0} \\
    H_{x_0, t_0}(r) &= \int_{t=t_0 - r^2} u^2 G_{x_0, t_0}
\end{split}
\end{equation*}
where
\begin{equation*}
    G_{x_0,t_0}(x,t) \equiv (4\pi(t_0-t))^{-\frac{n}{2}} \text{e}^{-\frac{|x-x_0|^2}{4(t_0 - t)}}.
\end{equation*}

We define the frequency for $r\ge 0$ to be
\begin{equation}\label{e:frequency_heat}
    N_{x_0,t_0}(r) = \frac{E_{x_0, t_0}(r)}{H_{x_0, t_0}(r)}.
\end{equation}

It turns out that $N$ is non-increasing in $r$. For the sake of completeness, here we include the proof, which follows Poon \cite{Poon}. 
\begin{theorem}[Poon]
    $N_{x_0, t_0}'(r) \geq 0$ for any $(x_0, t_0)$ and any $r>0$.
\end{theorem}

\begin{proof}
    Observe that
\begin{equation*}
    \nabla G_{x_0, t_0} = \frac{x-x_0}{2(t-t_0)} G_{x_0, t_0} \quad \text{ and } \quad \partial_t  G_{x_0, t_0} = - \Delta G_{x_0, t_0}.
\end{equation*}

For simplicity we omit the subscripts. Then integration by parts implies that 
\begin{equation*}
    E(r) = - 2r^2 \int_{t=t_0 - r^2} u (\Delta u + \nabla u \cdot \frac{x-x_0}{2(t - t_0)} )G_{x_0, t_0} = - 2r^2 \int_{t=t_0 - r^2} u (u_t + \nabla u \cdot \frac{x-x_0}{2(t-t_0)} )G_{x_0, t_0}
\end{equation*}

\begin{equation*}
H'(r) = -2r \int_{t=t_0 - r^2} 2u u_t G_{x_0, t_0} - u^2 \Delta G_{x_0, t_0} = - 4r\int_{t=t_0 - r^2} u(u_t + \nabla u \cdot \frac{x- x_0}{2(t-t_0)} )G_{x_0,t_0} = \frac{2 E(r)}{r}.
\end{equation*}

\begin{equation*}
\begin{split}
    E'(r) &= 4 r \int_{t=t_0 - r^2} |\nabla u|^2 G_{x_0,t_0}  -  4r^3 \int_{t=t_0 - r^2} 2\nabla u \cdot \nabla u_t G_{x_0,t_0} + 4 r^3 \int_{t=t_0 - r^2} |\nabla u|^2 \Delta G_{x_0,t_0} \\
    &= 4 r \int_{t=t_0 - r^2} |\nabla u|^2 G_{x_0,t_0}  -  8r^3 \int_{t=t_0 - r^2} (\nabla u \cdot \nabla u_t  + \nabla^2 u (\nabla u, \frac{x- x_0}{2(t - t_0)}))G_{x_0,t_0} \\ 
    &= - 8 r^3 \int_{t=t_0 - r^2} (\nabla u \cdot \nabla ( u_t  + \nabla u \cdot \frac{x- x_0}{2(t - t_0)}) G_{x_0,t_0} \\
    &= 8 r^3 \int_{t=t_0 - r^2} ( u_t  + \nabla u \cdot \frac{x- x_0}{2(t - t_0)})^2 G_{x_0,t_0}.
\end{split}
\end{equation*}

Therefore, 
\begin{equation*}
\begin{split}
    & \quad E'(r) H(r) - E(r) H'(r)\\
    &= 8 r^3  \Big(\int_{t=t_0 - r^2} u^2 G_{x_0,t_0}\Big) \Big(\int_{t=t_0 - r^2} \Big( u_t  + \nabla u \cdot \frac{x- x_0}{2(t - t_0)}\Big)^2 G_{x_0,t_0}\Big)  - 8 r^3 \Big(\int_{t=t_0 - r^2} u (u_t + \nabla u \cdot \frac{x-x_0}{2(t-t_0)} )G_{x_0, t_0}   \Big)^2.
\end{split}
\end{equation*}

Hence, by Cauchy-Schwarz inequality we have $N'(r) \geq 0$. 
\end{proof}

Since $H'(r) = 2E(r)/r$, we have $H'(r)/H(r) = 2N(r)/r$. Integrating from $r_1$ to $r_2$ we have
\begin{equation}\label{e:Global_Doubling}
    \log \frac{H(r_2)}{H(r_1)} = 2\int_{r_1}^{r_2} \frac{N(s)}{s} ds. 
\end{equation}
If we define the doubling index $D(r)$ as 
\begin{equation*}
    D(r) \equiv \log_4 \frac{H(2r)}{H(r)}.
\end{equation*}
Then from \eqref{e:Global_Doubling} and the monotonicity of $N(r)$ we can establish the equivalence between frequency and doubling index
\begin{equation}\label{e:Frequency_equv_DI_global}
    N(r) \le D(r) \le N(2r).
\end{equation}
This is an important global doubling control with respect to frequency. It could be used to prove the unique continuation for heat equation.

We say $u$ is a \textbf{homogeneous caloric polynomial} of order $k$ centered at $(x_0, t_0)$, if $u$ satisfies the heat equation and $u(\lambda x+x_0, \lambda^2 t+t_0) = \lambda^k u(x+x_0,t+t_0)$ for any $t< 0$ and $\lambda > 0$. 

\begin{lemma}\label{l:N_constant_equiv_polynomial}
$N(r)$ is constant if and only if $u$ is a homogeneous caloric polynomial of order $N(1)$. 
\end{lemma}

\begin{proof}
    We may assume $(x_0, t_0) = (0,0)$ and $u(0,0)=0$. By Cauchy-Schwarz inequality, $N'(r) \equiv 0$ if and only if there exists some $k$ such that
    \begin{equation} \label{e:Homoge_Poly_condition}
    2 u_t \cdot t + \nabla u \cdot x = k \cdot u.
\end{equation}

First we assume (\ref{e:Homoge_Poly_condition}) and prove that $u$ is a homogeneous caloric polynomial of order $k$. Fix $(x,t)$. Set $F(\lambda) = u(\lambda x, \lambda^2 t)-\lambda^k u(x,t)$. Note that by (\ref{e:Homoge_Poly_condition}) we have
\begin{equation*}
\begin{split}
    F'(\lambda) &= \nabla u(\lambda x, \lambda^2 t) \cdot x + u_t(\lambda x, \lambda^2 t) \cdot (2\lambda t) - k \lambda^{k-1} u(x,t)\\
    &= \lambda^{-1} (\nabla u(\lambda x, \lambda^2 t) \cdot (\lambda x) + u_t(\lambda x, \lambda^2 t) \cdot (2 \lambda^2 t)) - k \lambda^{k-1} u(x,t)\\
    &= \lambda^{-1} k u(\lambda x, \lambda^2 t) - k \lambda^{k-1} u(x,t) \\
    &= \lambda^{-1} k F(\lambda).
\end{split}
\end{equation*}
Since $F(1) = 0$, this implies that $F(\lambda) \equiv 0$. The other direction is trivial by chain rule. The proof is finished.

The following Liouville type theorem for ancient caloric function is standard.
\end{proof}
\begin{lemma}\label{l:polynomialcalaric}
    Let $u$ be a caloric function with polynomial growth of order $\Lambda$, i.e. $|u(x,t)| \le C(1 + |x| + |t|^{1/2})^{\Lambda}$. Then $u$ is a polynomial of order $d\le \Lambda$.
\end{lemma}
\begin{proof}
    For any $(x,t)$, denote $R=\left(|x|^2+|t|\right)^{1/2}$. For any $r\ge 2R$, by parabolic estimate we get for any $\ell+2k=d$
    \begin{align}
        |D_x^\ell D_t^k u(x,t)|\le C(n,d) r^{-\ell-2k}\left(\fint_{Q_{r}}|u(y,s)|^2dyds\right)^{1/2}\le C r^{-\ell-2k+\Lambda}.
    \end{align}
    If $\ell+2k>\Lambda$, letting $r\to \infty$ we have  $D_x^\ell D_t^k u(x,t)=0$. This implies $u$ is a polynomial with order at most $\Lambda$.
\end{proof}

As a corollary, we notice that the frequency can only be pinched near the integers. 
\begin{lemma}\label{l:Caloric_frequency_near_Z}
    Let $u$ be a caloric function with $N(1) \le \Lambda$ and with polynomial growth of order $\Lambda$, i.e. $|u(x,t)| \le C(1 + |x| + |t|^{1/2})^{\Lambda}$. For any $\epsilon > 0$, there exists some $\delta(\epsilon, \Lambda) > 0$ such that if $N(1) - N(1/10) \leq \delta$, then there exists some integer $d$ such that
    \begin{enumerate}
        \item $|N(r) -d| \leq \epsilon$ for any $r \in [1/10, 1]$.
        \item There exists some caloric homogeneous polynomial $P$ of order $d$ such that $|u - P|_{K_{\epsilon^{-1}}} \le \epsilon$, where $K_r = \{ (x, t) : |x| \le r, 1/10 + 1/r\le t \le 1 \}$. 
    \end{enumerate} 
\end{lemma}
\begin{proof}
    This is proved by contradiction argument. Let $u_i$ be a sequence of caloric functions with $N^i(1) \le \Lambda$ and $|u_i(x,t)| \le C_i(1 + |x| + |t|^{1/2})^{\Lambda}$. And $N^i(1) - N^i(1/10) \le 1/i$ with dist$(N(r), \mathbf{Z}) \geq \epsilon_0$ for some $r \in [1/10, 1]$. We may normalize it so that $H^i(1) = \int_{t=-1} u_i^2 G_{0,0} = 1$ for all $i$. By (\ref{e:Global_Doubling}) and $N^i(1) \le \Lambda$, we have $H^i(1/10) \geq 10^{-\Lambda}$. {Also, we can take $C_i = C$ for some fixed $C$ after the normalization. }
    
    Consider $K_R = \{ (x, t) : |x| \le R, 1/10 + 1/R\le t \le 1 \}$. 
    Since $G_{0,0}(x, t) \ge c(R)$ on $K_R$, we have $\int_{K_R} u_i^2 \le C(R)$. Hence for each $R$ we can extract a subsequence of $u_i$ that converges to $u_R$ in $C^{\infty}(K_{R/2})$. By diagonal argument, there exists some caloric function $u_{\infty}$ such that $u_i$ converges to $u_{\infty}$ in $C^{\infty}_{loc}(K_{\infty})$ by passing to a subsequence. Since $u_i$ has uniform polynomial growth, by dominated convergence theorem we have $H^i(r) \to H^{\infty}(r)$ and $E^i(r) \to E^{\infty}(r)$ for any $r \in [1/10 ,1]$. Since $H^i(r) \ge 10^{-\Lambda}$ for any $r \in [1/10, 1]$, we have 
    \begin{equation*}
        N^{i}(r) = \frac{E^i(r)}{H^i(r)} \to N^{\infty}(r).
    \end{equation*}

Hence by assumption and monotonicity of $N$, we have $N^{\infty}(1) = N^{\infty}(1/10) $. This implies that $u_{\infty}$ is indeed a homogeneous caloric polynomial and thus $N^{\infty} = d$ for some integer $d$. Contradiction arises.    
\end{proof}

The following Lemma says that away from integers the frequency drops at a definite rate.

\begin{lemma}\label{l:Caloric_fre_drop}
    Let $u$ be a caloric polynomial with $N(1) \le \Lambda$. For any $\epsilon > 0$, there exists some $\delta(\epsilon, \Lambda) > 0$ such that if $N(1) \leq d - \epsilon$, then $N(\delta) \leq d - 1 +\epsilon$. 
\end{lemma}
\begin{proof}
    When $\epsilon \geq 1/2$, the conclusion trivially holds by monotonicity of $N$. Assume $\epsilon < 1/2$. Let $\delta_0(\epsilon/10, \Lambda)$ be as in Lemma \ref{l:Caloric_frequency_near_Z}. 
    
    If $N(1/10) \geq N(1)-\delta_0$, then we have dist$(N(1), \mathbf{Z}) \le \epsilon/10$ by Lemma \ref{l:Caloric_frequency_near_Z}. This implies that $N(1) \le d-1+\epsilon/10$ and thus $N(1/10) \le N(1) \le d-1+\epsilon/10$ by monotonicity.

    Now suppose $N(1/10) \le N(1)-\delta_0$. Applying Lemma \ref{l:Caloric_frequency_near_Z} to the scale $1/10$, we can conclude that $N(10^{-2}) \leq N(1/10) - \delta_0$ by the same arguments. Therefore, after iterating by at most $K$ times with $K \le \delta_0^{-1}$, we can conclude that $N(10^{-K}) \le N(1) - K \delta_0 \le d-1+\epsilon$.
\end{proof}

\begin{remark}
    Instead of requiring $u$ to be a polynomial, actually we only need some minor growth assumption on $u$ for Lemma \ref{l:Caloric_frequency_near_Z} and Lemma \ref{l:Caloric_fre_drop} to hold. 
\end{remark}

\begin{remark}
    It is possible to prove that $\epsilon$ is independent of $\Lambda$ by delicate analysis like what \cite{NV} did for harmonic function. %This step is necessary if one wants to get an explicit volume estimate for singular set.
\end{remark}

The following Lemma proves the orthogonality of homogeneous caloric polynomial of different orders(see also \cite{CMcaloric}).

\begin{lemma}\label{l:homogeneous_Poly_orthogonal}
    Suppose $P$, $Q$ be two homogeneous caloric polynomials with order $d_1 \neq d_2$. Then we have
    \begin{equation*}
        \int_{t= t_0} PQ G_{0,0} =  \int_{t=t_0} (\nabla P \cdot \nabla Q )G_{0,0} = 0.
    \end{equation*}
\end{lemma}

\begin{proof}
    Note that
\begin{equation*}
    d_1 \int_{t= t_0} P Q G_{0,0} = \int_{t=t_0} (2 \partial_t P \cdot t + \nabla P \cdot x) Q G_{0,0} =    \int_{t=t_0} (2t_0 \Delta P + \nabla P \cdot x) Q G_{0,0}. 
\end{equation*}
Integration by parts gives
\begin{equation*}
     2t_0 \int_{t=t_0}  \Delta P  Q G_{0,0} = -2 t_0\int_{t=t_0} (\nabla P \cdot \nabla Q )G_{0,0} - \int_{t=t_0}(\nabla P \cdot x)  Q G_{0,0}.
\end{equation*}

Hence we have
\begin{equation*}
    d_1 \int_{t= t_0} P Q G_{0,0} =- 2t_0 \int_{t=t_0} (\nabla P \cdot \nabla Q )G_{0,0}. 
\end{equation*}

On the other hand, by similar arguments we have
\begin{equation*}
    d_2\int_{t= t_0} P Q G_{0,0} = -2t_0 \int_{t=t_0} (\nabla P \cdot \nabla Q )G_{0,0}.
\end{equation*}

Since $d_1 \neq d_2$, the result follows.
\end{proof}

\begin{remark}
    Let $P(x, t)$ be a homogeneous caloric polynomial of order $d$. Then $P(x, -1)$ is an eigenfunction of the Ornstein-Uhlenbeck operator $\cL(u) = \Delta u - \frac{1}{2} \langle \nabla u, x \rangle$ with eigenvalue $d/2$. Note that $\cL$ is self-adjoint with respect to the weighted volume $\langle f, g \rangle = \int_{\RR^n} (f g) \e^{-\frac{|x|^2}{4}}$. If $P, Q$ are homogeneous caloric polynomials with different orders $d_1 ,d_2$, we have $\langle P(x, -1), Q(x, -1) \rangle \equiv 0$ as they are eigenfunctions of $\cL$ with different eigenvalues. 
\end{remark}

Next we prove that the frequency of a polynomial is bounded by its order. This is a direct consequence of the previous orthogonality lemma. 

\begin{corollary}\label{c:Frequency_bound_by_order_Poly}
    Let $P$ be a caloric polynomial of order $d$. Then we have $N^P(r) \leq d$ for all $r>0$.
\end{corollary}

\begin{proof}
    Write $P= \sum_{i=0}^d P_i$, where each $P_i$ is a homogeneous caloric polynomial of order $i$. Then by previous Lemma \ref{l:homogeneous_Poly_orthogonal}
\begin{equation*}
\begin{split}
    \int_{t=-r^2} |\nabla P|^2 G_{0,0} &= \int_{t=-r^2} \langle \sum_i \nabla P_i, \sum_j \nabla P_j \rangle G_{0,0} = \sum_i \int_{t=-r^2} |\nabla P_i|^2 G_{0,0}\\
    \int_{t=-r^2} P^2 G_{0,0} &= \int_{t=-r^2} (\sum P_i)^2 G_{0,0} = \sum_i \int_{t=-r^2}  P_i^2 G_{0,0}. 
\end{split}
\end{equation*}

By Lemma \ref{l:N_constant_equiv_polynomial} or the calculation in Lemma \ref{l:homogeneous_Poly_orthogonal} we have $2 r^2\int_{t=-r^2} |\nabla P_i|^2 G_{0,0} = i \int_{t=-r^2}  P_i^2 G_{0,0}$ for any $i \le d$. This implies that $$2r^2 \int_{t=-r^2} |\nabla P|^2 G_{0,0} \leq d \int_{t=-r^2} P^2 G_{0,0}. $$ Thus $N^P(r) \le d$.
\end{proof}
\begin{corollary}\label{c:smallfrequency}
     Let $P$ be a caloric polynomial of order $d$. For any $0<\epsilon<1$, if $N^P(1)\le \epsilon$ then for any $R>1$
     \begin{align}
         \sup_{Q_R}|P(x,t)-a_0|\le C(n,d,R)  |a_0|\sqrt{\epsilon}
     \end{align}
     for some nonzero constant $a_0$.
\end{corollary}

\begin{proof}
    Write $P= \sum_{i=0}^d a_i P_i$, where each $P_i$ is a homogeneous caloric polynomial of order $i$ with unit $\int_{t=-1}|P_i(y,s)|^2G_{0,0}dyds$. By previous Lemma \ref{l:homogeneous_Poly_orthogonal}, we can compute that 
    \begin{align}
          \int_{t=-1} |\nabla P|^2 G_{0,0} &= \int_{t=-1} \langle \sum_i a_i\nabla P_i, \sum_j a_j\nabla P_j \rangle G_{0,0} = \sum_{i=1}^{d} \int_{t=-1} a_i^2|\nabla P_i|^2 G_{0,0}=\sum_{i=1}^{d} \int_{t=-1} a_i^2|\nabla P_i|^2 G_{0,0}.
    \end{align}
Noting that $2 \int_{t=-1} |\nabla P_i|^2 G_{0,0} = i \int_{t=-1}  P_i^2 G_{0,0}=i$, we get 
\begin{align}
    N^P(1)=\frac{2\int_{t=-1} |\nabla P|^2 G_{0,0}}{\int_{t=-1} | P|^2 G_{0,0}}=\frac{2\sum_{i=1}^da_i^2 i}{\sum_{i=0}^da_i^2}\le \epsilon.
\end{align}
This implies 
\begin{align}
 \sum_{i=1}^da_i^2\le \frac{\epsilon}{2-\epsilon}a_0^2.   
\end{align}
This gives the deserved estimate. 
\end{proof}

\subsection{Uniqueness of tangent maps of heat equations}
In this subsection we prove an effective estimate for tangent map of heat equation. Let $h$ define on $Q_2=B_2\times [-4,0]$. For any $(x_0,t_0)\in Q_1$, define a localized frequency for $h$ by
$$N^h_{x_0,t_0}(r)=\frac{E_{x_0,t_0}^h(r)}{H^h_{x_0,t_0}(r)}=\frac{ 2r^2\int_{Q_2\cap \{t = t_0 - r^2\}} |\nabla u|^2 G_{x_0, t_0} }{\int_{Q_2\cap \{t=t_0 - r^2\}} u^2 G_{x_0, t_0}}. $$

See a detailed discussion of localized frequency in \ref{d:localizedfrequency} in the next section. 
\begin{theorem}\label{t:heat_equ_unique_tangent}
    Let $h$ be a solution of heat equation on $Q_2$ with doubling assumption $\sup_{r_2\le r\le r_1}\text{\rm log}_4 \frac{\fint_{ Q_{2r}(0,0)} h^2}{\fint_{Q_{r}(0,0)} h^2}  \leq \Lambda$ . For any $\epsilon>0$ there exist $\delta_0(n,\Lambda,\epsilon)$ such that if the localized frequency $|N_{0,0}^{h}(r)-d|\le \delta\le \delta_0 $ for any $r_2\le r \le r_1\le 1$ and for some integer $d$. Then there exists a unique homogeneous caloric polynomial $P_d$ of order $d$ such that for any $r_2\le r\le  \delta r_1$ that 
    \begin{align}\label{e:effectiveuniquePd}
        \sup_{Q_1}|h_{0,0;r}-P_d|\le \epsilon
    \end{align}
    where $h_{0,0;r}(x,t)=\frac{h(r x,r^2 t)}{\left(\fint_{Q_1}|h(r x,r^2 t)|^2dxdt\right)^{1/2}}$.
\end{theorem}

%{\color{blue} Here the doubling assumption should hold for any $r_1 \le r \le 1$. The result should hold for $2r_1 \le s \le \delta$? $h_2$ is defined on $Q_{2}^{\eta}.$}

\begin{proof}
%By Lemma \ref{l:Frequency_Close_away_from_0}, we know that the localized frequency would be close to the frequency of an ancient caloric function which is monotone and pinching only at integers. Hence, by the condition $|N_{0,0}^h(s)-N_{0,0}^h(r_1)|\le \delta $, there exists a unique integer $d$ such that $N_{0,0}^h(s)$ will be sufficiently close to such $d$ for all $r_1\le s\le r_2$. 
Since $h$ is solution of heat equation, then $h$ is smooth. By Taylor expansion at $(0,0)$, we can get (see also Lemma 1.2 of Han \cite{Han98}) for all $(x,t)\in Q_1$
\begin{align}\label{e:expansion_uQ1}
    h(x,t)=P_0(x,t)+P_1(x,t)+P_2(x,t)+\cdots +P_d(x,t)+R(x,t)
\end{align}
where each $P_k$ is a caloric homogeneous polynomial of order $k$ and $R(x,t)$ satisfies for  $(x,t)\in Q_r$ with $r\le 1$ that 
\begin{align}
    |R(x,t)|\le C(n,d)r^{-d-1}|(x,t)|^{d+1}\left(\fint_{Q_{2r}}|h(y,t)|^2dydt\right)^{1/2}.
\end{align}
Furthermore, the coefficients of each $P_i$ are uniformly bounded by the $L^2$-norm of $h$ over $Q_1$. Since $R(x,t)$ is also a solution of heat equation, by parabolic estimates, we get for all $(x,t)\in Q_{r}$ with $r\le 1$ and $\ell+2k\le d$
\begin{align}\label{e:derivativeRellk}
    |D_x^\ell D_t^k R(x,t)|\le C(n,d)r^{-d-1+\ell+2k}|(x,t)|^{d+1-\ell-2k} \left(\fint_{Q_{2r}}|h(y,t)|^2dydt\right)^{1/2}.
\end{align}
We will show that the caloric polynomial $P_d$ from \eqref{e:expansion_uQ1} is the desired polynomial up to normalization with unit $L^2$-norm on $Q_1$.  To see this,  it suffices to show that for any $\epsilon>0$ if $\delta\le \delta(n,\Lambda,\epsilon)$ and $r_2\le r\le \delta r_1$ we have
\begin{align}\label{e:errorPiepsilon}
   \fint_{Q_r}|\sum_{i=0}^{d-1}P_i|^2\le \epsilon \fint_{Q_{r}}|h(y,t)|^2dydt 
\end{align}
and 
\begin{align}\label{e:errorRepsilon}
    \fint_{Q_r}|R(y,t)|^2dydt\le \epsilon \fint_{Q_{r}}|h(y,t)|^2dydt 
\end{align}
Let us first prove \eqref{e:errorRepsilon}. Assume there exists $\epsilon_0>0$ and a sequence of $r_{2,i}\le r_i\le \delta_i r_{1,i}$ and $\delta_i\to 0$ and heat solution $h_i(x,t)$ and $R_i(x,t)$ satisfying the assumption of the theorem and \eqref{e:expansion_uQ1}, however, 
\begin{align}
    \fint_{Q_{r_i}}|R_i(y,t)|^2dydt\ge \epsilon_0 \fint_{Q_{r_i}}|h_i(y,t)|^2dydt 
\end{align}
Denote $\hat{h}_i(y,s)=\frac{h_i(r_i y,r_i^2 s)}{\left(\fint_{Q_1}|h_i(r_i y,r_i^2 s)|^2dyds\right)^{1/2}}$. Then 
\begin{align}\label{e:hatRihatuicontradict}
    \fint_{Q_{1}}|\hat{R}_i(y,t)|^2dydt\ge \epsilon_0 \fint_{Q_{1}}|\hat{h}_i(y,t)|^2dydt 
\end{align}
By apriori estimate for parabolic equation and the polynomial growth assumption of $h_i$, we know that $\hat{h}_i$ and $\hat{R}_i$ will converge smoothly to ancient solutions of heat equation $\hat{h}_\infty$ and $\hat{R}_\infty$ on $\mathbb{R}^n$. From the expansion \eqref{e:expansion_uQ1}, we also get  
\begin{align}
    \hat{h}_\infty=\sum_{k=0}^d\hat{P}_{k,\infty}+\hat{R}_\infty
\end{align}
where $\hat{P}_{k,\infty}$ is the limit of $\hat{P}_{k,i}$ with $i\to \infty$ (since each $\hat{P}_{k,i}$ has uniformly bounded coefficients which must converge up to a subsequence). On the other hand, by smooth convergence one can easily show that the localized frequency converges to the standard frequency of $\hat{h}_\infty$ and equals to $d$ {(see also Lemma \ref{l:Frequency_Close_away_from_0} )}, then $\hat{h}_\infty$ is a caloric polynomial of order $d$.  We will show that $\hat{R}_\infty$ is zero which will lead to a contradiction from the limit of \eqref{e:hatRihatuicontradict}. Actually, if $\hat{R}_\infty$ is not zero, since $\hat{h}_\infty$ is a homogeneous polynomial of order $d$, $\hat{R}_\infty$ must be a polynomial of order at most $d$ (may not be homogeneous). Assume $\hat{R}_\infty$ contains a non zero term $x^\ell t^k$ with $\ell+2k\le d$. By the smooth convergence of $\hat{R}_i$ we have that 
\begin{align}
   D_x^{\ell}D_t^k\hat{R}_i(0,0)\to D_x^\ell D_t^k\hat{R}_\infty(0,0)\ne 0.
\end{align}
However, by \eqref{e:derivativeRellk} we have $D_x^{\ell}D_t^k\hat{R}_i(0,0)=0$ for each $i$. This is a contradiction. Hence we get \eqref{e:errorRepsilon}. Since $\hat{R}_\infty\equiv 0$ and $\hat{h}_\infty$ is a homogeneous polynomial of order $d$, this implies $\sum_{k=0}^{d-1}\hat{P}_{k,\infty}\equiv 0$. This gives \eqref{e:errorPiepsilon}. Therefore, by \eqref{e:expansion_uQ1}, \eqref{e:errorPiepsilon},\eqref{e:errorRepsilon}, we get for all $r_2\le r\le \delta r_1$ that
\begin{align}
    \fint_{Q_r}|h(x,t)-P_d(x,t)|^2dxdt\le 2\epsilon \fint_{Q_r}|h(x,t)|^2dxdt.
\end{align}
This implies \eqref{e:effectiveuniquePd}. We finish the proof.
\end{proof}

\subsection{Maximum Principle}

In this subsection, we recall the maximum principle of parabolic equation with weak coefficients. 
Let us consider the following equation in $Q_2:=\{(x,t): |x|<2, -4\le t \le 0\} \subset \RR^{n} \times \RR$
\begin{equation}\label{e:ParabolicEAP}
    \partial_t u = \partial_i(a^{ij}(x,t) \partial_j u) + b^i(x,t) \partial_i u +c(x, t) u,
\end{equation}
where the coefficients $a^{ij}$ are elliptic and assumed to be in $C^{\alpha, \alpha/2}({Q_2})$ with $\alpha \in (0, 1)$, and the coefficients $b, c$ are bounded,
\begin{equation}\label{e:assumption_aijbcAP}
    (1 + \lambda)^{-1} \delta^{ij} \leq a^{ij} \leq (1 + \lambda) \delta^{ij}, \quad |b^i|, |c| \leq \lambda \text{ and } |a^{ij}(x,t)-a^{ij}(y,s)|\le \lambda\left(|x-y|^2+|s-t|\right)^{\alpha/2}.
\end{equation}
\begin{theorem}[Maximum Principle]\label{t:maximum}
    Let $u$ be a Lipschitz function satisfying \eqref{e:ParabolicEAP} \eqref{e:assumption_aijbcAP} and the Dirichlet boundary condition $u=0$ on $\{|x|=2\}\times [-4,0]$. Suppose $u\le 0$ at $t=-4$. Then $u<0$ for $t>-4$ or $u\equiv 0$ for $t\ge -4$.
\end{theorem}
\begin{proof}
 Let us consider the positive part of $u$ and define $u_{+}(x,t)=\max\{0, u(x,t)\}$. Hence $u_+(x,-4)=0$. To prove the maximum principle it suffices to show $u_{+}\le 0$. Noting that $u(x,t)=0$ on $\partial B_2(0)$, multiplying $u_{+}$ to \eqref{e:ParabolicEAP} and integrating by parts we get
 \begin{align}
     \int_{B_2}a^{ij}\partial_j u_{+}\partial_i u_{+}(x,t)dx-\lambda\int_{B_2}|\nabla u_{+}| |u_{+}|(x,t)dx-\lambda \int_{B_2}|u_+|^2(x,t)dx\le -\frac{1}{2}\partial_t\int_{B_2}|u_+|^2(x,t)dx.
 \end{align}
 By Cauchy inequality we get 
 \begin{align}
      C_1(\lambda)\int_{B_2}|\nabla u_+|^2(x,t)dx-C_2(\lambda)\int_{B_2}|u_+|^2(x,t)dx\le -\partial_t\int_{B_2}|u_+|^2(x,t)dx
 \end{align}
 Integrating $t$ from $-4$ to $s$ we get 
 \begin{align}
      C_1(\lambda)\int_{-4}^s\int_{B_2}|\nabla u_+|^2(x,t)dx-C_2(\lambda)\int_{-4}^s\int_{B_2}|u_+|^2(x,t)dx\le -\int_{B_2}|u_+|^2(x,s)dx+\int_{B_2}|u_+|^2(x,-4)dx.
 \end{align}
Hence for any $s\ge -4$ that
\begin{align}
    \int_{B_2}|u_+|^2(x,s)dx\le C_2(\lambda)\int_{-4}^s\int_{B_2}|u_+|^2(x,t)dx.
\end{align}
Integrating $s$ from $-4$ to $\ell$ we get 
\begin{align}
   \int_{-4}^{\ell} \int_{B_2}|u_+|^2(x,s)dxds\le C_2(\lambda) (\ell+4)  \int_{-4}^{\ell} \int_{B_2}|u_+|^2(x,s)dxds.
\end{align}
If $\ell+4<\frac{1}{2C_2(\lambda)}$, then we have  $\int_{-4}^{\ell} \int_{B_2}|u_+|^2(x,s)dxds=0$ and thus $u_{+}\equiv 0$ in $B_2\times [-4,\ell]$. We can now start at $t=\ell$ and then by induction to deduce that $u_{+}(x,t)=0$ for all $t$.   Hence we have proven the weak maximum principle $u(x,t)\le 0$ for all $t$. 
To see the strong maximum principle, assume $u(x_0,t)=0$ for some $t>-4$ and $|x_0|<2$, by parabolic Harnack inequality (see \cite{Mo,Pa}) we can see that $u(x,-4)\equiv 0$ for all $|x|<2$. %Actually, for any $|x|<2$ we can consider $B_R(0)$ with $\max\{|x|,|x_0|\}<R<2$. Using parabolic inequality to this ball we get that $u(x,-4)=0$. 
    In this case, using weak maximum principle to $-u$ and $u$ we get $u\equiv 0$.
\end{proof}

\subsection{Quantitative Stratification at a time slice}   

In this subsection, we will mainly introduce the quantitative stratification which was first proposed in \cite{CN,CNV}. %The stratification separates points based on the tangential behavior of the functions, more precisely the number of symmetry of the leading homogeneous polynomial of the Taylor expansion of the function. In \cite{CNV}, the stratification was refined in a more quantitative way so that they obtain an effective Minkowski estimate on the quantitative strata. 

First we define the symmetry of functions.
\begin{definition}
Consider the continuous function $u: \mathbb{R}^n \to \mathbb{R}$.
\begin{itemize}
    \item[(1)] $u$ is called $0$-symmetric with respect to $x_0$ if $u$ is a homogeneous polynomial, i.e., $u(\lambda x+x_0)=\lambda^ku(x+x_0)$ for any $\lambda>0$ and some $k$.
    \item[(1)] $u$ is called $k$-symmetric with respect to $x_0+V$ if $u$ is $0$-symmetric with respect to $x_0$ and further symmetic with respect to some $k$-dimensional subspace $V$, i.e. $u(x+y) = u(x)$ for any $x \in \RR^n$ and $y \in V$.
\end{itemize}
\end{definition}
Next we define the quantitative symmetry for $u$.
\begin{definition}\label{d:quan_sym1}
Let $u: B_2 \to \RR$ be a continuous function. For any fixed $x\in B_1$, we define $u$ is $(k, \eta, r, x)$-symmetric if there exists a $k$-symmetric polynomial $P$ with $\fint_{B_1} |P|^2 = 1$ such that 
\begin{equation}
 \sup_{B_1}|u_{x,r}(y) - P(y)| \leq \eta.
\end{equation}
where $u_{x,r}(y)=\frac{u(x+ry)}{\left(\fint_{B_1}|u(x+ry)|^2\right)^{1/2}}$.
\end{definition}

%\begin{remark}
  %  If $u\not\equiv 0$ and  $u$ is $(n,\eta,r,x)$-symmetric for $\eta<1/2$, then $u(x)\ne 0$. In our case, $u$ is $t$-slice of a solution of \eqref{e:ParabolicE} in $Q_2$ with \eqref{e:assumption_aijbc} and \eqref{e:doublingassumption}. One can easily see by contradiction argument that $u(\cdot,t)\not\equiv 0$ in any $t$-slice. Hence if $u(\cdot,t)$ is $(n,\eta,r,x)$-symmetric for some $\eta<1/2$, then $u(x,t)\ne 0$.
%\end{remark}

With this we can now give the definition of quantitative stratification
\begin{definition}
  Let $u: B_2 \to \RR$ be a continuous function. Given $k,\eta>0,r>0$, the $(k,\eta)$-singular stratum is defined by 
    \begin{equation}
        \cS^k_{\eta}(u):= \{ x\in B_1 : u \text{ is not } (k+1, \eta, s, x)\text{-symmetric for any } s \geq 0 \}.
    \end{equation}
\end{definition}

The above definitions will be used in during our proofs of the main theorems.

\section{General Parabolic Equations and Almost Monotone Frequency}\label{s:3_Para}

In this section we will prove the almost monotonicity formula for frequency defined for general parabolic solution $u$. We will consider the solution in a bounded domain. Since the function is locally defined, we need to localized the frequency (see Definition \ref{d:localizedfrequency}) introduced in \cite{Poon} (see also \eqref{e:frequency_heat}). We will see that the localized frequency is approximating to the standard frequency and the localized frequency is almost monotone (see Theorem \ref{t:Frequency_almost_monotone}) which is good enough for the applications. We will 
also deduce some properties of this localized frequency based on the almost monotonicity. 

From now on we consider the following equation in $Q_2:=\{(x,t): |x|<2, -4<t\le 0\} \subset \RR^{n} \times \RR$
\begin{equation}\label{e:Parab_equa}
    \partial_t u = \partial_i(a^{ij}(x,t) \partial_j u) + b^i(x,t) \partial_i u +c(x, t) u,
\end{equation}
where the coefficients $a^{ij}$ are elliptic and assumed to be in $C^{\alpha, \alpha/2}({Q_2})$ with $\alpha \in (0, 1)$, and the coefficients $b, c$ are bounded,
\begin{equation}\label{e:assumption_without_c}
    (1 + \lambda)^{-1} \delta^{ij} \leq a^{ij} \leq (1 + \lambda) \delta^{ij}, \quad |b^i|, |c| \leq \lambda \text{ and } |a^{ij}(x,t)-a^{ij}(y,s)|\le \lambda\left(|x-y|^2+|s-t|\right)^{\alpha/2}
\end{equation}

We consider the following growth assumption at $t=0$ for solution $u$ (See \cite{HLparabolic} \cite{HJ}): there exists some positive constant $\Lambda$ such that 
\begin{equation}\label{e:Critical_DI_bound}
    \sup_{Q_{2r}(x,0) \subset Q_2} \text{log}_4 \frac{\fint_{ Q_{2r}(x,0)} u^2}{\fint_{Q_{r}(x,0)} u^2}  \leq \Lambda,   
\end{equation}
where $Q_r(x,t)=\{(y,s): |x-y|<r, t-r^2<s\le t\}.$

First we define the rescaled map of $u$. 
\begin{definition}
For $(x, t) \in Q_{1}$ and $\ell \leq 1$ we define
\begin{equation}\label{e:Dtangentmap}
       u_{x, t; \ell}(y, s) := \frac{  u(x +\ell A_{x,t}(y), t + \ell^2 s)}{\Big( \fint_{Q_1} (u(x +\ell A_x(y), t + \ell^2 s))^2 dy ds\Big)^{1/2}}.\end{equation}
Here $A_{x,t}(y) = (\sqrt{a})^{ij} y_i e_j$ and $(\sqrt{a})^{ij}$ is the square root of the coefficients matrix $a^{ij}(x,t)$.
\end{definition}
\begin{remark}
Write $\Tilde{u} = u_{x, t; \ell}$. Then $\Tilde{u}$ satisfies the following rescaled equation
\begin{equation*}
    \partial_s \Tilde{u} = \partial_i(\Tilde{a}^{ij}(y,s) \partial_j \Tilde{u}) + \Tilde{b}^i(y,s) \partial_i \Tilde{u}+\Tilde{c}\Tilde{u},
\end{equation*}
where
\begin{equation*}
    \begin{split}
        \Tilde{a}(y,s) &= a(x,t)^{-1}  \cdot a(x +rA_{x,t}(y), t + \ell^2 s)  \\
        \Tilde{b}(y,s) &= \ell \cdot \sqrt{a}(x,t)^{-1} \cdot b(x  +\ell A_{x,t}(y), t + \ell^2 s)\\
        \Tc(y,s) &=  \ell^2 \cdot c(x  +\ell A_{x,t}(y), t + \ell^2 s)
    \end{split}
\end{equation*}
Hence we have $\Tilde{a}(0,0) = (\delta^{ij})$ and moreover
\begin{equation*}
     |\Tilde{a}^{ij}(Y) -\Tilde{a}^{ij}(Z) | \leq C(\lambda,\alpha) \ell^{\alpha} d(Y, Z)^{\alpha}, (1 + C(\lambda,\alpha) \ell^{\alpha})^{-1} \delta^{ij} \leq \Tilde{a}^{ij} \leq (1 + C(\lambda,\alpha)) {\ell}^{\alpha}) \delta^{ij},  |\tilde{b}^i| \leq  (1+\lambda)^{-1/2} {\ell}, |c| \le \lambda {\ell}^2.
\end{equation*}
Here $d(Y,Z)$ is the parabolic distance define by $d((x,t), (y,s)) = ( ||x-y||^2 + |s-t|)^{1/2}$. 
\end{remark}

Next we generalize the definition of frequency to general parabolic solution which is only locally defined. Let $u$ be a solution to (\ref{e:Parab_equa}) (\ref{e:assumption_without_c}) defined on $Q_{2R_0}$ with doubling assumption (\ref{e:Critical_DI_bound}). For any $(x_0,t_0) \in Q_{R_0}$ and $0 \le r \le R_0$, we define
\begin{align}
    E^{u, R_0}_{x_0, t_0} (r) &= 2 r^2 \int_{ \{t = t_0 - r^2 \} \cap B_{R_0} } |\nabla u |^2 G_{x_0, t_0} \\
    H^{u, R_0}_{x_0, t_0} (r) &= \int_{ \{t = t_0 - r^2  \} \cap B_{R_0}} u^2 G_{x_0, t_0}.
\end{align}

\begin{definition}{(Localized Frequency)}\label{d:localizedfrequency}
    For any $(x_0,t_0) \in Q_{R_0}$, $\ell \le 1$, let $\Tilde{u} = u_{x_0, t_0, {\ell}}$. Then we define the localized frequency for any $r>0$ as 
    \begin{equation}
        N_{x_0, t_0}^{u; R_0}(r \ell) := N_{0,0}^{\tilde{u},R_0/\ell}(r):=\frac{ E^{\Tilde{u}, R_0/\ell }_{0, 0}(r)  } {H^{\Tilde{u}, R_0/\ell}_{0, 0}(r)}.
    \end{equation}
\end{definition}
\begin{remark}
    One can easily check that $N_{x_0, t_0}^{u; R_0}(r)$ is well defined which is independent of $\ell$.
\end{remark}
\begin{remark}
When $a^{ij} = \delta^{ij}$, then $N^{u; \infty}_{x_0, t_0}$ coincides with the standard global frequency defined in Poon \cite{Poon}.
\end{remark}

\begin{remark}
Throughout this article, we will assume $R_0 = 1$ and omit this superscript. Further, when $(x_0, t_0) = (0, 0)$, the subscript would be omitted as well. For instance, we write  $N^{u}(r) \equiv N_{0,0}^{u; 1}(r)$ for simplicity. { By rescaling, we have
    \begin{equation}\label{e:Frequency_Scaling}
         N_{x_0, t_0}^{u} ( r \ell) = \frac{2 \int_{ \{t = - 1 \} \cap B_{1/(r\ell ) }} |\nabla u_{x_0, t_0;r\ell} |^2 G_{0, 0} } {\int_{ \{t = - 1  \} \cap B_{1/(r\ell)}} u^2_{x_0, t_0;r\ell} G_{0, 0}}  = N^{u_{x_0,t_0;\ell};1 / \ell}(r) =  N^{u_{x_0,t_0;r\ell};1 / (r\ell)}(1)  .
    \end{equation}
    }
\end{remark}

\subsection{Almost Monotonicity for Localized Frequency}
In this section, we prove that the localized frequency function is almost non-decreasing. First we introduce an important approximation Lemma. 

\begin{lemma}{(Caloric approximation)}\label{l:Caloric_Approx}
Let $u$ be a solution to (\ref{e:Parab_equa}) (\ref{e:assumption_without_c}) on $Q_2$ with doubling assumption (\ref{e:Critical_DI_bound}). Let $\epsilon \in (0, 1/10)$. There exists $\ell_0 = C(n, \lambda, \alpha, \Lambda, \epsilon) $ such that for any $x\in B_1$ and $\ell \leq \ell_0$, there exists an ancient caloric polynomial $h:\RR^n \times (-\infty, 0] \to \RR$ of order at most $C(n,\lambda, \alpha) \Lambda$, such that 
\begin{equation}
    \fint_{Q_1} h^2 = 1, \quad \text{ and } \quad  ||h - u_{x,0;\ell} ||_{C^{1;1}(Q_{1/\epsilon})} \leq \epsilon.
\end{equation}
\end{lemma}

\begin{proof}
%In fact, in order to get an explicit dependence of $r_0$ on $\Lambda$, we may solve the Dirichlet problem and track the interior estimates carefully. 

We prove the result by contradiction argument. Suppose the result be false. Consider $u_{x,0; 1/i}$ with $i \to \infty$. Since $\fint_{Q_1} u_{x,0; 1/i}^2 = 1$, we have $\fint_{Q_{2^k}} u_{x,0; 1/i}^2 \le C(n,\lambda)^{k \Lambda}$ for any $i$ and any $k \le \log_2 i$ by doubling  assumption \ref{e:doublingat0}. Then by interior estimates, 
\begin{equation*}
    \sup_{i} ||u_{x,0;1/i}||_{C^{1+\alpha; 1+\alpha/2}(Q_{2^{k-1}})} \leq C(n,\lambda,\alpha)^{k \Lambda}.
\end{equation*}

For each $k$, there exists some $h_k$ such that $||u_{x,0;1/i} - h_k ||_{C^{1;1}(Q_{2^{k-1}})} \to 0$
by passing to a subsequence. Moreover each $h_k$ is a caloric function according to the estimates of $\Tilde{a}^{ij}$. By taking $k \to \infty$, we have $h_k \to h$ in $C^{1;1}$ on compact sets with $h$ satisfying the polynomial growth condition. {Then we can conclude by Lemma \ref{l:polynomialcalaric} that $h$ is indeed a caloric polynomial of order $C(n,\lambda, \alpha) \Lambda$.} The contradiction arises. 
\end{proof}  

Next we discuss several applications of this approximation lemma. The first is the almost monotonicity of the weighted integral in each time slice. 

Recall that $H^h(r) = \int_{t= -r^2} h^2 G_{0,0}$ is non-decreasing in $r$ for caloric function $h$. By the convergence, we can easily see that the localized version is also almost non-deceasing in $r$. As it is used frequently, for simplicity we define $H^{\Tilde{u}}$ with $\Tilde{u} = u_{x, 0; \ell}$ as
\begin{equation}
    H_{0,0}^{\Tilde{u}}(r) =\int_{\{t = -r^2\} \cap B_{1/\ell}} \Tilde{u}^2 G_{0,0}.
\end{equation}
Sometimes we omit the subscript $(0,0)$ if there is no confusion. 
\begin{lemma}\label{l:monotone_H}
    Let $u$ be a solution to (\ref{e:Parab_equa}) (\ref{e:assumption_without_c}) on $Q_2$ with doubling assumption (\ref{e:Critical_DI_bound}). Let $\epsilon \in (0, 1/10)$. There exists $\ell_0 = C(n, \lambda, \alpha, \Lambda, \epsilon) $ such that for any $x\in B_{1}$ and $\ell \leq \ell_0$, we have the following almost monotonicity for $H^{\Tilde{u}}$ where $\Tilde{u} = u_{x, 0; \ell}$
    \begin{equation}
        H^{\Tilde{u}}(r_1) \le H^{\Tilde{u}}(r_2) + \epsilon \quad \text{ for any } \quad 0 \le r_1 \le r_2 \le 1. 
    \end{equation}
\end{lemma}
\begin{remark}\label{r:monotoneofH}
    We can see from the proof below that $H^{\tilde{u}}_{(x_0,0)}$ is also almost monotone for any $(x_0,0)\in Q_R$ if $\ell\le \ell_0(n,\lambda,\alpha,\Lambda,\epsilon,R).$ This will be used in the proof of Lemma \ref{l:Equiv_Normalizaitons}. 
\end{remark}
\begin{proof}
    For any $\epsilon'>0$, by Lemma \ref{l:Caloric_Approx}, if $\ell\le \ell(n,\Lambda,\lambda,\alpha,\epsilon')$, there exists a caloric polynomial $h$ with order $C(n,\lambda, \alpha)\Lambda$ such that $\Tilde{u} = u_{x, 0; \ell}$ satisfies 
    \begin{equation}
   ||h - u_{x,0;\ell} ||_{C^{1;1}(Q_{1/\epsilon'})} \leq \epsilon'.
\end{equation}
Since $h$ is a polynomial but $G_{0,0}$ has exponential decay, for $\ell\le \ell(n,\Lambda,\lambda,\alpha,\epsilon')$ sufficiently small and any $r \in [0,1]$ we have 
\begin{equation}
    \Bigg|\int_{\{t = -r^2 \}} h^2 G_{0,0} - \int_{\{t = -r^2\} \cap B_{1/\ell}} h^2 G_{0,0} \Bigg| \le \epsilon'.
\end{equation}

Also, by the convergence and polynomial growth of $\tilde{u}$ and $h$, for any $r\in [0,1]$ we have 
\begin{equation*}
\Bigg| \int_{\{t = -r^2\} \cap B_{1/\ell}} h^2 G_{0,0} - \int_{\{t = -r^2\} \cap B_{1/\ell}} \Tilde{u}^2 G_{0,0} \Bigg| \le  \int_{\{t = -r^2\} \cap B_{1/\epsilon'}} |h^2-\tilde{u}^2| G_{0,0} +\int_{\{t = -r^2\} \cap A_{1/\epsilon',1/\ell}} |h^2-\Tilde{u}^2| G_{0,0}  \le C(n,\lambda,\Lambda)\epsilon'.
\end{equation*}

Recall that $(H^{h, \infty})'(r) \ge 0$. Then for any $0 \le r_1 \le r_2 \le 1$, we have 
\begin{equation*}
    H^{\Tilde{u}}(r_1) \le H^{h, \infty}(r_1) + C \epsilon' \le H^{h, \infty}(r_2) + 2 C \epsilon' \le H^{\Tilde{u}}(r_2)+ 3 C \epsilon'. 
\end{equation*}

The proof is finished by taking $\epsilon'=\epsilon'(\epsilon,n,\alpha,\lambda,\Lambda)$ for any given $\epsilon>0$.
\end{proof}

Before the proof of the almost monotonicity for frequency, in the next Lemma we establish the frequency closeness in the interval away from 0, say $[\delta, 1]$. 

\begin{lemma}\label{l:Frequency_Close_away_from_0}
Let $u$ be a solution to (\ref{e:Parab_equa}) (\ref{e:assumption_without_c}) on $Q_2$ with doubling assumption (\ref{e:Critical_DI_bound}). Let $\epsilon, \delta \in (0, 1/10)$. There exists $\ell_0 = C(\lambda, \alpha, \Lambda, \epsilon, \delta) $ such that for any $x\in B_{1}$ and $\ell \leq \ell_0$, there exists an ancient caloric polynomial $h:\RR^n \times (-\infty, 0] \to \RR$, of order at most $C(n,\lambda, \alpha) \Lambda$, such that for any $r \in [\delta, 1]$
\begin{equation}
    |N^{\Tilde{u}; 1/\ell}(r) - N^{h;\infty}(r)| \leq \epsilon,
\end{equation}
where $\Tilde{u} = u_{x, 0; \ell}$.
\end{lemma}
\begin{proof}
 For any $\epsilon'>0$, by Lemma \ref{l:Caloric_Approx}, if $\ell\le \ell(n,\Lambda,\lambda,\alpha,\epsilon')$, there exists a caloric polynomial $h$ with order $C(n,\lambda, \alpha)\Lambda$ such that $\Tilde{u} = u_{x, 0; \ell}$ satisfies 
    \begin{equation}
     \fint_{Q_1} h^2 = 1, \quad \text{ and } \quad   ||h - u_{x,0;\ell} ||_{C^{1,1}(Q_{1/\epsilon'})} \leq \epsilon'.
   \end{equation}
Thus, given $\delta>0$, by the monotonicity of $H^h(r) = \int_{t= -r^2} h^2 G_{0,0}$ for caloric function $h$ we can get for any $r\in [\delta, 1]$ that
\begin{equation}\label{e:Caloric_Appro_1}
    \int_{\{t = -r^2\}} h^2 G_{0,0} \geq c( \lambda, \alpha, \Lambda,\delta)>0.
\end{equation}

Since $h$ is a polynomial but $G_{0,0}$ has exponential decay, for $\ell\le \ell(n,\Lambda,\lambda,\alpha,\epsilon')$ sufficiently small and any $r \in [0, 1]$ we have 
\begin{equation}\label{e:Caloric_Appro_2}
\begin{split}
    \Bigg|\int_{\{t = -r^2\}} |\nabla h|^2 G_{0,0} &- \int_{\{t = -r^2\} \cap B_{1/\ell}} |\nabla h|^2 G_{0,0} \Bigg| \le \epsilon' \\
    \Bigg|\int_{\{t = -r^2\}} h^2 G_{0,0} &- \int_{\{t = -r^2\} \cap B_{1/\ell}} h^2 G_{0,0} \Bigg| \le \epsilon'.
\end{split}
\end{equation}

Also, by the approximating estimate and polynomial growth of $\tilde{u}$ and $h$, for any $r\in [0, 1]$ we have 
\begin{equation}\label{e:Caloric_Appro_3}
\begin{split}
    \Bigg| \int_{\{t = -r^2\} \cap B_{1/\ell}} |\nabla h|^2 G_{0,0} &- \int_{\{t = -r^2\} \cap B_{1/\ell}} |\nabla \Tilde{u}|^2 G_{0,0} \Bigg| \le \Psi(\epsilon' |~n,\lambda,\alpha,\Lambda)  \\
    \Bigg| \int_{\{t =-r^2\} \cap B_{1/\ell}} h^2 G_{0,0} &- \int_{\{t = -r^2\} \cap B_{1/\ell}} \Tilde{u}^2 G_{0,0} \Bigg| \le \Psi(\epsilon' |~n,\lambda,\alpha,\Lambda),
\end{split}
\end{equation}
The notation $\Psi(\epsilon' |~n,\lambda,\alpha,\Lambda)$ satisfies $\lim_{\epsilon'\to 0}\Psi(\epsilon' |~n,\lambda,\alpha,\Lambda)=0$ for any fixed $n,\lambda,\alpha,\Lambda$.

Moreover, since $h$ is a caloric polynomial of order at most $C(n,\lambda,\alpha)\Lambda$, by Corollary \ref{c:Frequency_bound_by_order_Poly} we have for $r\le 1$ that
\begin{equation}\label{e:Caloric_Appro_4}
   N^{h,\infty}(r)= \frac{2r^2\int_{\{t = -r^2\} } |\nabla h|^2 G_{0,0}}{\int_{\{t = -r^2\} } h^2 G_{0,0}} \leq C(n,\lambda,\alpha) \Lambda.
\end{equation}

Combining all these inequalities (\ref{e:Caloric_Appro_1}) (\ref{e:Caloric_Appro_2}) (\ref{e:Caloric_Appro_3}) (\ref{e:Caloric_Appro_4}), we can prove for any $r\in [\delta,1]$ that
\begin{equation*}
\begin{split}
    & \quad \frac{1}{2}|N^{h; \infty}(r) - N^{\Tilde{u}; 1/{\ell}}(r)| \\
    &= \Bigg| \frac{r^2\int_{\{t = -r^2\} \cap B_{1/{\ell}}} |\nabla \Tilde{u}|^2 G_{0,0}}{ \int_{\{t = -1\} \cap B_{1/{\ell}}} \Tilde{u}^2 G_{0,0}}  - \frac{r^2\int_{\{t = -r^2\}} |\nabla h|^2 G_{0,0}}{\int_{\{t = -r^2\}} h^2 G_{0,0}} \Bigg| \\
    \\
    &= r^2\Bigg|\frac{ (\int_{\{t = -r^2\} \cap B_{1/{\ell}}} |\nabla \Tilde{u}|^2 G_{0,0} )(\int_{\{t = -r^2\}} h^2 G_{0,0} ) 
    - (\int_{\{t = -r^2\}} |\nabla h|^2 G_{0,0})(\int_{\{t = -r^2\} \cap B_{1/{\ell}}} { \Tilde{u}^2 G_{0,0})}}  {{( \int_{\{t = -r^2\} \cap B_{1/{\ell}}} \Tilde{u}^2 G_{0,0}})(\int_{\{t = -r^2\}} h^2 G_{0,0})  }     \Bigg|\\
    \\
    &\le r^2\Bigg| \frac{\int_{\{t = -r^2\} \cap B_{1/{\ell}}} |\nabla \Tilde{u}|^2 G_{0,0} - \int_{\{t = -r^2\}} |\nabla h|^2 G_{0,0}} {\int_{\{t = -r^2\} \cap B_{1/{\ell}}}  \Tilde{u}^2 G_{0,0}}\Bigg|  + 
    r^2\Bigg| \frac{(\int_{\{t = -r^2\}} |\nabla h|^2 G_{0,0})(\int_{\{t = -r^2\}} h^2 G_{0,0} - \int_{\{t = -r^2\} \cap B_{1/{\ell}}} \Tilde{u}^2 G_{0,0} ) }{( \int_{\{t = -r^2\} \cap B_{1/{\ell}}} \Tilde{u}^2 G_{0,0})(\int_{\{t = -r^2\}} h^2 G_{0,0})}
    \Bigg|\\
    \\
    &\le \Psi(\epsilon' |~n,\delta,\lambda,\alpha,\Lambda).
\end{split}
\end{equation*}
Therefore for any $\epsilon>0$ by fixing $\epsilon'$ we have
\begin{equation*}
    |N^{h; \infty}(r) - N^{\Tilde{u}; 1/\ell}(r)| \leq \Psi(\epsilon' |~n,\delta,\lambda,\alpha,\Lambda)\le \epsilon,
\end{equation*}
which finishes the whole proof.
\end{proof}

Now we are ready to prove the almost monotonicity formula for $u$. 
\begin{theorem}\label{t:Frequency_almost_monotone}
    Let $u$ be a solution to (\ref{e:Parab_equa}) (\ref{e:assumption_without_c}) on $Q_2$ with doubling assumption (\ref{e:Critical_DI_bound}). Let $\epsilon \in (0, 1/10)$. There exists $\ell_0 = C(\lambda, \alpha, \Lambda, \epsilon) $ such that the following almost monotonicity formula holds: for any $x\in B_1$ and any $r_1 \le r_2 \leq \ell_0$,
\begin{equation}\label{e:Almost_Monotone}
    N^{u}_{x,0}(r_1) \le N^{u}_{x,0}(r_2) + \epsilon.
\end{equation}
Moreover, if  $N^{u}_{x,0}(r)\le \epsilon$ with $r\le \ell_0$, then  $\sup_{Q_r(x,0)}|u(y,s)-a_0|\le C(n,\Lambda)\sqrt{\epsilon} |a_0|$ for some nonzero constant $a_0$.
%Moreover, if $u(x,t) = 0$, we have $N^{u}_{x,t}(r) \ge 1-\epsilon$ for any $r\le r_0$. 
\end{theorem}

\begin{proof}
For simplicity, we assume $(x, t) = (0, 0)$ and we write $u_{0,0;\ell} \equiv u_\ell$. Also we write $N^u(r) \equiv  N^{u}_{0,0}(r)$. 

Take any $\epsilon \in (0, 1/10)$. Fix $\delta$ as in Lemma \ref{l:Caloric_fre_drop} corresponding to $\epsilon/20$. Then by Lemma \ref{l:Frequency_Close_away_from_0}, there exists some $\ell_0(\lambda,\alpha,\Lambda,\epsilon)$ such that for any $\ell \le \ell_0$, there exists some $h_\ell$ such that for any $r \in [\delta, 1]$
\begin{equation}\label{e:Almost_Monotone_appro}
      |N^{u}(\ell r) - N^{h_\ell;\infty}(r) | =  |N^{u_{\ell}; 1/\ell}(r) - N^{h_\ell;\infty}(r) | \le \epsilon/100.
\end{equation}
Along with the monotonicity of $N^{h_\ell;\infty}$, (\ref{e:Almost_Monotone_appro}) imply that $N^u(r) \le N^u(\ell_0) + \epsilon$ for any $r \in [\delta \ell_0, \ell_0]$. 
Hence, to prove the theorem, it suffices to prove that $N^u(r \ell_0) \le N^u(\ell_0) + \epsilon$ for any $r \le \delta$. We choose $d$ to be the minimal integer such that $|N^u(\ell_0) - d| \le 1/2.$ Hence we have $d-1/2< N^u(\ell_0) \le d+1/2$.

\textbf{Claim: } If $N^u(\ell_0) \le d+ \epsilon/10$, then $N^u(r\ell_0) \le d+ \epsilon/2$ for any $r \le 1$. 

We first suppose the claim be true and prove the theorem. It suffices to consider the following three cases.

\textbf{Case 1: } Suppose $N^u(\ell_0) \in [d + \epsilon/10, d + 1/2]$. Applying Lemma \ref{l:Caloric_fre_drop} to $h_{\ell_0}$ and by (\ref{e:Almost_Monotone_appro}), we have $N^u(\delta \ell_0) \le d + \epsilon/10$. According to the Claim, for any $r\le 1$, we have $N^u(r \delta r_0) \le d + \epsilon /2 \le N^u(\ell_0) + \epsilon$. This finishes the proof of Case 1.

\textbf{Case 2: } Suppose $N^u(\ell_0) \in [d - \epsilon/10, d + \epsilon/10]$. By the claim, we have $N^u(r\ell_0) \le d +\epsilon/2 < N^u(\ell_0)+\epsilon$ for any $r\le 1$. 

\textbf{Case 3: } Suppose $N^u(\ell_0) \in [d - 1/2, d - \epsilon/10]$. Applying Lemma \ref{l:Caloric_fre_drop} to $h_{\ell_0}$ and by (\ref{e:Almost_Monotone_appro}), we have $N^u(\delta \ell_0) \le d - 1 +\epsilon/10$. By the claim we have $N^u(r \delta \ell_0) \le d- 1 +\epsilon/2 < N^u(\ell_0)$ for any $r \le 1$. This finishes the proof.

In the following we will prove the claim. 
Since $N^u(\ell_0) \le d+ \epsilon/10$, we can define
\begin{equation}
    r_1 \equiv \inf \{ s\geq 0 : N^u(r\ell_0) \leq d + \epsilon/3 \text{ for any } r\in [s, 1] .\}
\end{equation}
Noting that frequency closeness (\ref{e:Almost_Monotone_appro}) and monotonicity of $N^{h;\infty}$ we have that $r_1 \le \delta$. We assume $r_1 > 0$. Otherwise the claim holds trivially. Then we apply Lemma \ref{l:Frequency_Close_away_from_0} at the scale $r_1\ell_0$ to obtain $h_{r_1\ell_0}$ such that for any $r\in [\delta, 1]$ we have
\begin{equation}\label{e:Almost_Monotone_appro_2}
      |N^{u}(rr_1\ell_0) - N^{h_{r_1\ell_0};\infty}(r) | \le \epsilon/100.
\end{equation}
Together with the monotonicity of $N^{h_{r_1\ell_0}; \infty}$, this implies that $N^u(r\ell_0) \le N^u(r_1\ell_0) + \epsilon/50 < d + \epsilon/2$ for any $r\in [\delta r_1, r_1]$. Since $N^{h_{r_1\ell_0};\infty}(1) \le d + 3\epsilon/8 < d+1 -\epsilon/20$, applying Lemma \ref{l:Caloric_fre_drop} to $h_{r_1\ell_0}$ and by (\ref{e:Almost_Monotone_appro_2}) we have $ N^u(\delta r_1\ell_0) \le d + \epsilon/10$.

Now we can finish the proof by induction. Indeed, for each $k$ we could define 
\begin{equation}
    r_{k+1} \equiv \inf \{ s\geq 0 : N^u(r\ell_0) \leq d + \epsilon/3 \text{ for any } r\in [s, r_k] .\}
    \end{equation}
and by the same arguments as above we can prove that 
\begin{enumerate}
    \item $r_{k+1} \leq \delta r_k$.
    \item $N^u(r\ell_0) \leq d+\epsilon/2$ for any $r \in [\delta r_{k+1}, 1]$.
    \item $N^u(\delta r_{k+1}\ell_0) \leq d + \epsilon/10$.
\end{enumerate}
This finishes the proof of almost monotonicity.
Noting Corollary \ref{c:smallfrequency}, Lemma \ref{l:Frequency_Close_away_from_0} and Lemma \ref{l:Caloric_Approx}, one can easily deduce the estimate with $N^{u}_{x,t}(r) \le \epsilon$ by contradiction. 
% Note that for any caloric polynomial $P$, the frequency is no less than $1$ at the zero point of $P$. The the proof is finished by contradiction argument and the almost monotonicity. 
\end{proof}

As a corollary, we prove that the frequency of $u$ is only pinched near integers and drops at a definite rate when it is away from integers. 
\begin{lemma}\label{l:parabolic_frequency_near_Z}
     Let $u$ be a solution to (\ref{e:Parab_equa}) (\ref{e:assumption_without_c}) on $Q_2$ with doubling assumption (\ref{e:Critical_DI_bound}). Let $\epsilon \in (0, 1/10)$ and $x\in B_1$. There exist some $\ell_0(n,\lambda,\alpha,\Lambda,\epsilon)$ and some $\delta(n,\lambda,\alpha,\Lambda,\epsilon)$ such that:
     \begin{enumerate}
         \item If $|N^u_{x,0}(r_1) -  N^u_{x,0}(r_2)| \le \delta$ for some $10r_1 \le r_2 \le \ell_0$, then there exists some integer $d$ such that $|N^u_{x,0}(r) -d |\le \epsilon$ for any $r\in [r_2,r_1]$.
         \item If $N^u_{x,0}(r) \le d-\epsilon$ for some $r\le \ell_0$, then $N^u_{x,0}(\tilde{r}) \le d-1+ \epsilon$ for any $\tilde{r}\le \delta r$. 
     \end{enumerate}
\end{lemma}

\begin{proof}
    The first conclusion follows directly from Lemma \ref{l:Caloric_frequency_near_Z} and Lemma \ref{l:Frequency_Close_away_from_0}. The second one can be proved using the same arguments as in Lemma \ref{l:Caloric_fre_drop} combined with the almost monotonicity theorem \ref{t:Frequency_almost_monotone}.
\end{proof}

Another consequence of the almost monotonicity is the following finitely many non-pinch result. This will be used in the proof of neck region decomposition in section 6. 
\begin{corollary}\label{c:finitemanypinch}
    Let $u$ be a solution to (\ref{e:Parab_equa}) (\ref{e:assumption_without_c}) on $Q_2$ with doubling assumption (\ref{e:Critical_DI_bound}). For any $\epsilon,\delta>0$ and $x\in B_1$
if $r\le \ell_0 (\lambda, \alpha, \Lambda, \epsilon,\delta) $ and $r_i=r \epsilon^i$, then there exists at most $K(n,\delta,\lambda,\alpha,\Lambda)$ many $\{i\ge 0\}$ such that 
\begin{align}
    |  N^{u}_{x,0}(r_i)-N^{u}_{x,0}(r_{i+1})|\ge \delta.
\end{align}
\end{corollary}
\begin{proof}
    By Theorem \ref{t:Frequency_almost_monotone} if $r\le \ell_0 (n,\lambda, \alpha, \Lambda, \delta) $ we have for all $r_i=r\epsilon^i$ that
    \begin{align}
        N^{u}_{x,0}(r_{i+1})\le N^{u}_{x,0}(r_i)+\delta^2
    \end{align}
    Denote $I=\{i\ge 0:  |  N^{u}_{x,0}(r_i)-N^{u}_{x,0}(r_{i+1})|\ge \delta\}=\{0\le i_1<i_2<i_3<\cdots\}$. Thus 
    \begin{align}
        N^u(x,0)(r_{i_1+1})\le N^{u}_{x,0}(r_{i_1})-\delta
    \end{align}
    By the almost monotonicity of the localized frequency we get 
    \begin{align}
        N^u_{x,0}(r_{i_2})\le N^u_{x,0}(r_{i_1+1})+\delta^2\le N^{u}_{x,0}(r_{i_1})-\delta+\delta^2
    \end{align}
    Since $|  N^{u}_{x,0}(r_{i_2})-N^{u}_{x,0}(r_{i_2+1})|\ge \delta$, we get 
    \begin{align}
        N^{u}_{x,0}(r_{i_3})\le N^{u}_{x,0}(r_{i_2+1})+\delta^2\le N^{u}_{x,0}(r_{i_2})-\delta+\delta^2 \le N^{u}_{x,0}(r_{i_1})-2(\delta-\delta^2).
    \end{align}
    By iteration we get 
        \begin{align}
        N^{u}_{x,0}(r_{i_k})\le  N^{u}_{x,0}(r_{i_1})-(k-1)(\delta-\delta^2).
    \end{align}
  Noting $0\le N_{x,0}^u(r)\le C(n,\lambda,\Lambda,\alpha)$, we  conclude that $\# I\le K(n,\delta,\lambda,\alpha,\Lambda)$. Hence we complete the proof. 
\end{proof}

\subsection{Frequency and Doubling Index}
In this subsection, we establish the equivalence of the localized frequency and the doubling index. Thus the doubling index is also almost monotone. Recall that we define $H^{\Tilde{u}}$ with $\Tilde{u} = u_{x, 0; \ell}$ as
\begin{equation}
    H^{\Tilde{u}}(r) = \int_{\{t = -r^2\} \cap B_{1/\ell}} \Tilde{u}^2 G_{0,0}.
\end{equation}

\begin{definition}
We define the localized doubling index $D^{\Tilde{u}}(r)$ with $\Tilde{u} = u_{x,0;\ell}$ as
    \begin{equation}
        D_{0,0}^{\Tilde{u}}(r) \equiv \log_4 \frac{H^{\Tilde{u}}(2r)}{H^{\Tilde{u}}(r)}.
    \end{equation}
\end{definition}
Sometimes we omit the subscript $(0,0)$ and write it as $D^{\Tilde{u}}(r)$ if there is no confusion.
Like the frequency function $N$, \eqref{e:Frequency_Scaling}, the doubling index also has similar scaling property
\begin{equation}\label{e:DI_Scaling}
    D^{u_{x,0;\ell}}(r) = D^{u_{x,0;r\ell}}(1).
\end{equation}

Next we can conclude that the doubling index is also almost monotone.

\begin{theorem}\label{t:DI_Almost_Monotone}
    Let $u$ be a solution to (\ref{e:Parab_equa}) (\ref{e:assumption_without_c}) on $Q_2$ with doubling assumption (\ref{e:Critical_DI_bound}). Let $\epsilon \in (0, 1/10)$. There exists $\ell_0 = C(\lambda, \alpha, \Lambda, \epsilon) $ such that for any $\ell \le \ell_0$ we have for any $r\le 1$ that 
    \begin{equation}
        N^{\Tilde{u},1/\ell}(r) - \epsilon \le D^{\Tilde{u}}(r) \le N^{\Tilde{u},1/\ell}(2r) + \epsilon.
    \end{equation}
    Moreover, the following almost monotonicity formula holds: for any $x\in B_1$ and any $r_1 \le r_2 \leq 1/2$,
\begin{equation}
    D^{\Tilde{u}}(r_1) \le D^{\Tilde{u}}(r_2) + \epsilon.
\end{equation}
 If $D^{\Tilde{u}}(r) \le \epsilon$ for some $r\le 1/2$ then there exists a nonzero constant $a_0$ such that  $\sup_{Q_r}|\tilde{u}-a_0|\le C(n,\Lambda)\sqrt{\epsilon} |a_0|.$
%Furthermore, if 
\end{theorem}

\begin{proof}
For $r=1$, we can prove the theorem by applying the contradiction argument and \eqref{e:Frequency_equv_DI_global}. The general case follows from the scaling property of $N$, see  (\ref{e:Frequency_Scaling}), and $D$, see (\ref{e:DI_Scaling}). Then the almost monotonicity follows directly from Theorem \ref{t:DI_Almost_Monotone}. 
\end{proof}

\subsection{Equivalence between different normalizations}
Recall that in the definition of tangent function we normalize such the function has unit $L^2$-norm over $Q_1$. However, sometimes we need to use other normalization. In the following lemma,  we will establish the equivalence between several different normalizations. 
\begin{lemma}\label{l:Equiv_Normalizaitons}
    Let $u$ be a solution to (\ref{e:Parab_equa}) (\ref{e:assumption_without_c}) on $Q_2$ with doubling assumption (\ref{e:Critical_DI_bound}). If $\ell\le \ell_0(n,\lambda,\alpha,\Lambda)$ is chosen small, then for any $x\in B_1$ and $\Tilde{u} = u_{x,0;\ell}$ the following holds
\begin{align}
    \fint_{Q_1}\tilde{u}^2\le    C_1(n,\lambda,\alpha,\Lambda)\int_{\{t=-1\} \cap B_{1/\ell}} \Tilde{u}^2 G_{0,0} \le C_2(n,\lambda,\alpha,\Lambda)\fint_{Q_1}\tilde{u}^2
\end{align}
\end{lemma}
\begin{comment}

\begin{remark}\label{r:equivalencecaloric}
    By Lemma \ref{l:Equiv_Normalizaitons}, if $h$ is a homogeneous caloric polynomial of order $d$, we have for any $0\le r\le 1$ that 
    \begin{align}
        \fint_{Q_1}|h|^2\le C(n,d)\fint_{\{t=-r^2\}\cap B_1}h^2\le C(n,d)\fint_{Q_1}|h|^2
    \end{align}
\end{remark}
\end{comment}
\begin{proof}
    Noting that  $\fint_{Q_1}\tilde{u}^2=1$, we only need to show that $\int_{\{t=-1\} \cap B_{1/\ell}} \Tilde{u}^2 G_{0,0}$ has uniform upper and lower bounds. 
    
To see the upper bound of $\int_{\{t=-1\} \cap B_{1/{\ell}}} \Tilde{u}^2 G_{0,0}$,  by the doubling assumption and parabolic estimate, we have point-wise estimate for $R\ge 1$ that 
\begin{align}
    \sup_{Q_R}|\tilde{u}|\le C(n,\lambda,\alpha,\Lambda) R^{C(n,\Lambda)}.
\end{align}
Since $G_{0,0}$ is exponentially decay, this implies 
\begin{align}
\int_{\{t=-1\} \cap B_{1/{\ell}}} \Tilde{u}^2 G_{0,0}\le C(n,\lambda,\alpha,\Lambda).
\end{align}
To see the lower bound of $L_0:=\int_{\{t=-1\} \cap B_{1/{\ell}}} \Tilde{u}^2 G_{0,0}$, for any $x_0\in B_1$ we have for any $R>1$ that
\begin{align}
   \int_{\{t=-1\} \cap B_{1/{\ell}}} \Tilde{u}^2 G_{x_0,0}&= \int_{\{t=-1\} \cap B_{1/{\ell}}} \Tilde{u}^2 (4\pi)^{-n/2}e^{-\frac{|x-x_0|^2}{4}}\\
   &\le \int_{\{t=-1\} \cap B_{1/{\ell}}\cap B_{R}(x_0)} \Tilde{u}^2 (4\pi)^{-n/2}e^{-\frac{|x-x_0|^2}{4}}+\int_{\{t=-1\} \cap B_{1/{\ell}}\setminus B_{R}(x_0)} \Tilde{u}^2 (4\pi)^{-n/2}e^{-\frac{|x-x_0|^2}{4}}\\
   &\le C(n,R)L_0+C(n,\lambda,\alpha,\Lambda)R^{C(n,\Lambda)}e^{-R^2/10}
\end{align}
By the monotonicity of $H$ in Lemma \ref{l:monotone_H},  for any $\epsilon>0$ if ${\ell}\le {\ell}_0(n,\lambda,\alpha,\Lambda,\epsilon)$  we have for any $s\le 1$ that
\begin{align}
     \int_{\{t=-s\} \cap B_{1/{\ell}}} \Tilde{u}^2 G_{x_0,0}\le  \int_{\{t=-1\} \cap B_{1/{\ell}}} \Tilde{u}^2 G_{x_0,0}+\epsilon\le C(n,R)L_0+C(n,\lambda,\alpha,\Lambda)R^{C(n,\Lambda)}e^{-R^2/10}+\epsilon.
\end{align}
In particular, 
\begin{align}
    \int_{\{t=-s\} \cap B_{\sqrt{s}}(x_0)}\Tilde{u}^2\le C(n)s^{n/2}\left(C(n,R)L_0+C(n,\lambda,\alpha,\Lambda)R^{C(n,\Lambda)}e^{-R^2/10}+\epsilon \right).
\end{align}
Since the above holds for any $x_0\in B_1$, using covering we conclude that for any $0\le s\le 1$
\begin{align}
   \int_{\{t=-s\} \cap B_1} \Tilde{u}^2\le C(n)\left(C(n,R)L_0+C(n,\lambda,\alpha,\Lambda)R^{C(n,\Lambda)}e^{-R^2/10}+\epsilon \right).
\end{align}
This implies 
\begin{align}
    1=\fint_{Q_1}\Tilde{u}^2\le C(n)\left(C(n,R)L_0+C(n,\lambda,\alpha,\Lambda)R^{C(n,\Lambda)}e^{-R^2/10}+\epsilon \right).
\end{align}
By fixing $R,\epsilon$ thus fixing $\ell_0=\ell_0(\lambda,\Lambda,n,\alpha)>0$ we get that $L_0\ge C(n,\Lambda,\lambda,\alpha)>0$. Hence we finish the whole proof.
\end{proof}

\section{Quantitative Uniqueness of Tangent Maps}\label{s:4_Uniq}

In this section we aim to prove the quantitative version of the uniqueness of tangent maps for parabolic solutions.  Based on the almost monotonicity of the localized frequency and parabolic estimate we can get quantitative estimate on each time slice. This is the key point to get nodal set estimate on time slice. The main theorem in this section is the following quantitative uniqueness of tangent maps. Some arguments are motivated by \cite{AV,Han98,HJ,NV}.
\begin{theorem}\label{t:mainuniqueness}
    Let $u$ be a solution to (\ref{e:Parab_equa}) (\ref{e:assumption_without_c}) on $Q_2$ with doubling assumption (\ref{e:Critical_DI_bound}) and $x \in B_1$.  For any $\epsilon>0$  there exist $\ell_0(n,\lambda,\alpha,\Lambda,\epsilon)$ and $\delta(n,\lambda,\alpha,\Lambda,\epsilon)$ such that if $|N^u_{x,0}(r_1) - N^u_{x,0}(r_2)| \le \delta$ for $10 \delta^{-1} r_2 \le r_1 \le \ell_0$, then there exists some integer $d$ and some homogeneous caloric polynomial $P_d$ with normalization $\fint_{Q_1} |P_d|^2 = 1$ such that for any $r\in [r_2, \delta r_1]$
    \begin{equation}\label{e:pointwiseUniqueness}
        \sup_{Q_1} |u_{x,0; r} - P_d| \le \epsilon.
    \end{equation}
%In particular, $u$ is uniformly $(0,\epsilon,x)$-symmetric at $t$-slice in $[r_2, \delta r_1]$. 
\end{theorem}

Recall that we say a function $u:B_1\to \mathbb{R}$ is $(k,\eta,r,x)$-symmetric if there exists a $k$-symmetric homogeneous  polynomial $P$ with $\fint_{B_1}|P|^2=1$ such that 
$
 \sup_{B_1}|u_{x,r}(y) - P(y)| \leq \eta,
$
where $u_{x,r}(y)=\frac{u(x+ry)}{\left(\fint_{B_1}|u(x+ry)|^2\right)^{1/2}}$.  We say $u$ is \textbf{uniformly $(k, \eta, x)$-symmetric in $[r_2, r_1]$ } if there exists a $k$-symmetric homogeneous polynomial $P$ with $\fint_{ B_1} |P|^2 = 1$ such that for any $r\in [r_2, r_1]$
\begin{equation}
    \sup_{B_1} |u_{x,r} - P| \leq \eta.
\end{equation}

By \eqref{e:pointwiseUniqueness}, we know that at time slice $s=0$, $\sup_{B_1}|u_{x,0;r}-P_d|(y,0)\le \epsilon$. By Lemma \ref{l:Equiv_Normalizaitons}, we have $C(n,d)^{-1}\le\fint_{B_1}|P_d|^2(x,0)dx\le C(n,d)$. Furthermore, by the definition of homogeneous polynomial, $P_d(x,0)$ is also a homogeneous polynomial in $x$. 
%Fix $t\in [-1,0]$. %Denote $\hat{u}(x)=u(x,t)$ %and $\hat{u}_r(x)=\frac{\hat{u}(rx)}{\fint_{B_1}\hat{u}(ry)dy}$
%Denote $\hat{N}^{u,t}(x,r)=N_{x,t}^u(r)$. Then by Theorem \ref{t:Frequency_almost_monotone} $\hat{N}^{u,t}(x,r)$ is almost monotone with respect to $r$. 
Hence as a consequence of Theorem \ref{t:mainuniqueness} we get
\begin{theorem}\label{t:quantiuniqueslice}
     Let $u$ be a solution to (\ref{e:Parab_equa}) (\ref{e:assumption_without_c}) on $Q_2$ with doubling assumption (\ref{e:Critical_DI_bound}). For any $\epsilon>0$ and $x\in B_1$ there exist $\ell_0(n,\lambda,\alpha,\Lambda,\epsilon)$ and $\delta(n,\lambda,\alpha,\Lambda,\epsilon)$ such that if $|{N}^{u}_{x,0}(r_1) - {N}^u_{x,0}(r_2)| \le \delta$ for $10 \delta^{-1} r_2 \le r_1 \le \ell_0$, 
     then there exists some integer $d$ and some homogeneous polynomial $\hat{P}_d:\mathbb{R}^n\to \mathbb{R}$ with normalization $\fint_{B_1} |\hat{P}_d|^2 = 1$ such that for any $r\in [r_2, \delta r_1]$
    \begin{equation}\label{e:pointwiseUniqueness1}
        \sup_{y\in B_1} |\hat{u}_{x,r}(y) - \hat{P}_d(y)|\le \epsilon.
    \end{equation}
    where $\hat{u}_{x,r}(y)= \frac{  u(x +r y, 0 )}{\Big( \fint_{B_1} |u(x +r y, 0 )|^2 dy\Big)^{1/2}}$. In particular, 
    $u(\cdot,0)$ is uniformly $(0,\epsilon,x)$-symmetric in $[r_2,\delta r_1]$.
\end{theorem}

We should remark that the homogeneous polynomial $\hat{P}_d(y)$ may not satisfy any equation. However, this is good enough to deduce the quantitative estimates as solution of elliptic PDEs in \cite{HJ}.

Theorem \ref{t:mainuniqueness} is part of the result of Theorem \ref{t:pinch_fre_symmetric}. Hence it suffices to prove Theorem \ref{t:pinch_fre_symmetric} in the following subsection. The proofs depends on several crucial estimates. 

\subsection{Growth estimates under doubling lower bound}
In this subsection we will prove a growth estimate for $u_{x,0; r}$ which is one key ingredient in the proof of quantitative uniqueness . 

Let $x \in B_1$ and choose $\ell<1$ be small. Write $\Tilde{u} = u_{x,0;\ell}$. Recall that $\Tilde{u}$ satisfies the equation
\begin{equation*}
    (\partial_t - \Delta )\Tu = \partial_i (\Ta^{ij} - \delta^{ij}) \partial_j \Tu + \Tb^i \partial_i \Tu + \Tc \Tu,
\end{equation*}
where we have $\Tilde{a}(0,0) = (\delta^{ij})$ and 
\begin{equation*}
     |\Tilde{a}^{ij}(Y) -\Tilde{a}^{ij}(Z) | \leq C(\lambda,\alpha) \ell^{\alpha} d(Y, Z)^{\alpha}, (1 + C(\lambda,\alpha) \ell^{\alpha})^{-1} \delta^{ij} \leq \Tilde{a}^{ij} \leq (1 + C(\lambda,\alpha)) \ell^{\alpha}) \delta^{ij},  |\tilde{b}^i| \leq  (1+\lambda)^{-1/2} \ell, |c| \le \lambda \ell^2.
\end{equation*}

Recall that 
\begin{equation*}
    H^{\Tilde{u}}(r) = \int_{\{t = -r^2\} \cap B_{1/\ell}} \Tilde{u}^2 G_{0,0} \quad \text{  and 
  }  \quad D_{0,0}^{\Tilde{u}}(r) \equiv \log_4 \frac{H^{\Tilde{u}}(2r)}{H^{\Tilde{u}}(r)}.
\end{equation*}

For any $\eta>1$, write the $\eta$-parabolic metric $d^{\eta}((x,t), (y,s)) \equiv \max\{\eta^{-1} ||x-y||, \sqrt{|t- s|} \}$. Define the backward $\eta$-parabolic ball as $Q^{\eta}_r(x_0, t_0) \equiv \{(x,t) : ||x-x_0||<\eta r, 0<t_0 - t < r^2\}$. Let us first prove the following growth estimate for $\tilde{u}$.
\begin{lemma}{\label{l:Lower_bound_imply_Grad_Estimates}}
       Let $u$ be a solution to (\ref{e:Parab_equa}) (\ref{e:assumption_without_c}) on $Q_2$ with doubling assumption (\ref{e:Critical_DI_bound}). For fixed $x\in B_1$ let $\Tilde{u} = u_{x,0;\ell}$. Assume $D_{0,0}^{\Tu}(r) \ge \gamma>0$ for any $r_0 \le r \le 1$. Then if $\ell\le \ell_0(n,\lambda,\alpha,\Lambda)$ we have the following estimates for all $r_0\le r\le 1$ and all $1\le \eta<\frac{1}{10\ell}$  that 
\begin{enumerate}
    \item \begin{equation}
    \sup_{Q_r^\eta} |\Tilde{u}|  \leq C(n,\lambda,\alpha,\Lambda,\eta)   ~r^{\gamma}.
    \end{equation}
    \item \begin{equation}
    \sup_{Q_r^\eta}  |\nabla \Tilde{u}| \leq C(n,\lambda,\alpha,\Lambda,\eta)  ~r^{\gamma-1}.
\end{equation}
\end{enumerate}
\end{lemma}
\begin{proof}
We will first prove the estimate for $r=r_0$. We will see that the proof works  for general $r>r_0$. Since $\fint_{Q_1} \Tu^2 = 1$,  by Lemma \ref{l:Equiv_Normalizaitons}, we have $\int_{\{t=-1\} \cap B_{1/\ell }} \Tu^2 G_{0,0} \le C(n,\lambda,\alpha, \Lambda)$ if $\ell\le \ell_0(n,\lambda,\alpha,\Lambda)$. Since $D_{0,0}^{\Tu}(r) \ge \gamma$ for any $r\in [r_0, 1]$, then $H_{0,0}^{\Tilde{u}}(r) \le Cr^{2\gamma}$ for $r\in [r_0,1]$. %For any point $(y,-s^2) \in \Omega^{101\eta/100}$, we have $G_{0,0}(y,-s^2) \ge e^{-101\eta/100}s^{-n}$. This proves that for any $s\ge s_0$
   % \begin{equation*}
    %    \int_{\Omega^{101\eta/100}_{-s_0^2, -s^2}} \Tu^2  \le C(n,\lambda,\alpha,\Lambda,\eta)  s^{2\gamma+n+2}. 
   % \end{equation*}
By Almost Monotonicity Theorem \ref{t:DI_Almost_Monotone} and taking $\epsilon_0 = \epsilon_0(n,\Lambda,\eta)>0$, for any $y\in B_{\eta r_0}(0)$, we have $D_{(y,0)}^{\Tu}(r) \ge \epsilon_0$ for all $ r\ge r_1\ge 0$ with some $r_1=r_1(y,0)\ge 0$,  and $\sup_{Q_{2\eta r_1}(y,0)}|\tilde{u}-a_0|\le |a_0|/100$ for some nonzero constant $a_0$. Let us now fix $y\in B_{\eta r_0}$. By Lemma \ref{l:Equiv_Normalizaitons} and doubling assumption of $\tilde{u}$, we can deduce that $H_{y,0}^{\tilde{u}}(r_0)\le C(n,\lambda,\Lambda,\eta) H_{0,0}^{\tilde{u}}(r_0)\le Cr_0^{2\gamma}.$
If $r_1(y,0)< r_0$, we have  for any $r_1\le r\le r_0$ that 
    \begin{equation}
        \int_{B_{r}(y)\times \{t=-r^2\}} \Tu^2  \le C(n,\lambda,\alpha,\Lambda,\eta)  r_0^{2\gamma-2\epsilon_0} r^{2\epsilon_0 + n}\le  C(n,\lambda,\alpha,\Lambda,\eta)  r_0^{2\gamma} r^{ n}
    \end{equation}
and for all $r\le r_1$
\begin{align}
    \int_{B_{r}(y)\times \{t=-r^2\}}\tilde{u}^2\le C(n,\lambda,\alpha,\Lambda,\eta)   r_0^{2\gamma-2\epsilon_0} r_1^{2\epsilon_0} r^n\le C(n,\lambda,\alpha,\Lambda,\eta)   r_0^{2\gamma}r^n. 
\end{align}
Thus if $r_1(y,0)<r_0$ for all $y\in B_{\eta r_0}$ then the above holds for all $y\in B_{\eta r_0}(0)$. By covering we have for all $0\le r\le r_0$ that 
\begin{align}
    \int_{B_{\eta r_0}(0)\times \{t=-r^2\}}\tilde{u}^2\le C(n,\lambda,\alpha,\Lambda,\eta)r_0^{2\gamma+n}
\end{align}
This implies
\begin{align}
 \int_{B_{\eta r_0}(0)\times [-r_0^2,0]}\tilde{u}^2\le  C(n,\lambda,\alpha,\Lambda,\eta)   r_0^{2\gamma+n+2}.
\end{align}
On the other hand, if $r_1(y,0)\ge r_0$ for some $y\in B_{\eta r_0}$, noting that $H_{y,0}^{\tilde{u}}(r_0)\le C(n,\lambda,\Lambda,\eta) H_{0,0}^{\tilde{u}}(r_0)$ and $\sup_{Q_{2\eta r_1}(y,0)}|\tilde{u}-a_0|\le |a_0|/100$, we get 
\begin{align}
  \int_{B_{\eta r_0}(0)\times [-r_0^2,0]}\tilde{u}^2\le  \int_{Q_{\eta r_0}(0,0)}\tilde{u}^2\le  \int_{Q_{2\eta r_0}(y,0)}\tilde{u}^2\le C(n,\lambda,\alpha,\Lambda,\eta)   r_0^{2\gamma+n+2}. 
\end{align}
Replacing $r_0$ by any $1>r>r_0$ and using the above argument, we can get 
\begin{align}
   \int_{B_{\eta r}(0)\times [-r^2,0]}\tilde{u}^2\le  C(n,\lambda,\alpha,\Lambda,\eta)   r^{2\gamma+n+2}.
\end{align}
The proof is now finished by standard parabolic estimates. 
\end{proof}

\subsection{Gradient estimates for Green's function expansion}

Next we study the expansion of the Green's function of the heat operator. Recall that for $s\le t$,\begin{equation*}
    G(x-y, t-s) = G(x,t;y,s) = \frac{1}{(4\pi (t-s))^{n/2}} \exp({-\frac{|x-y|^2}{4(t-s)}})
\end{equation*}
And $G(x,t;y,s) \equiv 0$ when $s > t$. It is direct that we have the following gradient estimates: for any $\mu, l \in \NN$, there exist constants $C_1(n,\mu, l)$, $c_2(n,\mu, l)$ that for $t>0$
\begin{equation}\label{e:Green_Gradient}
    |D_x^{\mu}D_t^{l} G(x,t)| \leq \frac{C_1}{|(x,t)|^{n+\mu+2l}} \exp(-\frac{|x|^2}{c_2 t}).
\end{equation}

By Taylor's theorem and simple rearrangement, for any $(y,s) \neq (0,0)$, we can write the expansion at $(x,t)=(0,0)$ in $Q_R$ with $R < |(y,s)|$ as
\begin{equation}\label{e:Green_Expansion}
    G(x,t;y,s) = \sum_{k=0}^d P_{y,s}^k(x,t) + R^d_{y,s}(x,t)
\end{equation}
where $P_{y,s}^k(x,t)$ are homogeneous caloric polynomials of degree $k$ with
\begin{equation}
    P_{y,s}^k(x,t) = \sum_{|\mu| + 2l = k} D^{\mu}_x D_t^{l} G(-y, -s) \frac{x^{\mu} t^l}{\mu!l!}.
\end{equation}
and the remainder term
\begin{equation}\label{e:Expansion_Remainder}
    R^d_{y,s}(x,t) =  \sum_{k=0}^d \sum_{\substack{|\mu| + l = k\\ |\mu|+2l > d}} D^{\mu}_x D_t^{l} G(-y, -s) \frac{x^{\mu} t^l}{\mu!l!} + \sum_{|\mu| + l = d+1} D^{\mu}_x D_t^{l} G(\kappa \cdot x-y, \theta \cdot t-s) \frac{x^{\mu} t^l}{\mu!l!}
\end{equation}
where $\kappa(x,t;y,s), \theta(x,t;y,s)$ are constants in $(0,1)$. 

We have the following estimates on the expansions which are crucial in the proof of uniqueness of tangent maps in the next section.

\begin{comment}
\begin{lemma}\label{l:Heat_Kernel_estimates}
    There exist constants $C_1, c_2$ depending on $n, d$ such that the following estimates hold for $s<0$
\begin{align}
    |P^k_{y,s}(x,t)| \le C_1 |(y,s)|^{-n-k} |(x,t)|^k \exp(\frac{|y|^2}{c_2s}), \quad |\nabla_y P^k_{y,s}(x,t)| \le C_1 |(y,s)|^{-n-k-1} |(x,t)|^\frac{k}{2} \exp(\frac{|y|^2}{c_2s}). 
\end{align}
Moreover, if $|(y,s)|_{\eta} \ge 2|(x,t)|_{\eta}$, then \eqref{e:Green_Expansion} holds with
\begin{align}
    |R^d_{y,s}(x,t)| &\le C  (\frac{|\kappa \cdot x-y|^2}{|\theta \cdot t-s|})^C G(\kappa \cdot x-y, \theta \cdot t-s) \sum_{j=d+1}^{2(d+1)} |(y,s)|^{-n-j}|(x,t)|^{j}, \\
    |\nabla_y R^d_{y,s}(x,t)| &\le  C  (\frac{|\kappa \cdot x-y|^2}{|\theta \cdot t-s|})^C G(\kappa \cdot x-y, \theta \cdot t-s) \sum_{j=d+1}^{2(d+1)} |(y,s)|^{-n-j-1}|(x,t)|^{j},
\end{align}
where $C$ depends on $n, d, \eta$ and $\theta$.
\end{lemma}    
\end{comment}

\begin{lemma}\label{l:Heat_Kernel_estimates}
    There exist constants $C$ depending on $n, d$ such that the following estimates hold:\\
For any $s<0$
\begin{align}
    |P^k_{y,s}(x,t)| \le C |(y,s)|^{-n-k} |(x,t)|^k, \quad |\nabla_y P^k_{y,s}(x,t)| \le C |(y,s)|^{-n-k-1} |(x,t)|^k. 
\end{align}
Moreover, if $|(y,s)|_{\eta} \ge 2|(x,t)|_{\eta}$, then \eqref{e:Green_Expansion} holds with
\begin{align}
    |R^d_{y,s}(x,t)| &\le C  \sum_{j=d+1}^{2(d+1)} |(y,s)|^{-n-j}|(x,t)|^{j}, \\
    |\nabla_y R^d_{y,s}(x,t)| &\le  C  \sum_{j=d+1}^{2(d+1)} |(y,s)|^{-n-j-1}|(x,t)|^{j},
\end{align}
where  $|(x,t)|_\eta=\max\{\eta^{-1} |x|,|t|\}.$
\end{lemma}

\begin{proof}
    The first two inequalities come from gradient estimate of Green's function \ref{e:Green_Gradient} directly. We prove the estimate for $R^d_{y,s}(x,t)$. And the estimate for the last inequality comes from gradient estimate \ref{e:Green_Gradient} and the Taylor expansion \ref{e:Green_Expansion}.

By gradient estimate \eqref{e:Green_Gradient}, we have
\begin{align*}
    &\quad\Big|\sum_{k=0}^d \sum_{\substack{|\mu| + l = k\\ |\mu|+2l > d}} D^{\mu}_x D_t^{l} G(-y, -s) \frac{x^{\mu} t^l}{\mu!l!}  \Big| \le C  \sum_{k=0}^d \sum_{\substack{|\mu| + l = k\\ |\mu|+2l > d}}    |(y,s)|^{-(n+\mu+2l)} |(x,t)|^{\mu + 2l} \exp(- \frac{|y|^2}{c_2 s}) \le C \sum_{j=d+1}^{2d} |(y,s)|^{-n-j}|(x,t)|^{j},
\end{align*}
where the second inequality follows from the substitution $j = \mu + 2l$. 

For the second term of \eqref{e:Expansion_Remainder}, observe that when $|(y,s)|_{\eta} \ge 2|(x,t)|_{\eta}$, we have
\begin{equation*}
    2|(\kappa \cdot x -y, \theta \cdot t -s)|_\eta \ge |(y,s)|_{\eta}.
\end{equation*}

Then by gradient estimate \eqref{e:Green_Gradient} we have
\begin{align*}
\big|\sum_{|\mu| + l = d+1} D^{\mu}_x D_t^{l} G(\kappa \cdot x-y, \theta \cdot t-s) \frac{x^{\mu} t^l}{\mu!l!} \big| \le C  \sum_{|\mu| + l = d+1} |(\kappa \cdot x-y, \theta \cdot t-s)|^{-n-\mu-2l} |(x,t)|^{\mu+2l}\le C \sum_{j=d+1}^{2(d+1)} |(y,s)|^{-n-j}|(x,t)|^{j}.
\end{align*}

This finishes the proof.
\end{proof}

\subsection{Quantitative Uniqueness of tangent maps}
In this subsection we will prove the quantitative uniqueness of tangent maps. The following Theorem \ref{t:Uniqueness} gives a uniqueness of tangent map with respect to a caloric function $h$. Based on the doubling estimate of $h$ in Theorem \ref{t:Uniqueness} and the quantitative uniqueness of heat equation in Theorem \ref{t:heat_equ_unique_tangent} we can prove the desired quantitative uniqueness Theorem \ref{t:pinch_fre_symmetric}.

\begin{theorem}\label{t:Uniqueness}
Let $u$ be a solution to (\ref{e:Parab_equa}) (\ref{e:assumption_without_c}) on $Q_2$ with doubling assumption (\ref{e:Critical_DI_bound}). Let $x\in B_1$ and $\epsilon \le \alpha/10$. Then for any $\eta>0$, there exists some $\ell_0(n,\lambda,\alpha,\Lambda,\epsilon, \eta)$ such that if $|N^u_{x,0}(r_1) - N^u_{x,0}(r_2)| \le \alpha/4$ for $10 r_2 \le r_1 \le \ell_0$, then there exists some caloric function $h$ defined on $Q_1^{\eta}$ such that for any  $r \in [r_2/r_1, 1]$ we have    \begin{equation}
        \sup_{(y,s)\in Q_r^\eta}|\Tilde{u}(y,s) - h(y,s)| \le \epsilon^2  (\fint_{Q_r} \Tu^2)^{1/2},
    \end{equation}
    where $\Tilde{u} = u_{x,0;r_1}$. Moreover, if $\epsilon \le \epsilon_0(n,\lambda,\alpha,\Lambda)$ then $h$ satisfies the doubling assumption at $(0,0)$ for all radius $r_2/r_1\le r\le 1$, i.e.
\begin{equation}
    \log_4\frac{\fint_{Q_{2r}} h^2}{\fint_{Q_{r}} h^2} \le C(n,\lambda)\Lambda.
\end{equation}
\end{theorem}

\begin{proof}
Let $\Tilde{r} \equiv r_2/r_1$ for simplicity. Since $|N^u_{x,0}(r_1) - N^u_{x,0}(r_2)| \le \alpha/4$ for $10 r_2 \le r_1 \le \ell_0$, by Theorem \ref{t:Frequency_almost_monotone} and \ref{t:DI_Almost_Monotone}, we have $D^{\Tilde{u}}(s) \in [\gamma, \gamma+\alpha/2]$ for any $s \in [\Tilde{r}, 1/2]$ and for some $\gamma$ if $\ell_0(n,\lambda,\alpha,\Lambda,\epsilon)$ small. Now we choose the largest integer $d$ such that $d< \gamma + \alpha$. Hence $d$ satisfies
\begin{equation*}
    \gamma + \alpha -d>0 \quad \text{ and } \quad \gamma + \alpha - (d+1) \le 0.
\end{equation*}

With this $d$, we construct the function $\phi$ as follows: For any $(x,t) \in B_{\eta} \times [-1,0]$, we define
\begin{equation}\label{e:unique_integral_phi}
    \phi(x, t) = \int_{Q^{\eta}_1} (G(x,t; y,s) - G(0,0; y,s)) f(y,s) dy ds - \sum_{k=1}^d \int_{Q^{\eta}_1 \setminus Q^{\eta}_{\tilde{r}}} P^k_{y,s}(x,t) f(y,s) dy ds
\end{equation}
where $P^k_{y,s}(x,t)$ is defined as in last section, see Lemma \ref{l:Heat_Kernel_estimates}. Here $f = (\partial_t - \Delta) \Tu$. Note that $f$ may not be well-defined pointwisely due to the low regularity of $u$. However we will address this issue using integral by parts in the proof. %However, by convolving the coefficients $a^{ij}$, $b^i$ and $c$ with a mollifier, we may assume that the parabolic equation has smooth coefficients and thus $f$ is well-defined. 
We will prove the uniform bound for $\phi$ provided that only the H\"older continuity of $a^{ij}$ is assumed. For simplicity, we assume the coefficients $b^i = c =0$ without loss of generality, as the terms involving $b^i$ or $c$ would have growth of higher order.

Set $r = |(x,t)|_{\eta} \equiv \max\{\eta^{-1} |x|, \sqrt{|t|}\}$. We have the following claim about the growth estimates for $\phi(x,t)$: 

\textbf{Claim :} If $r \ge \Tilde{r}$, we have
\begin{equation}
    |\phi(x,t)| \le C(n,\lambda,\alpha,\Lambda,\eta)r_1^{\alpha} |(x,t)|^{d+\alpha/2}.
\end{equation}
If $r  \le \Tilde{r}$, we have
\begin{equation}
    |\phi(x,t)| \le C(n,\lambda,\alpha,\Lambda,\eta)r_1^{\alpha} \Tilde{r}^{d+\alpha/2}.
\end{equation}

The claim implies $\phi$ is well-defined and finite. Noting that  $(\partial_t - \Delta)\phi(x,t) = f(x,t) $ we have the function $h \equiv \Tu - \phi $ is the caloric function we want defined in $Q^{\eta}_1$. In the following we prove the first estimate and the second one can proved in a similar way.

%Suppose $(x,t) \in Q^{\eta}_1 \setminus Q^{\eta}_{\Tilde{r}}$ and $r=|(x,t)|_\eta$. 
Suppose $(x,t) \in Q^{\eta}_1$ and $r=|(x,t)|_\eta$. 
If $r\ge \tilde{r}$, we can split the integral into the following three parts 
\begin{equation}
    \begin{split}
        I_1 &= \int_{Q^{\eta}_{2r}} (G(x,t; y,s) - G(0,0; y,s)) f(y,s) dy ds \\
        I_2 &= - \sum_{k=1}^d \int_{Q^{\eta}_{2r} \setminus Q^{\eta}_{\Tilde{r}}} P^k_{y,s}(x,t) f(y,s) dy ds\\
        I_3 &= \int_{Q^{\eta}_1 \setminus Q^{\eta}_{2r}} \Big( G(x,t; y,s) - \sum_{k=0}^d  P^k_{y,s}(x,t) \Big) f(y,s) dy ds.
        \end{split}
    \end{equation}
    If $r<\tilde{r}$, we can split the integral as above by replacing $r$ by $\tilde{r}$. We will only consider the case $r\ge \tilde{r}$. One will see that the same argument implies the estimate in the claim for $r<\tilde{r}$.

Note that \begin{equation*}
    f(y,s) = \partial_s \Tu - \Delta \Tu = (\partial_s \Tu - \Delta \Tu) - (\partial_s \Tilde{u} - \partial_i (\Ta^{ij} \partial_j \Tu)) = \partial_i \Big((\Ta^{ij} - \delta^{ij}) \partial_j \Tu\Big).
\end{equation*} 

Using integration by parts, we can write 
\begin{align*}
     I_1 &= \int_{Q^{\eta}_{2r}} G(x,t; y,s)  f(y,s) dy ds - \int_{Q^{\eta}_{2r}} G(0,0; y,s)  f(y,s) dy ds \\
     &= -\int_{Q^{\eta}_{2r}} \partial_{y_i} G(x,t; y,s)  (\Ta^{ij}(y,s) - \delta^{ij}) \partial_{y_j}\Tu(y,s) dy ds + \int_{\partial B_{2\eta r} \times [-(2r)^2, 0]}  G(x,t; y,s)  (\Ta^{ij}(y,s) - \delta^{ij}) \partial_{y_j}\Tu(y,s) \cdot n_i dy ds\\
     &+ \int_{Q^{\eta}_{2r}}  \partial_{y_i} G(0,0; y,s)  (\Ta^{ij}(y,s) - \delta^{ij}) \partial_{y_j}\Tu(y,s) dy ds + \int_{\partial B_{2 \eta r} \times [-(2r)^2, 0]}  G(0,0; y,s)  (\Ta^{ij}(y,s) - \delta^{ij}) \partial_{y_j}\Tu(y,s) \cdot n_i dy ds
\end{align*}
where $n_i$ is $i$-th component of the spatial outer normal.

Note that
\begin{align*}
     & \quad \int_{Q^{\eta}_{2r}} \partial_{y_i} G(x,t; y,s)  (\Ta^{ij}(y,s) - \delta^{ij}) \partial_{y_i}\Tu(y,s) dy ds  \\
     &\le C r_1^{\alpha} \int_{Q^{\eta}_{2r}} | \nabla_y G(x,t; y,s)|  |(y,s)|^{\alpha}    |\nabla \Tu(y,s)| dy ds\\
    &\le C r_1^{\alpha} r^{\gamma -1 +\alpha}  \int_{Q^{\eta}_{2r}} |(x-y, t-s)|_{\eta}^{-n-1} dy ds\\
    &\le C r_1^{\alpha} r^{\gamma -1 +\alpha}  \int_{Q^{\eta}_{3r}} |(y, s)|_{\eta}^{-n-1} dy ds\le C r_1^{\alpha} r^{\gamma +\alpha}.
\end{align*}

Next we deal with the boundary term. 
\begin{equation}
\label{e:uniq_boundary_term}
\begin{split}
     &\quad \int_{\partial B_{2\eta r} \times [-(2r)^2, 0]}  G(x,t; y,s)  (\Ta^{ij}(y,s) - \delta^{ij}) \partial_{y_j}\Tu(y,s) \cdot n_i dy ds \\  
    &\le C r_1^{\alpha} r^{\gamma + \alpha - 1} \int_{\partial B_{2\eta r} \times [-(2r)^2, 0]}   G(x,t; y,s) dy ds\le C r_1^{\alpha} r^{\gamma +  \alpha}.
\end{split}
\end{equation}

The remaining two terms of $I_1$ could be estimated in the same way. Therefore we prove that
    $|I_1| \le C r_1^{\alpha} r^{\gamma +  \alpha}.$
Next we deal with $I_2$. Integration by parts gives
\begin{align*}
    I_2 =& \sum_{k=1}^d \int_{Q^{\eta}_{2r} \setminus Q^{\eta}_{\Tilde{r}} } \partial_{y_i}  (P^k_{y,s}(x,t)) (\Ta^{ij}(y,s) - \delta^{ij}) \partial_{y_j}\Tu(y,s) dy ds \\
    &- \sum_{k=1}^d \int_{\partial B_{2\eta r} \times [-(2r)^2, 0] }  (P^k_{y,s}(x,t)) (\Ta^{ij}(y,s) - \delta^{ij}) \partial_{y_j}\Tu(y,s) \cdot n_i dy ds \\
    &+ \sum_{k=1}^d \int_{\partial B_{\eta \Tilde{r}} \times [-\Tilde{r}^2, 0] }  (P^k_{y,s}(x,t)) (\Ta^{ij}(y,s) - \delta^{ij}) \partial_{y_j}\Tu(y,s) \cdot n_i dy ds. 
\end{align*}

For the first term, we have
\begin{align*}
     \quad \sum_{k=1}^d &\int_{Q^{\eta}_{2r} \setminus Q^{\eta}_{\Tilde{r}} } \partial_{y_i}  (P^k_{y,s}(x,t)) (\Ta^{ij}(y,s) - \delta^{ij}) \partial_{y_j}\Tu(y,s) dy ds \\
    &\le C r_1^{\alpha} \sum_{k=1}^d \int_{Q^{\eta}_{2r} \setminus Q^{\eta}_{\Tilde{r}} }  |(y,s)|^{\alpha} | \nabla_y  (P^k_{y,s}(x,t)| |\nabla \Tu(y,s)| dy ds \\
    &\le C r_1^{\alpha} \sum_{k=1}^d \sum_{i\ge 0, 2^{-i+1}r\ge \tilde{r}} \int_{Q^{\eta}_{2^{1-i}r} \setminus Q^{\eta}_{2^{-i}r} } |(y,s)|^{\alpha} | \nabla_y  (P^k_{y,s}(x,t)| |\nabla \Tu(y,s)| dy ds \\
    &\le C r_1^{\alpha}  \sum_{k=1}^d r^{k} \sum_{i\ge 0, 2^{-i+1}r\ge \tilde{r}} \int_{Q^{\eta}_{2^{1-i}r} \setminus Q^{\eta}_{2^{-i}r} } |(y,s)|^{-n-k-1+\alpha}  |(y,s)|^{\gamma-1} dy ds \\
    &\le C r_1^{\alpha}  \sum_{k=1}^d r^{k} \sum_{i=0}^{\infty} (\frac{r}{2^i})^{\gamma +\alpha -k}\le C r_1^{\alpha} r^{\gamma+\alpha} \sum_{k=1}^d \sum_{i=0}^{\infty} (\frac{1}{2^{\gamma +\alpha -k}})^i \le C r_1^{\alpha} r^{\gamma+\alpha},
\end{align*}
where the last inequality follows from the fact that $\gamma + \alpha - d > 0$. The two boundary terms could be estimates similar as in \eqref{e:uniq_boundary_term}. 

Next we deal with $I_3$. We have
\begin{align*}
     I_3 =& -\int_{Q^{\eta}_1 \setminus Q^{\eta}_{2r}} \partial_{y_i} R^d_{y,s}(x,t)  (\Ta^{ij}(y,s) - \delta^{ij}) \partial_{y_j}\Tu(y,s) dy ds\\
     &+\int_{\partial B_{\eta} \times [-1,0]}  R^d_{y,s}(x,t)  (\Ta^{ij}(y,s) - \delta^{ij}) \partial_{y_j}\Tu(y,s) \cdot n_i dy ds\\
     &-\int_{\partial B_{2\eta r} \times [-(2r)^2, 0]}  R^d_{y,s}(x,t)  (\Ta^{ij}(y,s) - \delta^{ij}) \partial_{y_j}\Tu(y,s) \cdot n_i dy ds
\end{align*}

For the first term, by Lemma \ref{l:Heat_Kernel_estimates},
\begin{align*}
    & \quad \int_{Q^{\eta}_1 \setminus Q^{\eta}_{2r} }  \partial_{y_i} R^d_{y,s}(x,t)  (\Ta^{ij}(y,s) - \delta^{ij}) \partial_{y_j}\Tu(y,s) dy ds\\
    &\le C r_1^{\alpha}  \int_{Q^{\eta}_1 \setminus Q^{\eta}_{2r} } | \nabla_y R^d_{y,s}(x,t) | |(y,s)|^{\alpha} |\nabla \Tu(y,s)| dy ds\\ 
    &\le C r_1^{\alpha}  \sum_{i=1}^M \int_{Q^{\eta}_{2^{i+1}r} \setminus Q^{\eta}_{2^{i}r} }  |\nabla_y  (R^k_{y,s}(x,t)| |(y,s)|^{\alpha} | |\nabla \Tu(y,s)| dy ds \\
    &\le C r_1^{\alpha} \sum_{j=d+1}^{2(d+1)} r^j \sum_{i=1}^{M} \int_{Q^{\eta}_{2^{i+1}r} \setminus Q^{\eta}_{2^{i}r} }  |(y,s)|^{-n-j-1} \cdot |(y,s)|^{\alpha} | \cdot |(y,s)|^{\gamma -1} dy ds \\
    &\le C r_1^{\alpha} r^{\gamma + \alpha} \sum_{j=d+1}^{2(d+1)} \sum_{i=1}^{M} (2^{\gamma +\alpha -j})^i,
\end{align*}
where $M$ is an integer such that $2^{M} r \le 1 < 2^{M+1} r$.

Recall that $\gamma + \alpha -(d+1) \le 0$. If $\gamma +\alpha < d+1$, we have 
\begin{equation*}
    \sum_{j=d+1}^{2(d+1)} \sum_{i=1}^{M} (2^{\gamma +\alpha -j})^i \le C(\alpha,\Lambda).
\end{equation*}
If $\gamma+\alpha = d+1$, then we have
\begin{align*}
 \sum_{j=d+1}^{2(d+1)} \sum_{i=1}^{M} (2^{\gamma +\alpha -j})^i =M + \sum_{j=d+2}^{2(d+1)} \sum_{i=1}^{\infty} (2^{d+1 -j})^i \le \log_2 (1/r) + C(\alpha,\Lambda).
\end{align*}

To handle the boundary terms is similar. Therefore, we have the following estimate for $I_3$, 
\begin{equation*}
    I_3 \le C  r_1^{\alpha} r^{\gamma + \frac{\alpha}{2}}. 
\end{equation*}

Therefore, we have proved that 
\begin{equation}
    |\phi(x,t)| \le C r_1^{\alpha} r^{\gamma + \alpha/2}.
\end{equation}
For $|(x,t)|_{\eta} = r \le \Tilde{r}$ the same argument as above by replacing $r$ by $\tilde{r}$ we get 
\begin{equation}
    |\phi(x,t)| \le C r_1^{\alpha} \tilde{r}^{\gamma + \alpha/2}.
\end{equation}
Hence we finish the proof of the claim. 

Noting that $Q_r^\eta=\{(x,t): |(x,t)|_{\eta}\le r\}$, as a direct consequence of the claim, we get for any $r\in [\tilde{r},1]$ that
\begin{align}
    \sup_{Q_r^\eta}|\phi(x,t)|\le Cr_1^\alpha r^{\gamma+\alpha/2}.
\end{align}
On the other hand, Since $D^{\Tu}(r) \le \gamma+\alpha/2$ for $r\in [\Tilde{r}, 1/2]$ and by Lemma \ref{l:Equiv_Normalizaitons}, we have
\begin{equation*}
    H^{\Tilde{u}}(r) \ge C r^{2\gamma + \alpha}  H^{\Tilde{u}}(1) \ge C r^{2\gamma + \alpha}.
\end{equation*}
If we pick $\ell_0(n,\alpha,\lambda,\Lambda,\eta,\epsilon)$ small enough and note that $r_1\le \ell_0$, then for any $r\in [\tilde{r},1]$ we have 
\begin{equation}\label{e:unique_phi_final_estimate}
    \sup_{Q_r^\eta}|\phi(x,t)| \le  C r_1^{\alpha} r^{\gamma + \alpha/2} \le C r_1^{\alpha} (H^{\Tilde{u}}(r))^{1/2} \le \epsilon^2 (H^{\Tilde{u}}(r))^{1/2}. 
\end{equation}
Also, by Lemma \ref{l:Equiv_Normalizaitons} we have
\begin{equation*}
    \sup_{Q_r^\eta}|\phi(x,t)|\le C r_1^{\alpha} (H^{\Tilde{u}}(r))^{1/2} \le \epsilon^2 (\fint_{Q_r} \Tu^2)^{1/2}. 
\end{equation*}
Therefore, for any $r\in [\tilde{r},1]$ we have 
\begin{align*}
    \frac{\fint_{Q_{2r}} h^2}{\fint_{Q_{r}} h^2} \le 100 \frac{\fint_{Q_{2r}} \Tu^2 +\fint_{Q_{2r}} \phi^2  }{ \fint_{Q_{r}} \Tu^2 - \fint_{Q_{r}} \phi^2} \le 100 \frac{(1+C \epsilon^2) \fint_{Q_{2r}} \Tu^2}{(1-C \epsilon^2) \fint_{Q_{r}} \Tu^2}.
\end{align*}
This proves the doubling bound for $h$.
\end{proof}

Next we will prove the frequency closeness between $\Tu$ and the approximated caloric function $h$. Note that the definition of $H^{\Tilde{u}}$ requires the integral over the ball $B_{1/r_1}$. However we only have the approximation over the ball $B_{\eta}$. The following two lemmas show that there is no essential difference between these two integrals. 

\begin{lemma}\label{l:polynomial_small_tail}
    Let $P$ be a caloric polynomial of order at most $d$. Then for any $\epsilon$, there exists some $\eta_0(n,d, \epsilon)>0$ such that for any $\eta>\eta_0$ and $s>0$
    \begin{equation*}
        \int_{\{\RR^n \setminus B_{\eta s}\} \cap \{t = -s^2\}} P^2 G_{0,0}\le \epsilon  \int_{\{t = -s^2\}} P^2 G_{0,0}.   
    \end{equation*}
\end{lemma}

\begin{proof}
First we suppose $P$ be a caloric homogeneous polynomial of order $d$. By change of variable formula, we have
    \begin{equation*}
        \int_{\{\RR^n \setminus B_{\eta s}\} \cap \{t = -s^2\}} P^2 G_{0,0} = \frac{s^{2d}}{(4\pi)^{n/2}} \int_{\{\RR^n \setminus B_{\eta}\} \cap \{t = -1\}} P^2(x, -1) \exp(-\frac{|x|^2}{4}) dx.
    \end{equation*}
    Similarly, we have
    \begin{equation*}
        \int_{\{t = -s^2\}} P^2 G_{0,0} = \frac{s^{2d}}{(4\pi)^{n/2}} \int_{\{t=-1\}} P^2(x, -1)\exp(-\frac{|x|^2}{4}) dx.
    \end{equation*}
By normalization we may assume 
\begin{equation*}
    \int_{\{t=-1\}} P^2(x, -1)\exp(-\frac{|x|^2}{4}) dx =1. 
\end{equation*}
By lemma \ref{l:Equiv_Normalizaitons} and standard estimate, we have
\begin{equation*}
    |P(x,-1)| \le C(n,d) (|x|^d +1).
\end{equation*}
Therefore, by taking $\eta$ large enough we have
\begin{equation*}
    \int_{\{\RR^n \setminus B_{\eta}\} \cap \{t = -1\}} P^2(x, -1) \exp(-\frac{|x|^2}{4}) dx \le C'(n,d) \int_{\eta}^{\infty} r^{d+n-1} \exp(-\frac{r^2}{4})dr \le \epsilon.
\end{equation*}
This finishes the proof for caloric homogeneous polynomials. 

For general case, we can write $P = \sum_{i=0}^d P_i$ with each $P_i$ homogeneous polynomial of order $i$. Then by lemma \ref{l:homogeneous_Poly_orthogonal}, we have
\begin{equation*}
    \int_{\{t = -s^2\}} P^2 G_{0,0}   = \sum_i \int_{\{t = -s^2\}} P_i^2 G_{0,0}. 
\end{equation*}
And H\"older inequality gives
\begin{equation*}
    \int_{\{\RR^n \setminus B_{\eta s}\} \cap \{t = -s^2\}} P^2 G_{0,0} \le (d+1) \sum_i \int_{\{\RR^n \setminus B_{\eta s}\} \cap \{t = -s^2\}} P_i^2 G_{0,0}.
\end{equation*}
The proof is now finished using homogeneous case.
\end{proof}

The polynomial case will be used in the proof of the following general case.

\begin{lemma}\label{l:u_small_tail}
    Let $u$ be a solution to (\ref{e:Parab_equa}) (\ref{e:assumption_without_c}) on $Q_2$ with doubling assumption (\ref{e:Critical_DI_bound}). Let $x \in B_1$. Then for any $\epsilon>0$ there exists some $\ell_0(n,\lambda,\alpha,\Lambda,\epsilon)$ and $\eta_0(n,\lambda,\alpha,\Lambda,\epsilon)$ such that for any $\ell \le \ell_0$ and $\eta\ge \eta_0$ with $\Tilde{u}=u_{x,0;\ell}$ and any $0\le r\le 1$
    \begin{equation*}
        \int_{\{B_{1/\ell} \setminus B_{\eta r}\} \cap \{t = -r^2\}} \Tilde{u}^2 G_{0,0} \le \epsilon\int_{\{B_{1/\ell}\}  \cap \{t = -r^2\}} \Tilde{u}^2 G_{0,0} \equiv \epsilon H^{\Tilde{u}}(r).
    \end{equation*}
In particular, for $r \in (0,1]$ we have
\begin{equation*}
    \Big|\frac{\int_{\{ B_{\eta r}\} \cap \{t = -r^2\}} \Tilde{u}^2 G_{0,0}}{H^{\Tilde{u}}(r)} - 1 \Big| \le \epsilon.
\end{equation*}
\end{lemma}

\begin{proof}
    We prove this by contradiction. Suppose the result be false with some $\epsilon_0>0$. Then we can find a sequence $\Tilde{u}_i \equiv u_{x,0;1/i}$ with $i \to \infty$ and each $\Tilde{u}_i$ has polynomial growth of order $C(n,\lambda,\alpha)\Lambda$. By Lemma \ref{l:Caloric_Approx}, there exists some caloric polynomial $P$ of order $C(n,\lambda,\alpha)\Lambda$ such that $||\Tilde{u}_i -  P||_{C^{1;1}(\Omega)} \to 0$ in any compact $\Omega \subset Q_{1/i}$. 

    Take $\eta(n,\lambda,\alpha,\Lambda,\epsilon_0)$ as in Lemma \ref{l:polynomial_small_tail} with order $d = C(n,\lambda,\alpha)\Lambda$. Then we have
    \begin{equation*}
        \int_{\{\RR^n \setminus B_{\eta s}\} \cap \{t = -s^2\}} P^2 G_{0,0} \le \frac{\epsilon_0}{2} \int_{ \{t = -s^2\}} P^2 G_{0,0}.
    \end{equation*}
    Since $\Tilde{u}_i$ and $P$ has polynomial growth, then as $i \to \infty$ we have
    \begin{equation*}
        \int_{\{B_{i}\} \cap \{t = -s^2\}} \Tilde{u}^2_i G_{0,0} \to \int_{ \{t = -s^2\}} P^2 G_{0,0} \quad \text{ and } \quad   \int_{\{ B_{\eta s}\} \cap \{t = -s^2\}} \Tilde{u}^2_i G_{0,0} \to \int_{\{B_{\eta s}\} \cap \{t = -s^2\}} P^2 G_{0,0}.
    \end{equation*}
This proves that for large $i$
    \begin{equation*}
        \int_{\{B_i \setminus B_{\eta s}\} \cap \{t = -s^2\}} \Tilde{u}_i^2 G_{0,0} \le \epsilon_0 \int_{ \{t = -s^2\}} \Tilde{u}_i^2 G_{0,0}.
    \end{equation*}
The contradiction arises. 
\end{proof}

With the lemma above, we can prove that the closeness of frequency between $\Tilde{u}$ and $h$ in Theorem \ref{t:Uniqueness}.  

\begin{proposition}\label{p:frequency_close}
    Let $u$ be a solution to (\ref{e:Parab_equa}) (\ref{e:assumption_without_c}) on $Q_2$ with doubling assumption (\ref{e:Critical_DI_bound}). Let $x \in B_1$ and $\epsilon \le \alpha/10$. There exists some $\eta(n,\lambda,\alpha,\Lambda,\epsilon)$ and $\ell_0(n,\lambda,\alpha,\Lambda,\epsilon)$ such that if $|N^u_{x,0}(r_1) - N^u_{x,0}(r_2)| \le \alpha/4$ for $10 r_2 \le r_1 \le \ell_0$, let $h$ be the approximated caloric function as in Theorem \ref{t:Uniqueness}. Then we have the following frequency closeness
\begin{equation}
    |N^{\Tilde{u},1/r_1}(s) - N^{h,\eta}(s)|\le \epsilon. 
\end{equation}
    
\end{proposition}

\begin{proof}
    By  Theorem \ref{t:Uniqueness}, for any $r \in [\tilde{r},1]$ we have 
    \begin{equation*}
        \sup_{Q_r^\eta}|\Tilde{u}(y,s) - h(y,s)| \le \epsilon^2  (H^{\Tilde{u}}(r))^{1/2}.
    \end{equation*}
Hence for any $r \in [\tilde{r},1]$, we have
\begin{equation*}
    \int_{ B_{\eta r} \cap \{t = -r^2\}} |\Tilde{u} - h|^2 G_{0,0} \le \epsilon^4  H^{\Tilde{u}}(r)
\end{equation*}
and
\begin{equation*}
    \int_{B_{\eta r} \cap \{t = -r^2\}} |\Tilde{u} + h|^2 G_{0,0} \le  10 H^{\Tilde{u}}(r).
\end{equation*}
Therefore we have
\begin{equation*}
     \int_{ B_{\eta r} \cap \{t = -r^2\}} |\Tilde{u}^2 - h^2| G_{0,0} \le \int_{\{ B_{\eta r}\} \cap \{t = -r^2\}} \frac{|\Tilde{u} - h|^2}{\epsilon^2} G_{0,0} + \int_{ B_{\eta r} \cap \{t = -r^2\}} \frac{\epsilon^2 |\Tilde{u} + h|^2}{4} G_{0,0} \le 4\epsilon^2  H^{\Tilde{u}}(r).
\end{equation*}
This proves that
\begin{equation*}
    \int_{ B_{\eta r} \cap \{t = -r^2\}} \Tilde{u}^2  G_{0,0}  - 4\epsilon^2 H^{\Tilde{u}}(r) \le \int_{B_{\eta r} \cap \{t = -r^2\}} h^2 G_{0,0} \le \int_{B_{\eta r}\cap \{t = -r^2\}} \Tilde{u}^2 G_{0,0}  + 4\epsilon H^{\Tilde{u}}(r).
\end{equation*}
Now we take $\eta(n,\lambda,\alpha,\Lambda,\epsilon)$ as in Lemma \ref{l:u_small_tail}. We have
\begin{equation*}
    (1- 5\epsilon^2) H^{\Tilde{u}}(s) \le \int_{ B_{\eta s} \cap \{t = -s^2\}} h^2 G_{0,0} \le  (1 + 5\epsilon^2) H^{\Tilde{u}}(s).
\end{equation*}

One can obtain the same estimates for gradient terms $|\nabla h|$. This proves the closeness of frequency.
\end{proof}

Combining with Theorem \ref{t:heat_equ_unique_tangent}, we prove that $u$ is uniformly almost symmetric when frequency is pinched. 

\begin{theorem}\label{t:pinch_fre_symmetric}
    Let $u$ be a solution to (\ref{e:Parab_equa}) (\ref{e:assumption_without_c}) on $Q_2$ with doubling assumption (\ref{e:Critical_DI_bound}). Let $x \in B_1$ and $\epsilon \le c(n,\lambda,\alpha,\Lambda)$. Then there exist some $\ell_0(n,\lambda,\alpha,\Lambda,\epsilon)$ and $\delta(n,\lambda,\alpha,\Lambda,\epsilon)$ such that if $|N^u_{x,0}(r_1) - N^u_{x,0}(r_2)| \le \delta$ for $10 \delta^{-1} r_2 \le r_1 \le \ell_0$, there exists some integer $d$ such that: 
\begin{enumerate}
    \item For any $r\in [r_2, r_1]$ we have $|N^u_{x,0}(r) - d| \le \epsilon$.
    \item There exists some homogeneous caloric polynomial $P_d$ with normalization $\fint_{Q_1} |P_d|^2 = 1$ such that for any $r\in [r_2, \delta r_1]$
    \begin{equation*}
        \sup_{Q_1} |u_{x,0; r} - P_d| \le \epsilon.
    \end{equation*}
\end{enumerate}
%In particular, $u$ is uniformly $(0,\epsilon,x)$-symmetric at $t$-slice in $[r_2, \delta r_1]$. 
\end{theorem}

\begin{proof}
    The first conclusion comes from Lemma \ref{l:parabolic_frequency_near_Z}. We prove the second conclusion. Let $\Tu = u_{x,0;r_1}$. Denote $\hat{r}=r/r_1$. Let $h$ be the approximated caloric function as in Theorem \ref{t:Uniqueness}. By Proposition \ref{p:frequency_close}, we can apply Theorem \ref{t:heat_equ_unique_tangent} to find a homogeneous caloric polynomial $P_d$ of order $d$ such that for any $r\in [2r_2, \delta r_1]$ such that
    \begin{align}
        \sup_{Q_1}|h_{0,0;\Tr}-P_d|\le \epsilon.
    \end{align}
    
    On the other hand, by the scaling property of rescaled maps we have 
    \begin{equation*}
        u_{x,0;r} (y,s)= \frac{\Tu (\Tr z, \Tr^2 s)}{(\fint_{Q_{\Tr}} \Tu^2)^{1/2}} \quad \text{ and } \quad h_{0,0;r}(y,s) = \frac{h (\Tr y, \Tr^2 s)}{(\fint_{Q_{\Tr}} h^2)^{1/2}}.
    \end{equation*}
By Theorem \ref{t:Uniqueness}, we have
\begin{equation*}
    \Big| \fint_{Q_{\Tr}} \Tu^2 - \fint_{Q_{\Tr}} h^2 \Big| \le C\epsilon^2 \fint_{Q_{\Tr}} \Tu^2.
\end{equation*}
Hence 
\begin{align*}
    &\quad \sup_{Q_1}|u_{x,0;r} - h_{0,0;\Tr}| = \sup_{Q_{\Tr}} |\frac{\Tu}{(\fint_{Q_{\Tr}} \Tu^2)^{1/2}} - \frac{h}{(\fint_{Q_{\Tr}} h^2)^{1/2}}|\\
    &\le \sup_{Q_{\Tr}} |\frac{\Tu}{(\fint_{Q_{\Tr}} \Tu^2)^{1/2}} -  \frac{h}{(\fint_{Q_{\Tr}} \Tu^2)^{1/2}}| + | \frac{h}{(\fint_{Q_{\Tr}} \Tu^2)^{1/2}} - \frac{h}{(\fint_{Q_{\Tr}} h^2)^{1/2}}| \\
    &\le C \epsilon^2.
\end{align*}
By triangle inequality, now we have
\begin{equation*}
    \sup_{Q_1} |u_{x,0;r} - P_d|\le \sup_{Q_1} | u_{x,0;r}-h_{0,0;\Tr}|+ |h_{0,0;\Tr} - P_d| \le \epsilon.
\end{equation*}
The proof is now finished.
\end{proof}

\section{Cone-Splitting in a Fixed Time Slice}\label{s:5_Cone}

Recall that we say a continuous function $u$ is {uniformly $(k, \eta, x)$-symmetric in $[r_2, r_1]$ } if there exists a $k$-symmetric homogeneous polynomial $P$ with $\fint_{ B_1} |P|^2 = 1$ such that for any $r\in [r_2, r_1]$
\begin{equation}
    \sup_{B_1} |u_{x,r} - P| \leq \eta.
\end{equation}

The main purpose in this section is the following cone-splitting principle in the time $t$-slice.

\begin{theorem}[Quantitative Cone-Splitting]\label{t:Almost_Splitting}
    Fix $\tau < c(n,\lambda)$. Let $\eta > 0$ be given. Let $u$ be a solution to (\ref{e:Parab_equa}) (\ref{e:assumption_without_c}) on $Q_2$ with doubling assumption (\ref{e:Critical_DI_bound}). Let $x \in B_1$. Then there exist some $r_0(n,\lambda,\alpha,\Lambda,\eta)$ and $\epsilon(n,\lambda,\alpha,\Lambda,\eta)$ such that if 
\begin{enumerate}
        \item $u(\cdot,0)$ is uniformly $(k, \epsilon, x)$-symmetric in $[r_2, r_1]$ with respect to $k$-dimensional subspace $V$ with $10r_2 \le r_1 \le r_0$,
        \item $u(\cdot,0)$ is $(0,\epsilon, r, y)$-symmetric for some $y \in  B_{c(n,\lambda) r}(x) \setminus B_{\tau r }(x+V)$ and some $r \in [r_2, r_1]$,
\end{enumerate}
then $u(\cdot, 0)$ is uniformly $(k+1, \eta, x)$-symmetric in $[r_2, r_1]$.
\end{theorem}
\begin{proof}
Denote $f(y)=u(y,0)$ By the condition (1), there exists a $k$-symmetric homogeneous polynomial $P$ with $\fint_{B_1}|P|^2=1$ such that for all $r\in [r_2,r_1]$
    \begin{align}\label{e:unisymefP}
         \sup_{B_1} |f_{x,r} - P| \leq \epsilon,
    \end{align}
    and by condition (2), there exists a homogeneous polynomial $\tilde{P}$ with $\fint_{B_1}|\tilde{P}|^2=1$ such that 
     \begin{align}\label{e:symefP}
         \sup_{B_1} |f_{y,r} - \tilde{P}| \leq \epsilon,
    \end{align}
    where $y \in  B_{c(n,\lambda) r}(x) \setminus B_{\tau r }(x+V)$ and some $r \in [r_2, r_1]$. To prove the theorem, it suffices to show that there exist a $(k+1)$-symmetric homogeneous polynomial $\hat{P}$ with $\fint_{B_1}|\hat{P}|^2=1$ such that if $\epsilon\le \epsilon(n,\lambda,\alpha,\Lambda,\eta)$ we have
    \begin{align}\label{e:PhatPk1symm}
        \sup_{B_1}|P-\hat{P}|\le \eta.
    \end{align}
    To see this, by scaling, without loss of generality we assume (2) holds with $r=1$. By \eqref{e:unisymefP} and \eqref{e:symefP} we have for some $a,b>0$ that 
    \begin{align}
        \sup_{z\in B_1(x)\cap B_1(y)}|aP(z-x)-b\tilde{P}(z-y)|\le \epsilon (a+b).
    \end{align}
    In particular, for $\tilde{a}=a/(a+b)$ and $\tilde{b}=b/(a+b)$ we have
    \begin{align}
        \sup_{z\in B_1(x)\cap B_1(y)}|\tilde{a}P(z-x)-\tilde{b}\tilde{P}(z-y)|\le \epsilon
    \end{align}
    and $\tilde{a}+\tilde{b}=1$. By contradiction argument, to get \eqref{e:PhatPk1symm} it suffices to prove $P$ is $(k+1)$-symmetric when 
      \begin{align}\label{e:PtildePequal}
        \sup_{z\in B_1(x)\cap B_1(y)}|\tilde{a}P(z-x)-\tilde{b}\tilde{P}(z-y)|=0
    \end{align}
    and $y-x\notin V$ and $0<d(x,y)< 1$. Since $P$ and $\tilde{P}$ are polynomial, by \eqref{e:PtildePequal} we get that 
    \begin{align}
        \tilde{a}P(z-x)=\tilde{b}\tilde{P}(z-y)
    \end{align}
    Noting that $\tilde{a}+\tilde{b}=1$, we have $\tilde{a},\tilde{b}>0.$ Furthermore, by $0$-homogeneous of $P$ and $\tilde{P}$ it is easy to show that $P(z+(y-x))=P(z)$ for all $z$ which means $P$ is $(k+1)$-symmetric with respect to $\tilde{V}={\rm span}\{V, y-x\}$. Hence we finish the whole proof. 
\end{proof}

Next we define the set of points where the frequency of the parabolic function $u$ is pinched as
\begin{equation}
    \cV^{u}_{\delta, d, r} (x,0) \equiv \{ y \in B_{r}(x) : |N_{y, 0}^u(s) - d| \leq \delta \text { for any } s \in [ 10^{-2} r,  \delta^{-1} r] \}
\end{equation}

\begin{definition}[$(k, \tau)$-independent]
We say a subset $S \subset B_r \subset \RR^n$ is $(k, \tau)$-independent in $B_r$ if for any affine $(k-1)$-plane $L \subset \RR^n$ there exists some point $x \in S$ such that $d(x, L) \geq \tau r$.
\end{definition}

The following is the generalization of Theorem \ref{t:pinch_fre_symmetric}.

\begin{proposition}\label{p:k_indepen_k_symm}
    Let $u$ be a solution to (\ref{e:Parab_equa}) (\ref{e:assumption_without_c}) on $Q_2$ with doubling assumption (\ref{e:Critical_DI_bound}). Let $x \in B_1$ and $\epsilon \le \epsilon_0(n,\lambda,\alpha,\Lambda)$, $\tau>0$. Then there exist some $\ell_0(n,\lambda,\alpha,\Lambda,\epsilon,\tau)$ and $\delta(n,\lambda,\alpha,\Lambda,\epsilon,\tau)$ such that if $\cV^{u}_{\delta, d, r} (x,0)$ is $(k,\tau)$-independent in $B_r(x)$ for some $r\le \ell_0$, then $u(\cdot, 0)$ is uniformly $(k,\epsilon,y)$-symmetric in $[r/10, 10 r]$ for any $y \in \cV^{u}_{\delta, d, r} (x,0)$.
\end{proposition}

\begin{proof}
    This follows from Theorem \ref{t:quantiuniqueslice} and Quantitative Cone-splitting Theorem \ref{t:Almost_Splitting}. 
\end{proof}

\section{Neck Structure and Neck Decomposition}\label{s:6_NR}
In this section, we will introduce the neck region and prove the neck decomposition. This was first introduced in \cite{JN,NV19} to study the structure of Einstein manifolds and Yang-Mills equations. Such kind of technique turns out to be very useful to many interesting topics, see for instance \cite{BNS,CJN,NV24}. The main purpose in this section is the following theorem. Based on Proposition \ref{p:k_indepen_k_symm}, the proof of this result is quite standard now(see also \cite{HJ}).

\begin{theorem}[Neck Decomposition]\label{t:Neck_decompose}
    Let $u$ be a solution to (\ref{e:Parab_equa}) (\ref{e:assumption_without_c}) on $Q_2$ with doubling assumption (\ref{e:Critical_DI_bound}). For each $\eta >0$ and $\delta \leq \delta_0(n, \lambda, \alpha, \Lambda, \eta)$ we have
\begin{equation}
\begin{split}
    &B_1 \subset \bigcup_a ( \cN^a \cap B_{r_a}(x_a) \big) \cup \bigcup_b B_{r_b}(x_b) \cup \Big( S_0 \cup \bigcup_{a} \cC_{0, a} \Big) 
\end{split}
\end{equation}
such that 
\begin{itemize}
    \item $u(\cdot,0)$ is $(k+1, \eta, 2r_b, x_b)$-symmetric.
    \item $\cN^a\subset B_{r_a}(x_a)$ is $(d_a,k,\delta,\eta)$-neck region at $0$-slice for some integer $d_a\le C(n)\Lambda.$
    \item 
    $\sum_a r^k_a + \sum_b r^k_b + \cH^k\Big( S_0 \cup \bigcup_{a} \cC_{0, a} \Big)  \leq C(n, \lambda, \alpha, \Lambda, \epsilon, \eta).$
\item $ S_0 \cup \bigcup_{a} \cC_{0, a} $ is $k$-rectifiable and $\cH^k(S_0)=0$. 
\end{itemize}

Moreover, for any $\epsilon>0$, if $\eta < \eta(n, \lambda, \alpha, \Lambda, \epsilon)$ and $\delta < \delta(n, \lambda, \alpha, \Lambda, \eta, \epsilon)$, the quantitative stratum $\cS^k_{\epsilon}$ satisfies
\begin{equation}
    \cS^k_{\epsilon} \cap B_1 \subset \Big( S_0 \cup \bigcup_{a} \cC_{0, a} \Big) .
\end{equation}
\end{theorem}

\subsection{Neck Regions}
We construct the neck regions at $t$-slice. This is essentially equivalent to the neck regions in Euclidean space as defined in \cite{HJ}. Throughout this section, the time slice $t=0$ is fixed. 

\begin{definition}{(Neck Region)}
    Let $u$ be a solution to (\ref{e:Parab_equa}) (\ref{e:assumption_without_c}) on $Q_2$ with doubling assumption (\ref{e:Critical_DI_bound}). Fix $t=0$. Let $\cC \subset B_{r}$ a closed subset and $r_x: \cC \to \RR_{\ge 0}$ a radius function such that the closed balls \{$\bar{B}_{ r_x/5}(x)$\} are disjoint. The subset $\cN = B_r \setminus \Bar{B}_{r_x}(\cC)$ is called a \textbf{ $(d, k,  \delta, \eta)$-neck region at $0$-slice } if
\begin{enumerate}
    \item For any $x\in \cC$, we have $|N_{x,0}(s) - d| \leq \delta$ for any $s\in [r_x, \delta^{-1} r]$ with $r_x \leq 10^{-2}r$.
    \item For any $x\in \cC$, $u(\cdot, 0)$ is uniformly $(k,\delta,x)$-symmetric in $[r_x, 10r]$ with respect to $V_x$ but $u$ is not $(k+1,\eta,s,x)$-symmetric for any $s \in [r_x, 10r]$.  
    \item For any $B_{2s}(x) \subset B_{2r}$ with $s \geq r_x$, we have $x + V_x\subset \cup_{y\in \cC\cap B_s(x)}B_{10^{-10}(s+r_y)}(y)$. %and $\cC\cap B_s(x)\subset B_{\delta s}(x+V_x),$ 
\end{enumerate}
\end{definition}
\begin{remark}
    In the neck region definition we can assume $V_x$ is unique for all scale $s\in [r_x,10r]$, the reason is the quantitative uniqueness of tangent map in Theorem \ref{t:quantiuniqueslice}.
\end{remark}
\begin{remark}
    By Corollary \ref{c:finitemanypinch} and Cone-splitting Theorem \ref{t:Almost_Splitting}, for any $y\in \cN$ with $d(y,V_x)=s$ for some $V_x$, we have that $B_{\hat{s}}(y)$ is $(k+1,\eta)$-symmetric if $\delta\le \delta(\eta,n)$ for some $c(\eta)s\le \hat{s}\le s$. In particular, $y\notin \cS^k_{\eta}$
\end{remark}
%By quantitative uniqueness of tangent map the 
Denote $\cC_{+}=\{x\in \cC: r_x>0\}$ and $\cC_0=\{x\in \cC: r_x=0\}$. The following  neck structure theorem follows directly by the quantitative uniqueness of tangent map. 
\begin{theorem}[Neck Structure Theorem]\label{t:Neck_Struture}
    Let $\cN = B_{2r} \setminus \bar{B}_{r_x}(\cC)$ be a $(d,k, \delta, \eta)$-neck region at $0$-slice. For any $\epsilon>0$ if $\delta\le\delta(n,\eta,k,\epsilon,\Lambda,\lambda,\alpha)$ then 
    \begin{enumerate}
        \item Fix $x_0\in \cC$ and let $V=V_{x_0}$. Then the projection map $\pi:\cC\to V$ is a bilipschitz map, 
        \begin{align}
            (1-\epsilon)||x-y||\le ||\pi(x)-\pi(y)||\le ||x-y||
        \end{align}
        \item The measure bound holds 
        $$\sum_{x\in \cC_+} r_x^{k} + \cH^{k}({\cC_0}) \leq C(n) r^k$$
    \end{enumerate}
\end{theorem}

\begin{proof}
 To prove (1), noting that $\pi$ is a projection, the upper bound $||\pi(x)-\pi(y)||\le ||x-y||$ holds trivially. To see the lower bound, for any given $x,y\in \cC$ assume $s=d(x,y)$.
   Denote $\pi_x: \cC\to V_x$ is the projection to $V_x$. Since  $u(\cdot, 0)$ is $(k,\delta,s,x)$-symmetric with respect to $V_x$ and  $u$ is $(k,\delta,s,y)$-symmetric, but $u$ is not $(k+1,\eta,2s,x)$-symmetric,  for any $\epsilon$ by Theorem \ref{t:Almost_Splitting} if $\delta\le \delta(\epsilon,\eta)$ we have that $x-y\in B_{\epsilon s}(V_x)$. In particular we get 
   \begin{align}
       ||\pi_x(x-y)||\ge (1-\epsilon^2)||y-x||.
   \end{align}
   Similar argument implies $d_H(V_x\cap B_1,V_{x_0}\cap B_1)\le \epsilon^2$ if $\delta\le \delta(\epsilon,\eta)$ Hence we have 
   \begin{align}
       ||\pi_x(x-y)-\pi_{x_0}(x-y)||\le \epsilon^2 ||x-y||.
   \end{align}
   Combining the above estimates we get $||\pi(x-y)||\ge (1-\epsilon)||x-y||.$ Hence this proves (1).

   To prove (2), noting that $\{B_{r_x/5}(x):x\in \cC\}$ are pairwise disjoint, by (1) then $\{B_{r_x/10}(\pi(x)):x\in \cC\}$ are pairwise disjoint. Thus 
   \begin{align}
     \sum_{x\in \cC_+}r_x^k\le  C(n)\sum_{x\in \cC_+}\Vol_V(B_{r_x/10}(\pi(x))\le C(n)r^k,
   \end{align}
   where $\Vol_V$ is the $k$-Hausdorff measure in $V$.
   Since $\pi$ is bilipschitz, we have $\cH^k(\cC_0)\le C(n)\cH^k(\pi(\cC_0))\le C(n)r^k.$ Thus we get the upper bound estimate.% To see the lower bound, we notice that  $x + V_x\subset B_{\tau s}(\cC \cap B_{s}(x))$ for any $B_{2s}(x) \subset B_{2r}$ with $s \geq r_x$.
\end{proof}

\subsection{Neck Decomposition}

In this subsection we prove the decomposition theorem. %We will consider the decomposition at time $t=0$ slice. 

For each $x\in B_1$, we will decompose $B_r(x)$ into the balls in the following categories.
\begin{enumerate}
    \item [(a)] A ball $B_{r_a}(x_a)$ is associated with a $(d, k,\epsilon,\tau,\eta)$-neck region at $0$-slice $\cN_a = B_{2r_a}(x_a) \setminus \bar{B}_{r_{a,x}}(\cC_a)$. 
    \item [(b)] A ball $B_{r_b}(x_b)$ is such that $u$ is $(k+1, 2\eta, 2r_b, x_b)$-symmetric.
    \item [(c)] A ball $B_{r_c}(x_c)$ is not a $b$-ball and $\cV_{\delta,d,r_c}(x_c,0)$ is $(k, \tau)$-independent in $B_{r_c}(x_c)$.
    \item [(d)] A ball $B_{r_d}(x_d)$ is such that $\cV_{\delta,d,r_d}(x_d,0) \neq \emptyset$ and is not $(k, \tau)$-independent in $B_{r_d}(x_d)$.
    \item [(e)] A ball $B_{r_e}(x_e)$ is such that $\cV_{\delta,d,r_e}(x_e,0) = \emptyset$.
\end{enumerate}

Next we prove the decomposition for $c$-balls.

\begin{proposition}{(Covering of $c$-balls)}\label{p:c_ball_cover}
    Let $u$ be a solution to (\ref{e:Parab_equa}) (\ref{e:assumption_without_c}) on $Q_2$ with doubling assumption (\ref{e:Critical_DI_bound}). Let $\eta,\tau > 0$. Suppose $\delta \leq \delta_0(n, \lambda, \alpha, \Lambda, \eta,\tau)$ and $r \leq \ell_0(n,\lambda, \alpha, \Lambda, \epsilon, \eta,\tau)$. Suppose $\sup_{x \in B_r} N_{x,0}(s) \leq 
    d+\delta$ for any $s \leq C(n,\lambda)\delta^{-1} r$ and for some integer $d$. Furthermore we assume that $\cV_{\delta,d,r}(0,0)$ is $(k,\tau)$-independent in $B_r$ and that $B_{2r}$ is not $(k+1, 2\eta)$-symmetric. Then we can decompose
    \begin{equation}
     B_r \subset  \big(\cC_0 \cup \cN)  \cup \bigcup_{b \in B} B_{r_b}(x_{b}) \cup \bigcup_{d \in D} B_{r_d}(x_{d}) \cup \bigcup_{e \in E} B_{r_e}(x_{e}),
\end{equation}
where $\cN$ is a $(d, k, \epsilon, \eta)$-neck region at $0$-slice. Furthermore, we have the estimates
\begin{equation}
    \sum_{b \in B} r_b^k + \sum_{d \in D} r_d^k + \sum_{b \in E} r_e^k + \cH^k(\cC_0) \leq C(n) r^k.
\end{equation}
\end{proposition}

\begin{proof}
%For each point $x \in B_r$, we define 
%    \begin{equation}
%        r_x \equiv 10 \sup\{ 0 \leq s \leq \delta^{-1} r : N_{x,0}(s) \leq d - \delta/2 \}.
%    \end{equation}
%If no such $r_x$ exists, we set $r_x = 0$. Define $\cC_0 \equiv \{ x \in B_r : r_x = 0 .\} $ 
%By Almost Monotonicity Theorem \ref{t:Frequency_almost_monotone}, we have $x \in \cV_{\delta,d,r} $ if $r_x \leq 10^{-2} r$. 
Since $\cV_{\delta,d,r}$ is $(k,\tau)$-independent in $B_r$, for any $\epsilon>0$ if $\delta$ and $\ell_0$ are sufficiently small,  by Proposition \ref{p:k_indepen_k_symm} we have $u$ is $(k,\epsilon,y)$-symmetric in $[r/10,10r]$ with respect to $k$-plane $V_y$ for any $y\in \cV_{\delta,d,r}$. We also have $\cV_{\delta,d,r}\subset B_{\epsilon r}(y+V_y)$ by noting that $B_{2r}$ is not $(k+1,2\eta)$-symmetric. Fix $y_0\in \cV_{\delta,d,r}$, consider a Vitali subcovering of $\bigcup_{x\in y_0+V_{y_0}\cap B_r}B_{ r/10^2}(x)$. Denote by $\{B_{r/10^2}(x_i^1),i=1,\cdots,L_1\}$. Then 
$\cN^1:=B_{r}\setminus \cup_{i=1}^{L_1}B_{r/10^2}(x_i^1)$ is $(d,k,\epsilon,\eta)$-neck region. 
 Note that each ball must be one of the $b$, $c$, $d$ or $e$-type, we have a covering
\begin{equation}
     B_r \subset  \cN^1  \cup \bigcup_{b \in B^1} B_{r_b}(x_{b}) \cup \bigcup_{c \in C^1} B_{r_c}(x_{c}) \cup \bigcup_{d \in D^1} B_{r_d}(x_{d}) \cup \bigcup_{e \in E^1} B_{r_e}(x_{e}).
\end{equation}

Now for each $c$-ball in $C^1$ group, we have $r_c= 10^{-2}r$. We apply the same decomposition to obtain the refined covering 
\begin{equation}
     B_r \subset  \cN^2  \cup \bigcup_{b \in B^2} B_{r_b}(x_{b}) \cup \bigcup_{c \in C^2} B_{r_c}(x_{c}) \cup \bigcup_{d \in D^2} B_{r_d}(x_{d}) \cup \bigcup_{e \in E^2} B_{r_e}(x_{e}).
\end{equation}
with $r_c = 10^{-4}r$.

By iterating this decomposition, finally we will obtain a covering
\begin{equation}
     B_r \subset  \big(\cC_0 \cup \cN)  \cup \bigcup_{b \in B} B_{r_b}(x_{b}) \cup \bigcup_{d \in D} B_{r_d}(x_{d}) \cup \bigcup_{e \in E} B_{r_e}(x_{e}),
\end{equation}
where $\cC_0$ is the limit of $\{x_c,c\in C^i\}$ when $i\to \infty$. We need to prove that $\cN$ is indeed a neck region. The property 1 comes from our construction trivially. Property 2 and 3 follow from the Quantitative Cone-splitting Theorem \ref{t:Almost_Splitting}.

By Neck structure Theorem \ref{t:Neck_Struture}, we have
\begin{equation}
    \sum_{b \in B} r_b^k + \sum_{d \in D} r_d^k + \sum_{b \in E} r_e^k + \cH^k(\cC_0) \leq C(n)r^k.
\end{equation}

The proof is now finished.
\end{proof}

\begin{remark}
    The point here is that the approximated plane $V_{x}$ does not depend on the scale $r$. This essentially comes from the uniqueness of tangent maps \ref{t:pinch_fre_symmetric} and significantly simplifies the construction of neck regions out of $c$-balls. Similar uniqueness has been used in the elliptic case  \cite{HJ}. See \cite{CJN} for general constructions.  
\end{remark}

The covering for $d$-balls is standard. For completeness we include the proof here. See also \cite{CJN} or \cite{HJ}.

\begin{proposition}[Covering of $d$-balls]\label{p:d_ball_cover}
 Let $u$ be a solution to (\ref{e:Parab_equa}) (\ref{e:assumption_without_c}) on $Q_2$ with doubling assumption (\ref{e:Critical_DI_bound}).   Let $\delta \leq \delta_0= C(n, \lambda, \alpha, \Lambda)$ and $r \leq \ell_0(n,\lambda, \alpha, \Lambda, \delta)$. Suppose $\sup_{x \in B_r} N_{x,0}(s) \leq 
    d+\delta$ for any $s \leq C(n,\lambda)\delta^{-1} r$ and for some integer $d$. Furthermore we assume that $\cV_{\delta,d,r}(0,0) \neq \emptyset$ and is not $(k,\tau)$-independent in $B_r$ with $\tau < \tau_0(n,\lambda)$. Then we have the following decomposition
 \begin{equation}
     B_r \subset \bigcup_{b \in B} B_{r_b}(x_{b}) \cup \bigcup_{c \in C} B_{r_c}(x_{c}) \cup \bigcup_{e \in E} B_{r_e}(x_{e}) \cup S_d,
 \end{equation}
where each $B_{r_b}$ is a $b$-ball, $B_{r_c}(x_{c})$ is a $c$-ball and $B_{r_e}(x_{e})$ is an $e$-ball. Furthermore, we have the following measure estimates
\begin{equation}
\begin{split}
    & \sum_{b \in B} r_b^k + \sum_{e \in E} r_e^k \leq C(n,\tau) r^k. \\
    & \sum_{c \in C} r_c^k \leq C(n)\tau r^k. \\
    & \cH^k(S_d) = 0.
\end{split}
\end{equation}
    
\end{proposition}

\begin{proof}
    First we choose a Vitali covering of $B_r$ with uniform radius $10^{-1} \tau r$
\begin{equation*}
     B_r \subset \bigcup_{b \in B^1} B_{10^{-1}\tau r}(x^1_{b}) \cup \bigcup_{c \in C^1} B_{10^{-1}\tau r}(x^1_{c}) \cup \bigcup_{d \in D^1} B_{10^{-1}\tau r}(x^1_{d}) \cup \bigcup_{e \in E^1} B_{10^{-1}\tau r}(x^1_{e}).
\end{equation*}
and if $x^1$ and $y^1$ are any distinct centers, we require
\begin{equation*}
     B_{10^{-2}\tau r}(x^1) \cap B_{10^{-2}\tau r}(y^1) = \emptyset.
\end{equation*}

The disjointness implies that
\begin{equation}
    \sum_{b \in B^1} r_b^k + \sum_{e \in E^1} r_e^k \leq C_1(n,\tau) r^k. 
\end{equation}

Since $\cV_{\delta,d,r}(0) \neq \emptyset$ and is not $(k,\tau)$-independent in $B_r$, there exists some $V^{k-1}$ such that $\cV_{\delta,d,r}(0) \subset B_{\tau r}(V)$. Since $\cV_{\delta,d,r}(0)  \cap C^1 \neq \emptyset$ and $\cV_{\delta,d,r}(0) \cap D^1 \neq \emptyset$, by the disjointness we have 
\begin{equation}
    \sum_{c \in C^1} r_c^k + \sum_{d \in D^1} r_d^k \leq C_2(n) \tau^{-k+1} (\tau r)^k = C_2(n) \tau r^k.
\end{equation}

Next for each $d$-ball we can repeat the same Vitali covering argument, to refine the covering as
\begin{equation}
\begin{split}
    & B_r \subset \bigcup_{b \in B^2} B_{r_b}(x^2_{b}) \cup \bigcup_{c \in C^2} B_{r_c}(x^2_{c}) \cup \bigcup_{d \in D^2} B_{10^{-2}\tau^2 r}(x^2_{d}) \cup \bigcup_{e \in E^2} B_{r_e}(x^2_{e}). \\
    & \sum_{b \in B^2} r_b^k + \sum_{e \in E^2} r_e^k \leq C_1(n,\tau) (1 + C_2(n) \tau) r^k. \\
    & \sum_{c \in C^2} r_c^k  \leq C_2(n) \tau (1 + C_2(n) \tau) r^k. \\
    & \sum_{d \in D^2} r_d^k = \sum_{d \in D^2} (\frac{\tau}{10})^{2k} r^k \leq (C_2(n) \tau)^2 r^k.
\end{split}
\end{equation}

After $M$ times refinement, we obtain 
\begin{equation}
\begin{split}
    & B_r \subset \bigcup_{b \in B^M} B_{r_b}(x^M_{b}) \cup \bigcup_{c \in C^M} B_{r_c}(x^M_{c}) \cup \bigcup_{d \in D^M} B_{(10^{-1}\tau)^M r}(x^M_{d}) \cup \bigcup_{e \in E^M} B_{r_e}(x^M_{e}). \\
    & \sum_{b \in B^M} r_b^k + \sum_{e \in E^M} r_e^k \leq C_1(n,\tau) \sum_{i = 0}^{M-1} (C_2(n) \tau)^i r^k. \\
    & \sum_{c \in C^M} r_c^k  \leq C_2(n) \tau \sum_{i = 0}^{M-1} (C_2(n) \tau)^i r^k. \\
    & \sum_{d \in D^M} (\frac{\tau}{10})^{kN} \leq (C_2(n) \tau)^M.
\end{split}
\end{equation}
Let $C_2(n) \tau < 1/10$, then $\{x_d^M\} \to S_d$ in the Hausdorff sense with $\cH^k(S_d) = 0$ and the proof is finished.
\end{proof}

By applying the covering of $c$-balls and $d$-balls, we can prove the  Neck decomposition theorem \ref{t:Neck_decompose}.

\begin{proof}[Proof of Theorem \ref{t:Neck_decompose}]
    By Lemma \ref{l:Equiv_Normalizaitons}, we have $\sup_{x \in B_1} N_{x,0}(1) \le C_0(n,\lambda,\alpha,\Lambda)$ for some integer $C_0$. We can cover $B_1$ using Vitali covering with radius $r= c(n,\lambda)\delta$. Then we can apply Proposition \ref{p:c_ball_cover} and \ref{p:d_ball_cover} to decompose $c$-balls and $d$-balls iteratively. Then we obtain the following covering 
\begin{equation}\label{e:decomposition}
     B_1 \subset  \bigcup_{a} \big(\cC_{0,a} \cup \cN_a \cap B_a(x_a) \big)  \cup \bigcup_{b } B_{r_b}(x_{b}) \cup \bigcup_{e} B_{r_e}(x_{e}) \cup S_d,
\end{equation}
where $\cH^k(S_d) = 0$ and $\sum_a r^k_a + \sum_b r^k_b + \sum_e r^k_e + \cH^k(\bigcup_{a} \cC_{0,a}) \leq C(n, \lambda, \alpha, \Lambda, \delta)$.

Next we deal with $e$-balls. According to the definition of $e$-ball, we know $\cV_{\delta,d,r_e}(x_e,0) = \emptyset$. In other words, $\sup_{x\in B_e} N_{x,0}(s) \le C_0(n,\lambda,\alpha,\Lambda) - \delta$ for some $s\ge 10^{-2}r_e$. By Lemma \ref{l:parabolic_frequency_near_Z} we get $\sup_{x \in B_e} N_{x,0}(s) \leq C_0(n,\lambda,\alpha,\Lambda) - 1 +\delta$ for $s \ge c(n,\lambda,\alpha,\Lambda,\delta)r_e$. And we can cover $B_{r_e}$ with Vitali covering with this radius $c(n,\lambda,\alpha,\Lambda,\delta)r_e$. Then apply the arguments above again to refine the covering lemmas of $c$-balls and $d$-balls and continue the process for finitely many (at most $C_0(n,\lambda,\alpha,\Lambda))$ times. Finally we obtain the result. 

The second conclusion comes from the Quantitative Cone-splitting theorem \ref{t:Almost_Splitting}. Actually, if $x \in \cN^a$ or $x\in B_{r_b}(x_b)$, by Corollary \ref{c:finitemanypinch} and Theorem \ref{t:Almost_Splitting} we have $u$ is $(k+1, \epsilon, r', x)$-symmetric for some small $r'$ provided $\eta$ and $\delta$ is sufficiently small. This proves that $\cS^k_{\epsilon}  \cap B_1 \subset  \Big( S_0 \cup \bigcup_{a} \cC_{0, a} \Big) $.
\end{proof}

\section{Proof of the Main Theorems}\label{s:7_proof}
In this section we will prove our main theorems.
\subsection{Proof of Theorem \ref{t:mainHolder} }
%In this subsection we will prove Theorem \ref{t:mainHolder}.
\begin{comment}

First, as a direct consequence of the neck decomposition theorem we can get the Hausdorff measure estimate of the quantitative strata. 

\begin{theorem}\label{t:esti_Quantitative_strata}
    Let $u$ be a solution to (\ref{e:Parab_equa}) (\ref{e:assumption_without_c}) on $Q_2$ with doubling assumption (\ref{e:Critical_DI_bound}). For any $\epsilon>0$ and $k \leq n-1$ we have 
    \begin{align}
        \cH^k(\cS_\epsilon^k\cap B_1)\le C(n,\lambda,\alpha,\Lambda,\epsilon). 
    \end{align}
    where $\cS_\epsilon^k=\cS^k_\epsilon(u(\cdot, 0)).$
\end{theorem}
\begin{remark}
    By analyzing carefully the proof of neck-decomposition theorem as \cite{CJN} one can get a Minkowski estimate for $\cS_\epsilon^k$. See more details in \cite{CJN}.
\end{remark}
\begin{proof}
   For any $r>0$ to be fixed, choose a covering $\{ B_r(x_i), i=1,\cdots, K_r\}$ of $B_1$. 
For any $\epsilon>0$ if $r= \ell_0(n,\lambda,\Lambda,\alpha,\epsilon)$ is sufficiently small we can apply Theorem \ref{t:Neck_decompose} to each $B_r(x_i)$ so that we get 
\begin{align}
    \cH^k(\cS_\epsilon^k\cap B_r(x_i))\le C(n,\lambda,\Lambda,\epsilon)r^k.
\end{align}
Hence 
\begin{align}
     \cH^k(\cS_\epsilon^k\cap B_1)\le \sum_{i=1}^{K_r}  \cH^k(\cS_\epsilon^k\cap B_r(x_i))\le K_r C(n,\lambda,\Lambda,\epsilon)r^k\le C(n,\lambda,\Lambda,\alpha,\epsilon).
\end{align}
This finishes the proof.
\end{proof}
\end{comment}

%Now we are ready to prove the main Theorem \ref{t:mainHolder}.
%\begin{proof}[Proof of Theorem \ref{t:mainHolder}]
   For any $ \epsilon\le 1/10$ denote $\cS^{n-1}_\epsilon=\cS^{n-1}_\epsilon(u(\cdot,0))$. We will see that $Z_0\cap B_1 \subset \cS^{n-1}_{\epsilon}\cap B_1$ where $Z_0=\{x: u(x,0)=0\}$. To see this, it suffices to show that $u(y,0)\ne 0$ for any $y\notin \cS_\epsilon^{n-1}$.Let $y\notin \cS_\epsilon^{n-1}$ we have for some $s>0$ that $u(\cdot, 0)$ is $(n,\epsilon,s,y)$-symmetric. This means that $f(x)=\frac{u(y+sx,0)}{\left(\fint_{B_1}|u(y+sx,0)|^2dx\right)^{1/2}}$ is $\epsilon$-close to $1$. Hence $|f(0)-1|\le \epsilon$ and thus $u(y,0)\ne 0$ if $\epsilon<1/2$. Hence we have $Z_0\cap B_1 \subset \cS^{n-1}_{\epsilon}\cap B_1$. The Hausdorff estimate follows now directly since $\cH^{n-1}(\cS_\epsilon^{n-1}\cap B_1)\le C<\infty$ by Theorem \ref{t:Neck_decompose}. Hence we complete the proof. \qed

\subsection{Proof of Theorem \ref{t:mainLip} and Theorem \ref{t:non-increasingNodal}}

The following doubling estimate is taken from \cite{EFV}.
\begin{lemma}[Theorem 3 in \cite{EFV}]\label{l:doublingEFV}
       Let $u$ be a nonzero solution of \eqref{e:ParabolicE} in $Q_2$ satisfying \eqref{e:assumption_aijbc} and \eqref{e:Lipschitzcondition}. Then, there exists $N(n,\lambda)$ such that the following holds when $0<r\le (N\log(N\Theta))^{-1/2}$ that 
       \begin{align}
           \int_{Q_{2r}}u^2dxdt\le e^{N\log(N\Theta)\log(N\log(N\Theta))}\int_{Q_r}u^2dxdt,
       \end{align}
       where $\Theta=\frac{\int_{Q_2}u^2dxdt}{\int_{B_{1/2}}u^2(x,0)dx}$ and $Q_r=Q_r(0,0)=\{(x,t): |x|< r, -r^2< t\le 0\}$.
\end{lemma}
As a direct consequence we can get the following doubling estimate 
\begin{corollary}\label{c:doubling}
       Let $u$ be a solution of \eqref{e:ParabolicE} in $Q_2$ satisfying \eqref{e:assumption_aijbc} and \eqref{e:Lipschitzcondition}. If $u$ is nonzero at $t=0$, then there exists $\Lambda(n,\lambda,\Theta)$ and $r_0(n,\lambda,\Theta)$ such that for any $x\in B_1$ and $0<r\le r_0$ that 
       \begin{align}
           \int_{Q_{2r}(x,0)}u^2dxdt\le \Lambda\int_{Q_r(x,0)}u^2dxdt.
       \end{align}
       where  $\Theta=\frac{\int_{Q_2}u^2dxdt}{\int_{B_{3/2}}u^2(x,0)dx}$.
\end{corollary}
\begin{proof}
    If $u(\cdot,0)$ is not vanishing in any open subset of $B_1$, we can directly apply Lemma \ref{l:doublingEFV} by replacing center $0$ with given $x\in B_1$. If $u(\cdot,0)$ is vanishing in an open subset $U\subset B_1$, we will show that $u(\cdot,0)\equiv 0$ which contradicts to the assumption that $u$ is nonzero at $t=0$. To see this, since $u(\cdot,0)$ is nonzero, there exists some  $x_0\in U$ and $\hat{r}>0$ such that $u(\cdot,0)$ is nonzero in $B_{\hat{r}}(x_0)$ and $B_{2\hat{r}}(x_0)\subset B_2$. Applying Lemma \ref{l:doublingEFV} we get doubling estimate of $u$ at center $(x_0,0)$. We need to consider the following two cases. First, if $\int_{Q_r(x_0,0)}u^2dxdt=0$ for some small $r\le \hat{r}$, then $u$ vanishes to infinite order at $(x_0,0)$ in space-time sense as in  \cite{EsFr}. Therefore, by \cite{EsFr} we have that $u(\cdot,0)\equiv 0$.  On the other hand, if $\int_{Q_r(x_0,0)}u^2dxdt\ne 0$ for all $r\le \hat{r}$, consider the tangent map of $u$ at $(x_0,0)$. One can argue as Lemma \ref{l:Caloric_Approx} to get that the tangent map is a nonzero caloric polynomial. However $u(\cdot,0)$ is zero near $x_0$ which implies that the tangent map is zero at $t=0$. By the backward uniqueness of heat equation(see \cite{Poon,WuZh}), this is a contradiction since tangent map is a nonzero caloric polynomial. Thus we finish the proof.
\end{proof}

\begin{proof}[Proof of Theorem \ref{t:mainLip}]
    Theorem \ref{t:mainLip} follows directly by Corollary \ref{c:doubling} and Theorem \ref{t:mainHolder}.
\end{proof}

Let us now turn to the proof of  Theorem \ref{t:non-increasingNodal}. The following result is standard.
\begin{lemma}\label{l:dimlowerboundNodal}
 Let $u$ be a continuous function in $B_1\subset \RR^n$. Assume both $\{x:u(x)>0\}$ and $\{x: u(x)<0\}$ are not empty.  Then 
 \begin{align}
     \dim Z_u\ge n-1,
 \end{align}
 where $Z_u=\{x: u(x)=0\}.$
\end{lemma}
\begin{proof}
   Assume $x_0,y_0\in B_1$ satisfy $u(x_0)>0$ and $u(y_0)<0$. Since $u\in C^0$, then there exists $r_0>0$ such that $\inf_{B_{r_0}(x_0)} u(x)>0$ and $\sup_{B_{r_0}(y_0)} u(x)<0$. Consider two hyperplanes $H_1$ and $H_2$ satisfying
   \begin{itemize}
       \item $H_1$ and $H_2$ are perpendicular to $v:=x_0-y_0$
       \item $x_0\in H_1$ and $y_0\in H_2$.
   \end{itemize}
  Denote line $\ell_{x,y}$ to be the line segment which is parallel to $v$ and $\ell_{x,y}\cap H_1=\{x\}$ and $\ell_{x,y}\cap H_2=\{y\}$.  For any $x\in H_1\cap B_{r_0}(x_0)$ there exists a unique $y_x\in H_2\cap B_{r_0}(y_0)$ such that $\ell_{x,y_x}// v$.  
 Noting that $u|_{H_1\cap B_{r_0}(x_0)}>0$ and  $u|_{H_1\cap B_{r_0}(y_0)}<0$ and $u$ is continuous, there must exist $z_x\in \ell_{x,y_x}$ such that $u(z_x)=0$. Define a map $\tau: H_1\cap B_{r_0}(x_0)\to Z_u$ by $\tau(x)=z_x$. Then $|\tau(x)-\tau(y)|\ge |x-y|$.  Hence by the definition of Hausdorff measure, we have 
 $$\cH^{n-1}(\{z: u(z)=0, z\in \ell_{x,y_x}, x\in H_1\cap B_{r_0}(x_0)\})\ge \cH^{n-1}(H_1\cap B_{r_0}(x_0))\ge c(n)r_0^{n-1}>0.$$
 In particular, $\dim Z_u\ge n-1.$
\end{proof}
Now we are ready to prove Theorem \ref{t:non-increasingNodal}.
\begin{proof}[Proof of Theorem \ref{t:non-increasingNodal}]
       By  the standard backward uniqueness (see  \cite{CM22,EsFr,Linunique,WuZh}), we know that $u$ is not vanishing at any $t>-4$ otherwise $u\equiv 0$ in $Q_2$. Therefore by  Theorem \ref{t:mainLip} we know that $\dim Z_t\le n-1$ for any $t>-4$ (see also \cite{Chen98b}). Therefore, we only need to consider the following two cases. 

    \textbf{Case 1:} Assume $s=-4$ and $t>-4$. If $\dim Z_{-4}\ge n-1$, then $\dim Z_t\le n-1\le \dim Z_{-4}$.  If $\dim Z_{-4}<n-1$, then by Lemma \ref{l:dimlowerboundNodal} we can assume $u(x,-4)\ge 0$. By maximum principle Theorem \ref{t:maximum} we get $u(x,t)>0$ for any $t>-4$. In particular $Z_t=\emptyset$. Hence $\dim Z_t\le \dim Z_s$. 

    \textbf{Case 2:} Assume $-4<s<t\le 0$. We have $\dim Z_t\le n-1$ and $\dim Z_s\le n-1$. If $\dim Z_s=n-1$ then we are done. If $\dim Z_s<n-1$, the same argument as above we get $Z_t=\emptyset$ and we are also done. Thus we complete the proof.
\end{proof}


\begin{thebibliography}{HHHH}

\bibitem{AV} Alessandrini, G., \& Vessela, S. (1988). Local behaviour of solutions to parabolic equations. Communications in Partial Differential Equations, 13(9), 1041-1058.
 
\bibitem{AAG} Altschuler, S., Angenent, S. B., \& Giga, Y. (1995). Mean curvature flow through singularities for surfaces of rotation. Journal of Geometric Analysis, 5, 293–358.


\bibitem{An} Angenent, S. (1988). The zero set of a solution of a parabolic equation. Journal für die reine und angewandte Mathematik, 390, 79-96.

\bibitem{An2} Angenent, S. B. (1991). On the formation of singularities in the curve shortening problem. Journal of Differential Geometry, 33, 601-633.

\bibitem{An05} Angenent, S. (2005). Curve shortening and the topology of closed geodesics on surfaces. Annals of Mathematics, 162(3), 1185-1239.

\bibitem{BNS}Bruè, E., Naber, A. \& Semola, D., (2022). Boundary regularity and stability for spaces with Ricci bounded below. Inventiones mathematicae, 228(2), pp.777-891

\bibitem {CJN} Cheeger, J., Jiang, W., \& Naber, A. (2021). Rectifiability of singular sets of noncollapsed limit spaces with Ricci curvature bounded below. Annals of Mathematics, 193(2), 407-538.

 \bibitem {CN} Cheeger, J., \& Naber, A. (2013). Lower bounds on Ricci curvature and quantitative behavior of singular sets. Inventiones mathematicae, 191, 321-339.
 

 \bibitem {CNV} Cheeger, J., Naber, A., \& Valtorta, D. (2015). Critical sets of elliptic equations. Communications on Pure and Applied Mathematics, 68(2), 173-209.

\bibitem{Chen98a} Chen, X.-Y. (1998). A strong unique continuation theorem for parabolic equations. Mathematische Annalen, 311(4), 603-630.

\bibitem{Chen98b} Chen, X.-Y. (1998). On the scaling limits at zeros of solutions of parabolic equations. Journal of Differential Equations, 147(2), 355-382.

 \bibitem{CZ} Chou, K.-S., \& Zhu, X.-P. (1998). Shortening complete plane curves. Journal of Differential Geometry, 50(3), 471-504.

\bibitem{CM} Colding, T. H., \& Minicozzi, W. P. (2011). Lower bounds for nodal sets of eigenfunctions. Communications in Mathematical Physics, 306(3), 777-784.

\bibitem{CMcaloric} Colding, T. H., \& Minicozzi, W. P. (2021). Optimal bounds for ancient caloric functions. Duke Mathematical Journal, 170(18), 4171-4182.

\bibitem{CM22} Colding, T. H., \& Minicozzi, W. P. (2022). Parabolic frequency on manifolds. International Mathematics Research Notices, 2022(15), 11878-11890.



\bibitem{CHH} Choi, K., Haslhofer, R., \& Hershkovits, O. (2022). Ancient low-entropy flows, mean-convex neighborhoods, and uniqueness. Acta Mathematica, 228(2), 217-301.

\bibitem{DHS} Daskalopoulos, P., Hamilton, R., \& Sesum, N. (2010). Classification of compact ancient solutions to the curve shortening flow. Journal of Differential Geometry, 84(3), 455-464.

\bibitem{DP} Del Santo, D., \& Prizzi, M. (2005). Backward uniqueness for parabolic operators whose coefficients are non-Lipschitz continuous in time. Journal de Mathématiques Pures et Appliquées, 84, 471-491.

 \bibitem {Dong} Dong, R.-T. (1992). Nodal sets of eigenfunctions on Riemann surfaces. Journal of Differential Geometry, 36, 493-506.

 \bibitem {DF1} Donnelly, H., \& Fefferman, C. (1988). Nodal sets of eigenfunctions on Riemannian manifolds. Inventiones mathematicae, 93(1), 161-183.

 \bibitem {DF3} Donnelly, H., \& Fefferman, C. (1990). Nodal sets for eigenfunctions of the Laplacian on surfaces. Journal of the American Mathematical Society, 3(2), 333-353.

\bibitem{EsFr} Escauriaza, L., \& Fernández, F. J. (2003). Unique continuation for parabolic operators. Arkiv för matematik, 41(1), 35-60.

\bibitem{EFV} Escauriaza, L., Fernández, F. J., \& Vessella, S. (2006). Doubling properties of caloric functions. Applicable Analysis, 85(1-3), 205-223.

\bibitem{ESV} Escauriaza, L., Seregin, G. A., \& Šverák, V. (2003). Backward uniqueness for parabolic equations. Archive for Rational Mechanics and Analysis, 169(2), 147-157.

\bibitem{EV}Escauriaza, L., \& Vega, L.(2001) Carleman inequalities and the heat operator. II. Indiana Univ. Math. J. 50, no. 3, 1149–1169.

\bibitem{Gr}Grayson, M.(1989). Shortening Embedded Curves.  Annals of Mathematics, 129(1) 71-111.

\bibitem{Han98} Han, Q. (1998). On the Schauder estimates of solutions to parabolic equations. Annali della Scuola Normale Superiore di Pisa - Classe di Scienze, 27(1), 1-26.

\bibitem{HHL} Han, Q., Hardt, R., \& Lin, F. (1998). Geometric measure of singular sets of elliptic equations. Communications on Pure and Applied Mathematics, 51(11-12), 1425-1443.

\bibitem{HL94a} Han, Q., \& Lin, F.-H. (1994). On the geometric measure of nodal sets of solutions. Journal of Partial Differential Equations, 7(2), 111-131.



\bibitem{HLparabolic} Han, Q., \& Lin, F.-H. (1994). Nodal sets of solutions of parabolic equations. II. Communications on Pure and Applied Mathematics, 47(9), 1219-1238.

\bibitem{HHN} Hardt, R., Hoffmann-Ostenhof, M., Hoffmann-Ostenhof, T., \& Nadirashvili, N. (1999). Critical sets of solutions to elliptic equations. Journal of Differential Geometry, 51(2), 359-373.

\bibitem{HS} Hardt, R., \& Simon, L. (1989). Nodal sets for solutions of elliptic equations. Journal of Differential Geometry, 30(2), 505-522.

\bibitem{HJ} Huang, Y., \& Jiang, W. (2023). Volume estimates for singular sets and critical sets of elliptic equations with Hölder coefficients. arXiv preprint arXiv:2309.08089.


\bibitem{JN} Jiang, W., \& Naber, A. (2021). $L^2$ curvature bounds on manifolds with bounded Ricci curvature. Annals of Mathematics, 193(1), 107-222.

\bibitem{KZZ} Kenig, C. E., Zhu, J., \& Zhuge, J. (2022). Doubling inequalities and nodal sets in periodic elliptic homogenization. Communications in Partial Differential Equations, 47(3), 549-584.

\bibitem{Lin91} Lin, F.-H. (1991). Nodal sets of solutions of elliptic and parabolic equations. Communications on Pure and Applied Mathematics, 44(3), 287-308.

\bibitem{Linunique} Lin, F.-H. (1990). A uniqueness theorem for parabolic equations. Communications on Pure and Applied Mathematics, 43(1), 127-136.

\bibitem{LinShen19} Lin, F., \& Shen, Z. (2019). Nodal sets and doubling conditions in elliptic homogenization. Acta Mathematica Sinica, English Series, 35(6), 815-831.

\bibitem{LM} Lions, J. L., \& Malgrange, B. (1960). Sur l’unicité rétrograde dans les problèmes mixtes paraboliques. Mathematica Scandinavica, 8, 277-286.

\bibitem{LTY} Liu, F., Tian, L., \& Yang, X. (2024). Measure upper bounds of nodal sets of Robin eigenfunctions. Mathematische Zeitschrift, 306(1), Paper No. 14, 14 pp.

\bibitem{Loglower} Logunov, A. (2018). Nodal sets of Laplace eigenfunctions: proof of Nadirashvili's conjecture and of the lower bound in Yau's conjecture. Annals of Mathematics, 187(1), 241-262.

\bibitem{Logupper} Logunov, A. (2018). Nodal sets of Laplace eigenfunctions: polynomial upper estimates of the Hausdorff measure. Annals of Mathematics, 187(1), 221-239.

\bibitem{LMNN} Logunov, A., Malinnikova, E., Nadirashvili, N., \& Nazarov, F. (2021). The sharp upper bound for the area of the nodal sets of Dirichlet Laplace eigenfunctions. Geometric and Functional Analysis, 31(5), 1219-1244.

\bibitem {MS} Malecki, J., \& Serafin, G. (2020). Dirichlet heat kernel for the Laplacian in a ball. Potential Analysis, 52, 545-563.

\bibitem{Man} Mandache, N. (1996). On a counterexample concerning unique continuation for elliptic equations in divergence form. Matematicheskaya Fizika, Analiz, Geometriya, 3(3-4), 308-331.

\bibitem{Miller} Miller, K. (1974). Non-unique continuation for uniformly parabolic and elliptic equations in self-adjoint divergence form with Hölder-continuous coefficients. Archive for Rational Mechanics and Analysis, 54, 105-117.

\bibitem{Mo} Moser, J. (1964). A Harnack inequality for parabolic differential equations. Communications on Pure and Applied Mathematics, 17, 101-134.


 \bibitem {NV} Naber, A., \& Valtorta, D. (2017). Volume estimates on the critical sets of solutions to elliptic PDEs. Communications on Pure and Applied Mathematics, 70(10), 1835-1897.

\bibitem{NV19} Naber, A.,\& Valtorta, D. (2019). Energy identity for stationary Yang Mills. Invent. math. 216, 847–925.

\bibitem{NV24}Naber, A.,\& Valtorta, D. (2024). Energy Identity for Stationary Harmonic Maps. arXiv:2401.02242v1 [math.AP]. 

\bibitem{Pa} Paronetto, F. (2023). Harnack inequality for parabolic equations with coefficients depending on time. Advances in Calculus of Variations, 16(4), 791-821.

\bibitem {P} Plis, A. (1963). On non-uniqueness in Cauchy problem for an elliptic second order differential equation. Bulletin of the Polish Academy of Sciences: Series of Sciences, Mathematics, Astronomy and Physics, 11, 95-100.

\bibitem{Poon} Poon, C. C. (1996). Unique continuation for parabolic equations. Communications in Partial Differential Equations, 21(3-4), 521-539.

\bibitem{SZ} Sogge, C. D., \& Zelditch, S. (2011). Lower bounds on the Hausdorff measure of nodal sets. Mathematical Research Letters, 18(1), 25-37.

\bibitem{WuZh} Wu, J., \& Zhang, L. (2019). Backward uniqueness for general parabolic operators in the whole space. Calculus of Variations and Partial Differential Equations, 58(4), Paper No. 155, 19 pp.


\end{thebibliography}
\end{document}